\documentclass{elsart}

\usepackage{amssymb}
\usepackage{color}
\usepackage{graphicx}\usepackage{subfigure}
\usepackage{amsmath}
\usepackage{amsfonts}
\usepackage{epstopdf}
\usepackage[title]{appendix}
\allowdisplaybreaks[4]
\usepackage{psfrag}

\usepackage{paralist}
\usepackage[colorlinks=true]{hyperref}

\usepackage{multirow}
\renewcommand{\vec}[1]{\mbox{\boldmath \small $#1$}}

\newtheorem{lemma}{Lemma}[section]
\newtheorem{corollary}{Corollary}[section]
\newtheorem{remark}{Remark}[section]

\numberwithin{equation}{section}
\numberwithin{figure}{section}
\numberwithin{table}{section}
\numberwithin{thm}{section}

\newenvironment{proof}[1][Proof]{\begin{trivlist}
\item[\hskip \labelsep {\bfseries #1}]}{\end{trivlist}}
\renewcommand{\qed}{\hfill \nobreak \ifvmode \relax \else
      \ifdim\lastskip<1.5em \hskip-\lastskip
      \hskip1.5em plus0em minus0.5em \fi \nobreak
      \vrule height0.75em width0.5em depth0.25em\fi}
\begin{document}

\begin{frontmatter}
\title{Globally hyperbolic moment model of arbitrary order  for
	three-dimensional special relativistic Boltzmann equation}

\author{Yangyu Kuang}
\ead{kyy@pku.edu.cn}
\address{School of Mathematical Sciences, Peking University,
Beijing 100871, P.R. China}
\author[label2]{Huazhong Tang}
\thanks[label2]{Corresponding author. Tel:~+86-10-62757018;
Fax:~+86-10-62751801.}
\ead{hztang@math.pku.edu.cn}
\address{HEDPS, CAPT \& LMAM, School of Mathematical Sciences, Peking University,
Beijing 100871, P.R. China; School of Mathematics and Computational Science, Xiangtan University,
Xiantan 411105, Hunan Province, P.R. China}
 \date{\today{}}
\maketitle

\begin{abstract}
This paper extends the model reduction method by the operator projection
to the three-dimensional special relativistic Boltzmann equation.
 The derivation of arbitrary order  moment system is built on
our careful study of infinite families of the complicate
Grad type orthogonal polynomials depending on a parameter and the real spherical harmonics.
We derive the recurrence relations of the polynomials, calculate
their derivatives with respect to the independent variable and parameter respectively,
and study their zeros. 
The recurrence relations and partial derivatives of the real spherical harmonics are also
given.
It is proved that our moment system is globally hyperbolic,
 and linearly stable. Moreover, the Lorentz covariance  is also studied in the 1D space.
\end{abstract}

\end{frontmatter}


\section{Introduction}
\label{sec:intro}

The relativistic kinetic theory of gases has been widely  applied in the astrophysics and cosmology \cite{RB:2002,RK:1980}. 
Different from a non-relativistic monatomic gas,  a relativistic gas has
a bulk viscosity. It has called much attention  to a number of applications of this theory: the
effect of neutrino viscosity on the evolution of the universe
 and the study of galaxy
formation, neutron stars, and controlled thermonuclear fusion etc.,
%
%
and the interest has grown in recent years as experimentalists are now able to make reliable measurements on physical systems where relativistic effects are no longer negligible.
%
%
However, the relativistic kinetic theory
 has been sparsely used to model phenomenological matter in comparison to fluid models.
In the non-relativistic case, the kinetic theory has been intensively  studied as a mathematical subject during several decades, and   also played an important role from an engineering point of view,
see e.g. \cite{RB:1988,Cha:1970}.
One could determine the  distribution function
and  transport coefficients of  gases from the Boltzmann equation, however it was not so easy.
An  approximate solution of the integro-differential
equation was obtained by  Hilbert  from a power series expansion of a parameter
(being proportional to the mean free path).
 The transport coefficients were independently  calculated by Chapman and Enskog   for gases whose molecules interacted according to any kind of spherically symmetric potential function.
Another method   is to
expand the distribution
function in terms of tensorial Hermite polynomials and introduce
the balance equations corresponding to higher order moments of the distribution function \cite{GRAD:1949,GRAD2:1949}.
%
%
The   Chapman-Enskog  method
 assumes that in the hydrodynamic regime
the distribution function can be expressed as a function of the hydrodynamic variables and their gradients,
%
and has been
extended to the relativistic cases, see e.g. \cite{CHA2:2008,CHA:2012,CHA:2008,NHH:1983,NHH:1985},
but unfortunately, there exists difficulty  to derive the  relativistic hydrodynamic equations from the kinetic theory \cite{DM:2012}.
The moment method  is  also generalized to the relativistic cases,
see e.g. \cite{Ander:1970,MLL:2006,IS:1976,IS:1979,IS2:1979,Kranys:1972,Stewart:1977}.
Combining the Chapman-Enskog method with
the moment method has been attempted   in order to state the influence of the Knudsen number \cite{DM:2012,RDH:2013}.
%
 Recently, some latest progresses were gotten on the Grad moment method  in the non-relativistic case.
  A regularization  was presented  for the 1D Grad moment system to achieve global hyperbolicity \cite{1DB:2013}.
It was based on the observation that the characteristic polynomial of the flux Jacobian
of the Grad's moment system did not depend on the intermediate moments,
 and  further extended to the multi-dimensional case \cite{HY:2014,HG:2014}.
By  replacing the exact integration with  some quadrature rule,
the quadrature based projection methods could derive hyperbolic PDE systems for  the Boltzmann equation
  \cite{QI:2014,HG2:2014}.
In the 1D case,  it is similar to the  regularization in  \cite{1DB:2013}. 
Those contributions  led to  well understanding the hyperbolicity of  the Grad moment systems.
Based on the operator projection,  a general framework of model reduction technique was recently presented  \cite{MR:2014}, in which
the time and space derivatives in the kinetic equation  were synchronously projected  into a finite-dimensional weighted polynomial space,
and most of the existing moment systems mentioned above might be regotten.
 %
%

There exists a huge difficulty in deriving the relativistic hyperbolic moment system  of higher order
since the family of Grad type orthogonal polynomials can not be found easily.
Several attempts  have been conducted to construct them analogous to the Hermite polynomials, see e.g \cite{GRADP:1974,RK:1980},
%
where the weighted polynomial space was spanned by the products of the weighted function,
Grad type orthogonal polynomials, and  irreducible tensors.
Their application can be found in \cite{DM:2012,RDH:2013,PRO:1998}.
%
%
Unfortunately,
there seems no explicit expression of the moment systems
 if the order of the moment system is larger than 3, and
 the hyperbolicity of  existing general moment systems is not proved,
 even for the  second order moment system (e.g. the general Israel and Stewart system).
 For a special  case with heat conduction and no viscosity,
 Hiscock and Lindblom proved that the Israel and Stewart moment system  was globally hyperbolic and linearly stable in  the
  Landau frame, but converse  in  the Eckart frame.
     In the general case, they only proved that the Israel and Stewart moment system
     was hyperbolic near the equilibrium. The readers are referred to \cite{NHH:1983,NHH:1988,NNH:1989}.
It is easy to show by the approach \cite{NHH:1987,NHH:1988} that  the Israel and Stewart moment system
 is not globally hyperbolic  in the Landau frame too if  the viscosity exists. There does not exist any result on the hyperbolicity or  loss of hyperbolicity  of  (existing) general higher-order moment systems for the relativistic kinetic equation. Such
proof is very difficult and challenging.
The loss of hyperbolicity may  cause the solution blow-up when the distribution is far away from the equilibrium state. Even for the non-relativistic case,  increasing the number of moments seems not to avoid  such blow-up \cite{FAR:2012}.
Recently,  globally hyperbolic moment system of arbitrary order was derived for the 1D special relativistic Boltzmann equation   \cite{Kuang:2017} with the aid of
 the model reduction method by the operator projection \cite{MR:2014}.
 The key is to  choose the weight function and define the weighted polynomial spaces and their basis as well as the projection operator.

  The aim of this paper is to derive globally hyperbolic moment system of arbitrary order for the three-dimensional special relativistic Boltzmann equation.
 Unlike the 1D case \cite{Kuang:2017}, the basis of the weighted polynomial space are very difficult to be obtained, because  the irreducible tensors \cite{RK:1980} are linearly dependent, and
it is difficult to find their maximally linearly independent set.
Our contribution are that the irreducible tensors are replaced with the real spherical harmonics
in order to find the basis of  weighted polynomial spaces and  the properties of infinite families of the complicate Grad type orthogonal polynomials depending on a parameter and  the real spherical harmonics are carefully studied. 
The paper is organized as follows.
Section \ref{sec:RB} introduces the special relativistic Boltzmann equation
and some macroscopic quantities defined via the kinetic theory.
 Section \ref{sec:orth} gives infinite families of orthogonal polynomials dependent on a parameter,
and studies their properties: recurrence relations, derivative relations with respect to the variable and the parameter, and zeros. Section \ref{sec:SH} introduces  the real  spherical harmonics and their properties. Section \ref{sec:moment} derives the moment system of the special relativistic Boltzmann equation and discusses  its limitation  and corresponding quasi-1D cases. 
Section \ref{sec:prop} studies the properties of our moment system:  the hyperbolicity,  linear stability, and Lorentz covariance.
Section \ref{sec:conclud} concludes the paper.

\section{Preliminaries and notations}  
\label{sec:RB}
In the special relativistic kinetic theory of gases \cite{RB:2002},
a microscopic gas particle of rest  mass $m$ is characterized by the $4$ space-time
coordinates $(x^{\alpha})=(x^0,\vec{x})$ and momentum $4$-vector $(p^{\alpha})=(p^{0},\vec{p})$, where
$x^0=ct$, $c$ denotes the speed of light in vacuum, and $t$ and $\vec{x}$ are the time and 3D spatial coordinates,
 respectively. 
Besides  the contravariant notation (e.g. $p^{\alpha}$),  the covariant notation such as $p_{\alpha}$
 will  also be used and is related to the contravariant by
 $
 p_{\alpha}=g_{\alpha\beta}p^{\beta}$ or $ p^{\alpha}=g^{\alpha\beta}p_{\beta}$,
%
%
where $(g^{\alpha\beta})$ is chosen as the  Minkowski  space-time metric tensor,
 $(g^{\alpha\beta})={\rm diag}\{1,-\vec{I}_{3}\}$,
$\vec{I}_{3}$ is the unit matrix of size 3, $(g_{\alpha\beta})$ is
the inverse of $(g^{\alpha\beta})$, and the Einstein summation convention over repeated indices has been used.
%
For a free relativistic particle,  one has the relativistic energy-momentum
relation (aka ``on-shell'' or ``mass-shell'' condition)
$E^2-\vec p^2 c^2=m^2 c^4$. If putting $p^0= c^{-1}E=\sqrt{\vec{p}^2+m^2c^2}$,
then the ``mass-shell'' condition may be rewritten as
$p^{\alpha}p_{\alpha}=m^2c^2$.

The special relativistic Boltzmann equation provides a statistical description of a gas of relativistic particles that are interacting through binary collisions in Minkowski space \cite{RB:2002}.
%
%
For a single gas it reads
\begin{equation}
\label{eq:Boltz}
p^{\alpha}\frac{\partial f}{\partial x^{\alpha}}=Q(f,f),
\end{equation}
where $f=f(\vec x, \vec p, t)$ denotes one-particle distribution function,
 $Q(f,f)$ is the collision term
  depending  on the product of the distribution functions of two particles at collision, e.g.
$$
Q(f,f)=\int_{\mathbb{R}^{{3}}}\int_{\mathbb{S}^{{2}}_{+}}\left(f_{*}'f'-f_{*}f\right)B d{\Omega} \frac{d^{{3}}\vec{p}_{*}}{p_{*}^{0}},
$$
here the distributions $f$ and $f_*$  depend on  the momenta before a collision, while
 $f'$ and $f_*'$  rely on the momenta after the collision,
 $d\Omega$ denotes the element of the solid angle,
the collision kernel is given by $B=\sigma \sqrt{(p_{*}^{\alpha}p_{\alpha})^{2}-m^2c^2}$ for a single non degenerate gas (e.g. electron gas),
and $\sigma$ denotes the differential cross section of collision.
The collision term  satisfies
\begin{equation}
\label{eq:colcon}
    \int_{\mathbb{R}^{{3}}}Q(f,f)\frac{d^{3}\vec{p}}{p^{0}}=0, \quad \int_{\mathbb{R}^{{3}}}p^{\alpha}Q(f,f)\frac{d^{3}\vec{p}}{p^{0}}=0,
\end{equation}
so that  $1$ and $p^{\alpha}$ are called {\em collision invariants}.
%
Moreover,   \eqref{eq:Boltz}  satisfies
the entropy dissipation relation (in the sense of classical statistics)  
\[
\int_{\mathbb{R}^{3}}Q(f,f)\ln(f)\frac{d^3\vec p}{p^{0}}\leq0,
\]
where the  equal sign corresponds to the local thermodynamic equilibrium.

The macroscopic description of gas can
be represented by the first and second moments of the distribution function $f$, namely,
the partial particle 4-flow $N^\alpha$
and the partial energy-momentum tensor
$T^{\alpha\beta}$,  defined respectively  by
%
\begin{equation}
\label{eq:NTab}
  N^{\alpha}=c\int_{{\mathbb{R}^{3}}} p^{\alpha}f\frac{d^{3}\vec{p}}{p^{0}},\quad T^{\alpha\beta}=c\int_{{\mathbb{R}^{3}}} p^{\alpha}p^{\beta}f\frac{d^{3}\vec{p}}{p^{0}},
\end{equation}
which  have the Landau-Lifshitz decompositions
\begin{align*}
N^{\alpha}&=nU^{\alpha}+n^{\alpha},
\ \
T^{\alpha\beta}=c^{-2}\varepsilon U^{\alpha}U^{\beta}-\Delta^{\alpha\beta}(P_{0}+\Pi)+ \pi^{\alpha\beta}, 
\end{align*}
where $(U^{\alpha})=\left(\gamma(\vec{u}) c, \gamma(\vec{u})\vec u\right) $
denotes the macroscopic velocity $4$-vector of gas,
$\gamma(\vec{u})=(1-c^{-2}|\vec{u}|^{2})^{-\frac{1}{2}}$
is the Lorentz factor, $\Delta^{\alpha\beta}$ is   defined by
\begin{equation*}
\Delta^{\alpha\beta}:=g^{\alpha\beta}-c^{-2}U^{\alpha}U^{\beta},
\end{equation*}
which is  a symmetric projector onto the 3D subspace and orthogonal to $U^\alpha$, i.e. $\Delta^{\alpha\beta}U_{\beta}=0$.
Here,  the number density $n$, the particle-diffusion current $n^{\alpha}$, the energy density $\varepsilon$, the shear-stress tensor $\pi^{\alpha\beta}$, and the sum of thermodynamic pressure $P_{0}$ and bulk viscous pressure $\Pi$ are   related to the distribution $f$ by
\begin{equation}
\label{eq:variable}
\begin{aligned}
&n:=c^{-2}U_{\alpha}N^{\alpha}=c^{-1}\int_{\mathbb{R}^{3}} Ef\frac{ d^{3}\vec{p} }{p^{0}},\\
&n^{\alpha}:=\Delta^{\alpha}_{\beta}N^{\beta}=c\int_{\mathbb{R}^{3}} p^{<\alpha>}f\frac{d^{3}\vec{p}}{p^{0}},\\
&\varepsilon:=c^{-2}U_{\alpha}U_{\beta}T^{\alpha\beta}=c^{-1}\int_{\mathbb{R}^{3}} E^2f\frac{ d^{3}\vec{p}}{p^{0}},\\
&\pi^{\alpha\beta}:= \Delta_{\mu\nu}^{\alpha\beta} T^{\mu\nu}=c\int_{\mathbb{R}^{3}} p^{<\alpha }p^{\beta>}f\frac{ d^{3}\vec{p} }{p^{0}},\\
&P_{0}+\Pi:=-\frac{1}{3}\Delta_{\alpha\beta}T^{\alpha\beta}=\frac{1}{3c}\int_{\mathbb{R}^{3}} (E^2-m^2c^4)f\frac{ d^{3}\vec{p}}{p^{0}},
\end{aligned}
\end{equation}
where  $E:=U_{\alpha}p^{\alpha}$, $p^{<\alpha>}:=\Delta_{\beta}^{\alpha}p^{\beta}$
and $p^{<\alpha} p^{\beta>}:=\Delta^{\alpha\beta}_{\mu\nu}{p^{\mu}p^{\nu}}$
are the first and second order irreducible tensors, respectively,
 and
 \[
 \Delta^{\alpha\beta}_{\mu\nu}:= \frac{1}{2}\left( \Delta_{\mu}^{\alpha}\Delta_{\nu}^{\beta}+ \Delta_{\mu}^{\beta}\Delta_{\nu}^{\alpha}-\frac{2}{3}\Delta_{\mu\nu}\Delta^{\alpha\beta} \right).
 \]
At the local thermodynamic equilibrium,
$n^\alpha$, $\Pi$, and $\pi^{\alpha\beta}$
 	will be zero.
Multiplying   \eqref{eq:Boltz}
by $1$ and $p^{\alpha}$ respectively, integrating them over $\mathbb{R}^{3}$ with respect to
$\vec{p}$, and using \eqref{eq:colcon} can give the following conservation laws
\begin{equation}
\label{eq:Tab}
    \partial_{\alpha} N^{\alpha}=0,\quad
    \partial_{\alpha} T^{\alpha\beta}=0.
\end{equation}
The present work chooses the
macroscopic velocity 4-vector $U^{\alpha}$    as the velocity of the energy transport
   (the Landau-Lifshitz frame \cite{Lau:1949})
  \begin{equation}
  \label{eq:landau}
  U_{\beta}T^{\alpha\beta}=\varepsilon U^{\alpha},
  \end{equation}
  i.e.
   	\begin{equation}
   	\label{eq:condition-222}
   	\Delta^{\alpha}_{\beta}T^{\beta\gamma}U_{\gamma}=c\int_{{\mathbb{R}^{3}}} Ep^{<\alpha>}f\frac{d^3\vec p}{p^{0}}=0,
   	\end{equation}
and considers the Anderson-Witting model \cite{AW:1974}
 	\begin{equation}
 	Q(f,f)=-\frac{U_{\alpha}p^{\alpha}}{\tau c^2}(f-f^{(0)}),\label{eq:colAW}
 	\end{equation}
 	where  $f^{(0)}=f^{(0)}(\vec x, \vec p, t)$  denotes the distribution function at  the local thermodynamic equilibrium, and the relaxation time
 	$\tau$  may rely on $n$ and $\theta$ and be  defined by
$
\tau=\frac{1}{n\pi d^2\bar{g}}$,
 $d$ denotes the diameter of gas particles,
and $\bar{g}$ is proportional to the mean relative speed $\bar{\xi}$  between two particles, e.g.
$\bar{g}=\sqrt{2}\bar{\xi}$ or $\bar{\xi}$ \cite{RB:2002}.
%
%
%
%
%
%
The local-equilibrium distribution $f^{(0)}$ can be explicitly given by
the  Maxwell-J\"{u}ttner distribution
\begin{equation}
  \label{eq:equm1}
   f^{(0)}=n g^{(0)}, \ g^{(0)}=\frac{\zeta}{4\pi m^{3}c^{3}K_{2}(\zeta)}\exp\left(-m^{-1}c^{-2}\zeta E\right),
\end{equation}
which
can completely determine  the number density $n$ and energy  density $\varepsilon$
by
\begin{equation}
\label{eq:condition}
\begin{aligned}
   n = &n_{0}:=c^{-1}\int_{\mathbb{R}} Ef^{(0)}\frac{d^{3}\vec{p}}{p^{0}},\
\varepsilon=\varepsilon_{0}:=c^{-1}\int_{\mathbb{R}} E^2f^{(0)}\frac{d^3\vec{p}}{p^{0}}=nm c^2\left(G(\zeta)-\zeta^{-1}\right),
\end{aligned}
\end{equation}
where $G(\zeta):=K_{2}^{-1}(\zeta)K_{3}(\zeta)$,
 $\zeta={(k_{B}T)^{-1}(mc^2)}$ is  the ratio between the particle rest energy $mc^2$
and the thermal energy of the gas $k_BT$, $k_B$ denotes the Boltzmann constant,  $T$ is the {thermodynamic} temperature,
and $K_{\nu}(\zeta)$ denotes the modified Bessel function of the second kind, defined by
\begin{equation}
\label{eq:bessel}
   K_{\nu}(\zeta)=\int_{0}^{\infty}\cosh(\nu\vartheta)\exp(-\zeta\cosh\vartheta)d\vartheta,
\end{equation}
satisfying the recurrence relation
\begin{equation}
\label{eq:besselrec}
   K_{\nu+1}(\zeta)=K_{\nu-1}(\zeta)+2\nu\zeta^{-1}K_{\nu}(\zeta).
\end{equation}
The particles behave as non-relativistic (resp. ultra-relativistic) for $\zeta\gg 1$
(resp. $\zeta\ll 1$).

It is  possible to determine the other macroscopic quantities  from the knowledge of
  $f^{(0)}$ by
\begin{equation}
 \label{eq:variable0}
    \begin{aligned}
    &n^{\alpha}_{0}:=c\int_{\mathbb{R}^3} p^{<\alpha>}f^{(0)}\frac{d^3\vec{p}}{p^{0}}=0,\
    P_{0}:=\frac{1}{3c}\int_{\mathbb{R}^3} (E^2-m^2c^4)f^{(0)}\frac{d^3\vec{p}}{p^{0}}=n k_BT=n mc^2\zeta^{-1},
    \end{aligned}
\end{equation}
and rewrite the conservation laws \eqref{eq:Tab} into the following form
\begin{equation}
\label{eq:conser}
\begin{aligned}
&\frac{\partial \left(n U^{\alpha}\right)}{\partial x^{\alpha}}=0,\ \
\frac{\partial \left(c^{-2}n h U^{\alpha}U^{\beta}-g^{\alpha\beta}P_{0}\right)}{\partial x^{\beta}}=0,
\end{aligned}
\end{equation}
where $h:=n^{-1}(\varepsilon+P_{0})=mc^2G(\zeta)$ denotes the specific enthalpy.
It means that the special relativistic hydrodynamic equations \eqref{eq:conser} can be derived from  the special relativistic Boltzmann equation \eqref{eq:Boltz} when $f=f^{(0)}$.
How to find the reduced model equations to describe the states with $f\neq f^{(0)}$?
This paper attempts to answer it and derive its globally hyperbolic  moment model of arbitrary order by extending the moment method by operator projection \cite{MR:2014} and \cite{Kuang:2017}
to \eqref{eq:Boltz}, see Section \ref{sec:moment}.

Before ending this section,  let us discuss
 calculation of  the physically admissible macroscopic states $\{n,\vec{u},\theta  ={\zeta}^{-1}:$
$n>0,|\vec{u}|<{c}, \theta>0\}$ from
a given nonnegative distribution   $f(x,\vec{p},t)$,
which is not identically zero.

\begin{thm}
	\label{thm:admissible}
	For the nonnegative distribution  $f(x,\vec{p},t)$,  which is not identically zero, the density current $N^{\alpha}$ and  energy-momentum tensor $T^{\alpha\beta}$ calculated by \eqref{eq:NTab}
satisfy the properties:
\begin{itemize}
  \item[(i)] $T^{\alpha\beta}$ is positive definite, and the matrix-pair $(T^{\alpha\beta},g^{\alpha\beta})$ has an unique positive generalized eigenvalue $\varepsilon$ larger than
       $nm c^2$ and  corresponding generalized eigenvector $U_{\alpha}$ satisfying $U_{0}=\sqrt{U_{1}^2+U_{2}^2+U_{3}^{2}+c^2}$, i.e., $T^{\alpha\beta}U_{\beta}=\varepsilon g^{\alpha\beta}U_{\beta}$.
    \item[(ii)] The macroscopic velocity vector
      $\vec{u}:=-c(U_{0}^{-1}U_{1},U_{0}^{-1}U_{2},U_{0}^{-1}U_{3})^{T}$ is bounded by the speed of light, that is,
       $|\vec{u}|<c$.
  \item[(ii)] The number density $n:=c^{-2}U_{\alpha}N^{\alpha}$ is positive.
  \item[(iii)] The equation
  \begin{equation}
  \label{eq:admT}
  G(\theta^{-1})-\theta=n^{-1}m^{-1}c^{-2}\varepsilon,
  \end{equation}
   has an unique positive solution for $\theta \in (0,+\infty)$.
\end{itemize}
\end{thm}
\begin{proof}
{\em (i)} For  any nonzero vector $\vec{X}=(x_{0},x_{1},x_{2},x_{3})^{T}\in \mathbb R^4$ and the given nonnegative distribution $f(x,\vec{p},t)$, which is not identically zero,
		using \eqref{eq:NTab} gives
\begin{align*}
    &\vec{X}^{T}\left(T^{\alpha\beta}\right)\vec{X}=c\vec{X}^{T}
    \left(\int_{\mathbb{R}^{3}}p^{\alpha}p^{\beta}f\frac{d^3\vec{p}}{p^{0}}\right)\vec{X}
    =c\int_{\mathbb{R}^{3}}x_{\alpha}p^{\alpha}p^{\beta}x_{\beta}f\frac{d^3\vec{p}}{p^{0}}
    =c\int_{\mathbb{R}^{3}}(x_{\alpha}p^{\alpha})^2f\frac{d^3\vec{p}}{p^{0}}>0.
\end{align*}
It means  that $T^{\alpha\beta}$ is positive definite.

Thanks of  \eqref{eq:landau}, 
the matrix-pair $(T^{\alpha\beta},g^{\alpha\beta})$ has an unique positive generalized eigenvalue $\varepsilon$, which satisfies
$$
U_{\alpha}T^{\alpha\beta}U_{\beta}=\varepsilon U_{\alpha}g^{\alpha\beta}U_{\beta}, \ \
U_{\alpha}\neq 0.
$$
Because $T^{\alpha\beta}$ is positive definite, the left hand side is larger than zero
and thus one has  $U_{0}^2>U_{1}^2+U_{2}^2+U_{3}^{2}$ and   $U_0=\sqrt{U^2_1+U^2_2+U^2_3+c^2}$  via multiplying $(U_{\alpha})$ by a scaling constant
$c(U^2_0-U^2_1-U^2_2-U^2_3)^{-1/2}$. As a result, the macroscopic velocity vector $\vec{u}$
can be calculated by
$\vec{u}=-c(U^{-1}_0U_1,U^{-1}_0U_2,U^{-1}_0U_3)^{T}$ and satisfies
\begin{equation}
  |\vec{u}| = cU^{-1}_0\sqrt{U^2_1+U^2_2+U^2_3}<c.
\end{equation}

{\em (ii)}
Using the Cauchy-Schwarz inequality gives
\[
E-mc^2=\left(\sqrt{\sum_{i=1}^3 U_{i}^2+c^2}\sqrt{\sum_{i=1}^3 p_{i}^2+m^2c^2}-\sum_{i=1}^3 U_{i}p_{i}\right)-mc^2>0,
\]
which implies
\[
n=c^{-2}U_{\alpha}N^{\alpha}=c^{-1}\int_{\mathbb{R}^{3}}Ef\frac{d^{3}\vec{p}}{p^{0}}>0.
\]

{\em (iii)} Because $E>mc^2$ in (ii), one has
\begin{align*}
\varepsilon-nm c^2&=c^{-2}U_{\alpha}T^{\alpha\beta}U_{\beta}-mU_{\alpha}N^{\alpha}
=c^{-1}\left(\int_{\mathbb{R}^{3}}E^2f\frac{d^{3}\vec{p}}{p^{0}}-mc^2\int_{\mathbb{R}^{3}}Ef\frac{d^{3}\vec{p}}{p^{0}}\right)\\
&=c^{-1}\int_{\mathbb{R}^{3}}E(E-mc^2)f\frac{d^{3}\vec{p}}{p^{0}}>0.
\end{align*}

On the other hand,
it holds 
	\[
	\lim_{\theta\rightarrow0}\left(G(\theta^{-1})-\theta\right)=1,\
	\lim_{\theta\rightarrow +\infty}\left(G(\theta^{-1})-\theta\right)=\lim_{\theta\rightarrow +\infty}3\theta= +\infty,
	\]
	and
{\begin{align*}
	\frac{\partial (G(\theta^{-1})-\theta)}{\partial \theta}
	=&-\theta^{-2}
	\left(G(\theta^{-1})^2-5G(\theta^{-1})\theta+\theta^{2}-1\right)=:\tilde{\psi}(G(\theta^{-1}),\theta).
	\end{align*}
Because
\begin{align*}
0&<mc\int_{\mathbb{R}^{3}}
\frac{E-mc^2}{E+mc^2}
f^{(0)}\frac{d^{3}\vec{p}}{p^{0}}=-{n}\theta \left((3\theta+2)G(\theta^{-1})-2(6\theta^2+4\theta+1)\right),\\
0&<c^{-1}\int_{\mathbb{R}^{3}}(E-mc^2)^2f^{(0)}\frac{d^{3}\vec{p}}{p^{0}}={nm} c^{2}(2G(\theta^{-1}) - 5\theta-2),
\end{align*}
one gets
\[
\frac{5}{2}\theta+1<  G(\theta^{-1})<\frac{2(6\theta^2+4\theta+1)}{3\theta+2}.
\]
Hence  one has
\[
\tilde{\psi}(G(\theta^{-1}),\theta)>
\tilde{\psi}\left(\frac{2(6\theta^2+4\theta+1)}{3\theta+2},\theta\right)>3(3\theta+2)^{-2}(9\theta^2+12\theta+1)>0,
\]
i.e.
\[
\frac{\partial (G(\theta^{-1})-\theta)}{\partial \theta}>0,
\]}
which implies that 	$G(\theta^{-1})-\theta$  is a strictly monotonic function  of $\theta$ in the interval $(0,+\infty)$.

%
Thus   \eqref{eq:admT} has an unique solution in the interval $(0,+\infty)$.
	The proof is completed.
	\qed\end{proof}

Furthermore, the following conclusion holds.
	\begin{thm}
		\label{lem:admissible1}
		Under the assumptions of Theorem \ref{thm:admissible}, the bulk viscous pressure $\Pi$
		satisfies
		\[
		\Pi >  -nm {c^{2}}\theta.
		\]
	\end{thm}
	\begin{proof}
Using Theorem \ref{thm:admissible}
yields that $\{n,\vec{u},\theta\}$ satisfy
\begin{equation}
\label{eq:positivity}
		n>0, \quad |\vec{u}|<{c}, \quad \theta>0.
\end{equation}		
Combining them with the last equations in \eqref{eq:variable} and \eqref{eq:variable0} gives
		\[ {\Pi=-\frac{1}{3}\int_{\mathbb{R}^{3}}\Delta_{\alpha\beta}p^{\alpha}p^{\beta}f\frac{d^{3}\vec{p}}{p^{0}}-n mc^{2} \theta=-\frac{1}{3c}\int_{\mathbb{R}^{3}}(E^2-m^{2}c^{4})f\frac{d^{3}\vec{p}}{p^{0}}-nm c^{2}\theta>-nm c^{2}\theta}.
		\]
 	The proof is completed.
		\qed\end{proof}


\begin{remark}
Theorem \ref{thm:admissible} provides a recovery procedure of the admissible primitive variables $\rho, \vec{u}$, and $\theta$  from the nonnegative distribution $f(\vec{x},\vec{p},t)$ or the given density current $N^{\alpha}$  and  energy-momentum tensor $T^{\alpha\beta}$ satisfying
\begin{equation}
\label{eq:consTN}
 U_{\alpha}N^{\alpha}>0, \mbox{and } T^{\alpha\beta}\ \mbox{is positive definite}.
\end{equation}
	    It is  useful in  the derivation of the moment system as well as the numerical scheme.
\end{remark}


\begin{remark}
If  setting
$$
\vec{x}=L\hat{\vec{x}}, ~ \vec{p}=c\hat{\vec{p}}, ~ p^{0}=c\hat{p^{0}},~ t=\frac{L}{c}\hat{t},~
g=c\hat{g},~ f=\frac{n_{0}}{c^3}\hat{f},
$$
where  $L$,
 $n_{0}$, and $\theta_{0}=m c^2/k_{B}$ are the macroscopic characteristic length, the reference
 number density and  temperature,
 respectively, then  the relativistic Boltzmann equation \eqref{eq:Boltz} with the Anderson-Witting model \eqref{eq:colAW} can be   non-dimensionalized as follows
\begin{equation}\label{EQ:0000000001}
\hat{p}^{\alpha}\frac{\partial \hat{f}}{\partial \hat{x}^{\alpha}}=
\frac{\hat{n}}{K_{n}}\hat{U}_{\alpha}\hat{p}^{\alpha}\left(\hat{f}^{(0)}-
\hat{f}\right),
\end{equation}
where  $K_{n}=\frac{\lambda}{L}=\frac{\tau_{0}c}{L}=\frac{1}{n_{0}L\pi d^2}$ denotes
 the Knudsen number. 
If considering $\tilde\tau:=\frac{K_{n}}{\hat{n}}$   as a new ``relaxation time'',
then in  \eqref{EQ:0000000001} the collision term
is the same as that of the non-relativistic Bhatnagar-Gross-Krook  model, and
 the notations $\tilde\tau$, $\hat{\vec{x}}$, $\hat t$, $\hat f$, $\hat{\vec{p}}$, $\hat p^{0}$, $\hat n$
can still be  replaced with $\tau$, $\vec{x}$, $t$, $f$, $\vec{p}$, $p^{0}$, $n$  respectively.
\end{remark}

\section{Families of orthogonal polynomials}
\label{sec:orth}
This section  introduces (infinite) families of Grad type orthogonal polynomials dependent on a parameter $\zeta$,
 and studies their properties.
 The polynomials are the same as those in \cite{GRADP:1974}, but different from those in \cite{Kuang:2017}.
 Their properties are not  discussed in the literature  but are crucial for
 deriving  globally hyperbolic moment model of arbitrary order for the three-dimensional special
relativistic Boltzmann equation.

If considering
\begin{equation}
\label{eq:omegal}
  \omega^{(\ell)}(x;\zeta)=\frac{\zeta(x^2-1)^{\ell+\frac{1}{2}}}{(2\ell+1)K_{2}(\zeta)}\exp\left(-\zeta x\right),\ \ell\in\mathbb{N},
\end{equation}
as  the weight functions in the interval $[1,+\infty)$, where  $\zeta\in\mathbb{R}^{+}$ denotes
 a parameter, and $\mathbb{N}$ is the set of the non-negative integers,
then  the inner product with respect to $\omega^{(\ell)}(x;\zeta)$ may be introduced as follows
\[
\left(f,g\right)_{\omega^{(\ell)}}:=\int_{1}^{+\infty}f(x)g(x)\omega^{(\ell)}(x;\zeta)dx,
\ \ f,g\in L^{2}_{\omega^{(\ell)}}[1,+\infty),\ \ell\in\mathbb{N},
\]
where  $L^{2}_{\omega^{(\ell)}}[1,+\infty):=\left\{f\big|\int_{1}^{+\infty}f(x)^2\omega^{(\ell)}(x;\zeta)dx<+\infty\right\}$.
It is worth noting that
the choice of the weight function $\omega^{(\ell)}(x;\zeta)$
is dependent on the equilibrium distribution {$f^{(0)}(x,p, t)$} in
\eqref{eq:equm1}.

Let $\{P_{k}^{(\ell)}(x;\zeta)\}$, $\ell\in\mathbb{N}$, be infinite families of standard orthogonal polynomials
with respect to the weight function $\omega^{(\ell)}(x;\zeta)$ in the interval  $[1,+\infty)$, i.e.
\begin{equation}
  \label{eq:P01orth}
     \left(P_{i}^{(\ell)},P_{k}^{(\ell)}\right)_{\omega^{(\ell)}}=\delta_{i,k},\quad
  \ell\in\mathbb{N},
\end{equation}
where the degree of $P_{k}^{(\ell)}(x;\zeta)$ is equal to $k$,
$\delta_{i,k}$ denotes the Kronecker delta function,
which is equal to 1 if $i=k$, and 0 otherwise.
Obviously,   $\{P_{k}^{(\ell)}(x;\zeta)\}$ satisfies
\begin{equation}
\label{eq:P01orth-2}
\left(P_{k}^{(\ell)},x^{i}\right)_{\omega^{(\ell)}}=0,\ i\leq k-1,
\end{equation}
and
\begin{equation}
\label{eq:P01orth-3}
Q(x;\zeta)=\sum_{i=0}^{k}\left(Q(x;\zeta),P_{k}^{(\ell)}\right)_{\omega^{(\ell)}}P_{k}^{(\ell)}(x;\zeta),
\end{equation}
for any   polynomial $Q(x;\zeta)$ of degree $\leq k$ in $L^{2}_{\omega^{(\ell)}}[1,+\infty)$.

The orthogonal polynomials $\{P_{k}^{(\ell)}(x;\zeta)\}$
may be obtained by using the Gram-Schmidt  process.
For example, several orthogonal polynomials of lower degree  
are given as follows
\begin{align}\label{EQ-3.3aaaaaaa}
	\begin{aligned}
     P_{0}^{(0)}(x;\zeta)&=\frac{1}{\sqrt{G(\zeta)-4\zeta^{-1}}},\\
     P_{1}^{(0)}(x;\zeta)&=\frac{\sqrt{G(\zeta)-4\zeta^{-1}}}{\sqrt{G(\zeta)^2-5\zeta^{-1}G(\zeta)+4\zeta^{-2}-1}}\left(x-\frac{1}{G(\zeta)-4\zeta^{-1}}\right),\\
     P_{2}^{(0)}(x;\zeta)&=\frac{\zeta\sqrt{G(\zeta)^2-5\zeta^{-1}G(\zeta)+4\zeta^{-2}-1}}{\sqrt{3}\sqrt{2G(\zeta)^3-13\zeta ^{-1}G(\zeta)^2-2G(\zeta)+20\zeta^{-2}G(\zeta)+3\zeta^{-1}}}\\
     &\cdot \left(x^2-3\frac{G(\zeta)^2-4\zeta^{-1}G(\zeta)-1}{\zeta\left(G(\zeta)^2-5\zeta^{-1}G(\zeta)+4\zeta^{-2}-1\right)}x
     -\frac{G(\zeta)^2-5\zeta^{-1}G(\zeta)+\zeta^{-2}-1}{G(\zeta)^2-5\zeta^{-1}G(\zeta)+4\zeta^{-2}-1}\right),\\
     P_{0}^{(1)}(x;\zeta)&=\sqrt{\zeta},\\
     P_{1}^{(1)}(x;\zeta)&=\frac{\sqrt{\zeta}}{\sqrt{-G(\zeta)^2+5\zeta^{-1}G(\zeta)+1}}\left(x-G(\zeta)\right),\\
     P_{0}^{(2)}(x;\zeta)&=\frac{\zeta}{\sqrt{3G(\zeta)}},
\end{aligned}\end{align}
 and plotted  in Fig. \ref{fig:polyp} with respect to $x$ and $\zeta$.
Obviously,
the coefficients in those  orthogonal polynomials are so irregular
that it is quite complicate to  study
 the properties of $\{P_{k}^{(\ell)}(x;\zeta)\}$.
%
Let $c_{k}^{(\ell)}$  be the leading coefficient of
$P_{k}^{(\ell)}(x;\zeta)$, and without loss of generality,  assume $c_{k}^{(\ell)}>0$, $\ell\in\mathbb{N}$.
%
Because the polynomial   $P_{k}^{(\ell)}(x;\zeta)$ has   exactly $n$ real simple zeros
in the interval $(1,+\infty)$, denoted by $\{x_{i,k}^{(\ell)}\}_{i=1}^{k}$ in an increasing order,
 the polynomial $P_{k}^{(\ell)}(x;\zeta)$ can be rewritten as follows
\begin{equation}
     \label{eq:Preprezero}
     P_{k}^{(\ell)}(x;\zeta)=c_{k}^{(\ell)}\prod_{i=1}^{k}(x-x_{i,k}^{(\ell)}),\ \ell\in \mathbb N.
\end{equation}

\begin{figure}
  \centering
  \subfigure[ $P_{k}^{(\ell)}(x,\zeta)$ at  $\zeta=1$.]
  {
  \includegraphics[width=0.4\textwidth]{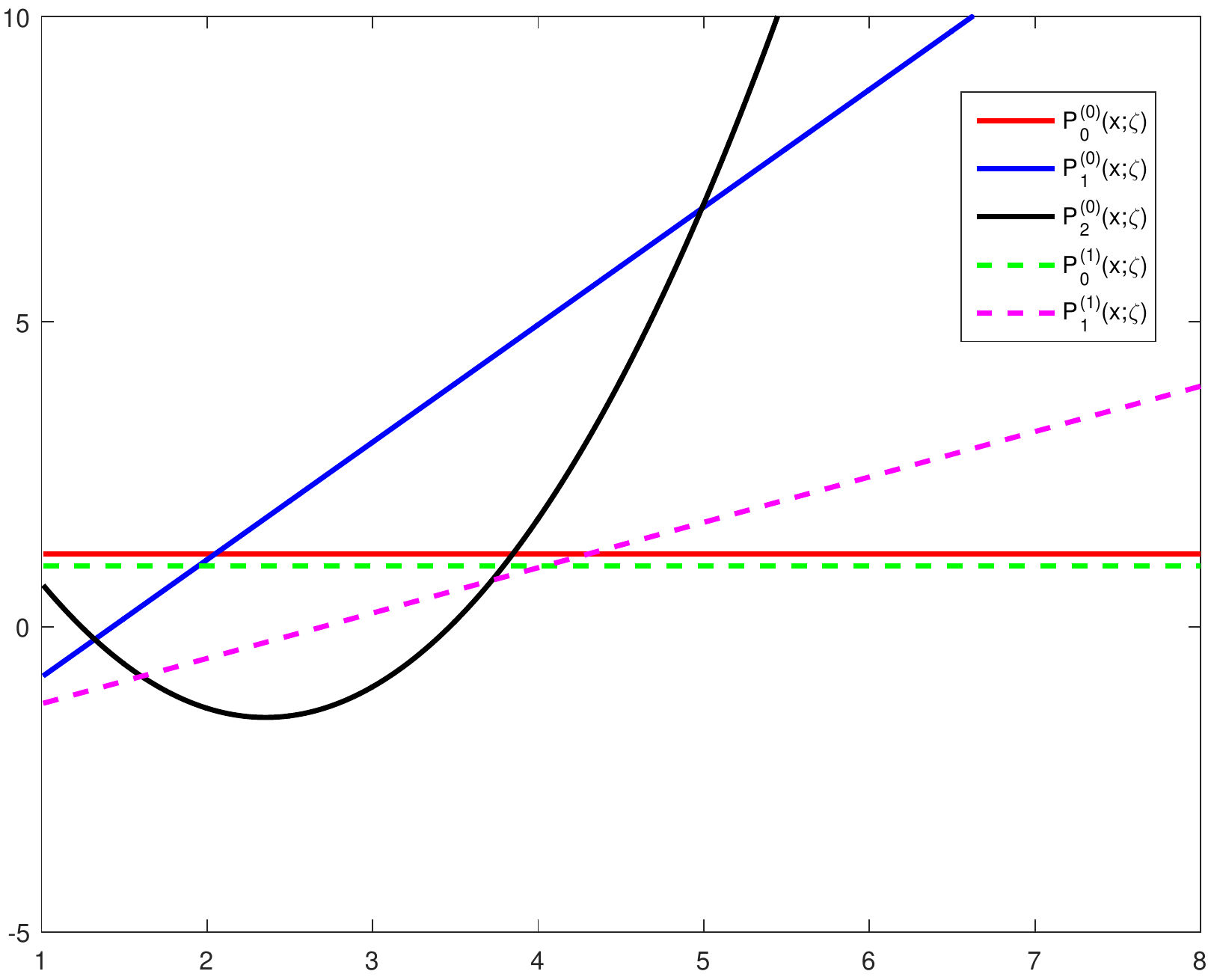}
  }
  \subfigure[ $P_{k}^{(\ell)}(x,\zeta)$ at $x=3$.]
  {
  \includegraphics[width=0.4\textwidth]{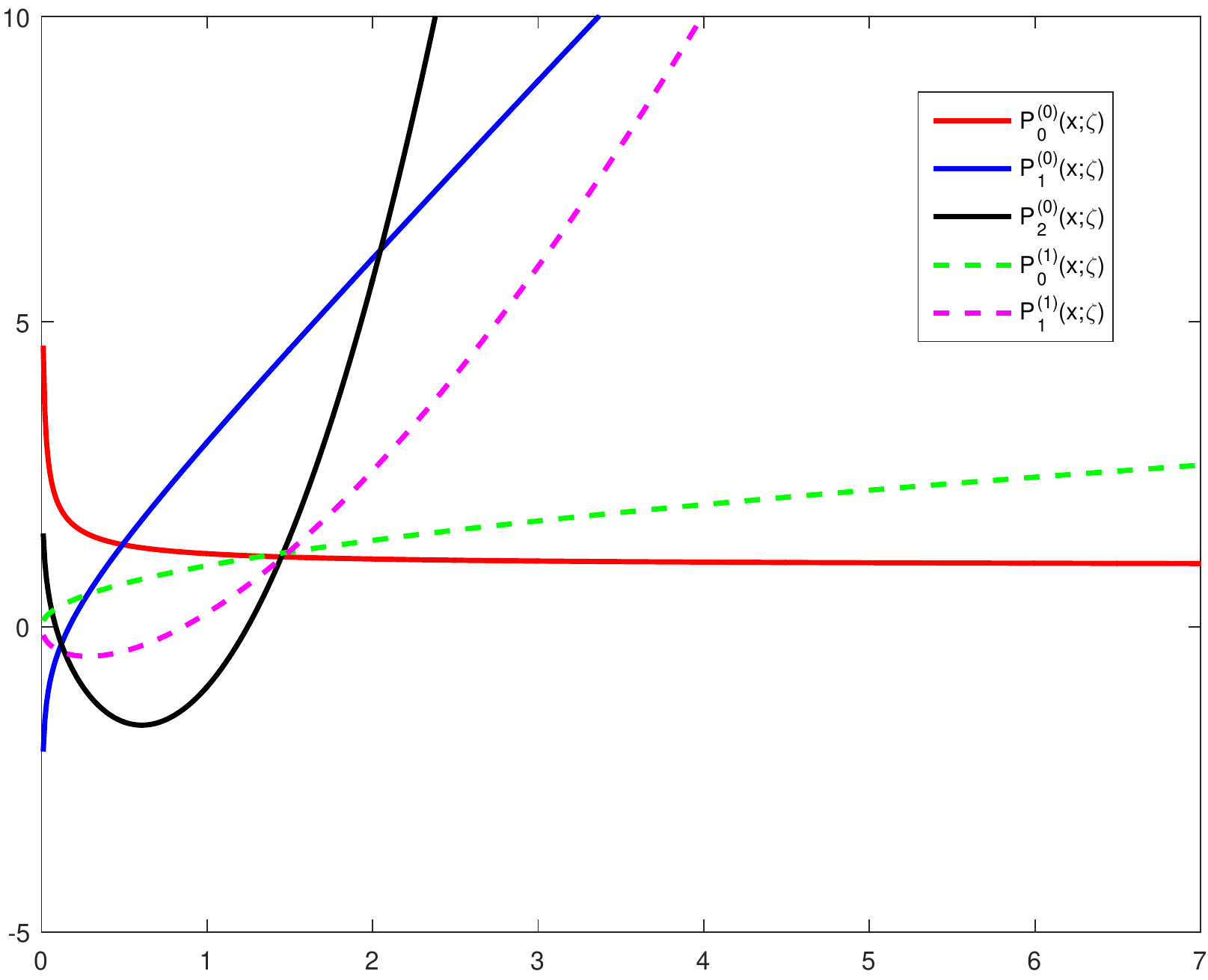}
   }
  \caption{\small The polynomials $P_{k}^{(\ell)}(x,\zeta)$ given in \eqref{EQ-3.3aaaaaaa}.}
\label{fig:polyp}
\end{figure}

 In the following, we will derive the  recurrence relations of $\{P_{n}^{(\ell)}(x;\zeta)\}$,
 calculate their derivatives with respect to $x$ and $\zeta$,  respectively,
 and study the properties of zeros and coefficient matrices in the  recurrence relations.


\subsection{Recurrence relations}
\label{subsec:recurr}
This section presents  the recurrence relations for the orthogonal polynomials
 $\{P_{k}^{(\ell)}(x;\zeta)\}$, $\ell\in\mathbb{N}$,
 the recurrence relations between  $\{P_{k}^{(\ell)}(x;\zeta)\}$ and  $\{P_{k}^{(\ell-1)}(x;\zeta)\}$, $\ell\in\mathbb{N}^{+}=\mathbb N/\{0\}$,
 and the specific forms of the coefficients in those recurrence relations.

Using  the three-term recurrence relation
and the existence theorem of zeros of general orthogonal polynomials in Theorems 3.1 and 3.2 of \cite{SP:2011}
gives the following conclusion.

\begin{thm}
\label{thm:rec}
For $\ell\in\mathbb{N}$,
a three-term recurrence relation for
the orthogonal polynomials $\{P_{k}^{(\ell)}(x;\zeta)\}$ can be given by
\begin{equation}
  xP_{k}^{(\ell)}=a_{k-1}^{(\ell)}P_{k-1}^{(\ell)}+b_{k}^{(\ell)}P_{k}^{(\ell)}+a_{k}^{(\ell)}P_{k+1}^{(\ell)}, \label{eq:recP0P1}
\end{equation}
or in the matrix-vector form
\begin{equation}
\label{eq:recP0P1mat}
x\vec{P}_{k}^{(\ell)}=\vec{J}_{k}^{(\ell)}\vec{P}_{k}^{(\ell)}+a_{k}^{(\ell)}P_{k+1}^{(\ell)}\vec{e}_{k+1},
\  \vec{P}_{k}^{(\ell)}:=(P_{0}^{(\ell)},\cdots,P_{k}^{(\ell)})^{T},
\end{equation}
where  both coefficients
\begin{equation}
\label{eq:abn}
a_{k}^{(\ell)}:=\left(xP_{k}^{(\ell)},P_{k+1}^{(\ell)}\right)_{\omega^{(\ell)}}=\frac{c_{k}^{(\ell)}}{c_{k+1}^{(\ell)}}, \quad b_{k}^{(\ell)}:=\left(xP_{k}^{(\ell)},P_{k}^{(\ell)}\right)_{\omega^{(\ell)}}=\sum_{i=1}^{k+1}x_{i,k+1}^{(\ell)}-\sum_{i=1}^{k}x_{i,k}^{(\ell)},
\end{equation}
are positive,
 $\vec{e}_{k+1}$ is the last column of the identity matrix of order $(k+1)$,
 and
\[
     \vec{J}_{k}^{(\ell)}:=\begin{pmatrix}
     b_{0}^{(\ell)} &  a_{0}^{(\ell)}  & 0        &            &   & \\
     a_{0}^{(\ell)}& b_{1}^{(\ell)} & a_{1}^{(\ell)}&             &  &   \\
 &          \ddots    &     \ddots     &  \ddots  &      &     \\
                         &    &    &  a_{k-2}^{(\ell)}   & b_{k-1}^{(\ell)} &   a_{k-1}^{(\ell)}   \\
                            &           &    &    0 & a_{k-1}^{(\ell)}   & b_{k}^{(\ell)} \\
      \end{pmatrix}\in {\mathbb R}^{(k+1)\times(k+1)},
\]
which is a symmetric positive definite  tridiagonal matrix
with the spectral radius  larger than $1$.
\end{thm}

Besides those, the recurrence relations between  $\{P_{k}^{(\ell)}(x;\zeta)\}$ and  $\{P_{k}^{(\ell-1)}(x;\zeta)\}$
can also be obtained.

\begin{thm}\label{thm:huxiang}
 For $\ell\in\mathbb{N}^{+}$, it holds:\\
{\tt(i)} The three-term recurrence relations between $\{P_{k}^{(\ell)}(x;\zeta)\}$ and $\{P_{k}^{(\ell-1)}(x;\zeta)\}$ can be given
by
\begin{align}
 &\frac{2\ell-1}{2\ell+1}(x^2-1)P_{k}^{(\ell)}=p_{k}^{(\ell)}P_{k}^{(\ell-1)}+q_{k}^{(\ell)}P_{k+1}^{(\ell-1)}+r_{k+1}^{(\ell)}P_{k+2}^{(\ell-1)},\label{eq:recP01}\\
    &P_{k+1}^{(\ell-1)}=r_{k}^{(\ell)}P_{k-1}^{(\ell)}+q_{k}^{(\ell)}P_{k}^{(\ell)}+p_{k+1}^{(\ell)}P_{k+1}^{(\ell)},\label{eq:recP10}
\end{align}
or in the matrix-vector form
\begin{align}
\vec{P}_{k+1}^{(\ell-1)}&=(\vec{\tilde{J}}_{k}^{(\ell)})^{T}\vec{P}_{k}^{(\ell)}+p_{k+1}^{(\ell)}P_{k+1}^{(\ell)}\vec{e}_{k+2},\label{eq:recPQ}\\
\frac{2\ell-1}{2\ell+1}(x^2-1)\vec{P}_{k}^{(\ell)}&=\vec{\tilde{J}}_{k}^{(\ell)}\vec{P}_{k+1}^{(\ell-1)}+r_{k+1}^{(\ell)}P_{k+2}^{(\ell-1)}\vec{e}_{k+1}\label{eq:recQP},
\end{align}
where
\begin{equation}
\label{eq:pqrn}
\begin{aligned}
p_{k}^{(\ell)}:=\frac{c_{k}^{(\ell-1)}}{c_{k}^{(\ell)}},\quad
 r_{k}^{(\ell)}:=&\frac{2\ell-1}{2\ell+1}\frac{c_{k-1}^{(\ell)}}{c_{k+1}^{(\ell-1)}},\\
 q_{k}^{(\ell)}:=\frac{2\ell-1}{2\ell+1}\frac{c_{k}^{(\ell)}}{c_{k+1}^{(\ell-1)}}\left(\sum_{i=1}^{k+2}x_{i,k+2}^{(\ell-1)}-\sum_{i=1}^{k}x_{i,k}^{(\ell)}\right)
&=\frac{c_{k+1}^{(\ell-1)}}{c_{k}^{(\ell)}}
\sum_{i=1}^{k+1}  ( x_{i,k+1}^{(\ell)}-x_{i,k+1}^{(\ell-1)}),
\end{aligned}
\end{equation}
 and
\[
    \vec{\tilde{J}}_{k}^{(\ell)}:=\begin{pmatrix}
    p_{0}^{(\ell)} & q_{0}^{(\ell)} & r_{1}^{(\ell)} & 0     & 0 & \cdots & 0\\
    0     & p_{1}^{(\ell)} & q_{1}^{(\ell)} & r_{2}^{(\ell)} & 0 & \cdots & 0\\
    \     & \ddots& \ddots& \ddots& \ & \      & \ \\
    \     & \     &     \ &    \  & 0 & p_{k}^{(\ell)}  & q_{k}^{(\ell)}
    \end{pmatrix}
    \in {\mathbb R}^{(k+1)\times (k+2)}.
\]
{\tt (ii)}   The two-term recurrence relations
 between  $\{P_{k}^{(\ell-1)}(x;\zeta)\}$ and  $\{P_{k}^{(\ell)}(x;\zeta)\}$
 can be derived as follows
\begin{align}
    &\frac{2\ell-1}{2\ell+1}(x^2-1)P_{k}^{(\ell)}=\tilde{p}_{k}^{(\ell)}(x+\tilde{q}_{k}^{(\ell)})P_{k+1}^{(\ell-1)}+\tilde{r}_{k}^{(\ell)}P_{k}^{(\ell-1)},
    \label{eq:recP01x}\\
    &P_{k+1}^{(\ell-1)}=\frac{2\ell-1}{2\ell+1}\frac{1}{\tilde{p}_{k}^{(\ell)}}(x-\tilde{q}_{k}^{(\ell)})P_{k}^{(\ell)}-\frac{a_{k-1}^{(\ell)}}{a_{k}^{(\ell-1)}}\tilde{r}_{k}^{(\ell)}P_{k-1}^{(\ell)},\label{eq:recP10x}
\end{align}
where
\begin{equation}
\label{eq:pqrnx}
\tilde{p}_{k}^{(\ell)}:=\frac{2\ell-1}{2\ell+1}\frac{c_{k}^{(\ell)}}{c_{k+1}^{(\ell-1)}},\
 \tilde{q}_{k}^{(\ell)}:=\sum_{i=1}^{k+1}x_{i,k+1}^{(\ell-1)}-\sum_{i=1}^{k}x_{i,k}^{(\ell)},
\ \tilde{r}_{k}^{(\ell)}:={p_{k}^{(\ell)}\left(1-\frac{2\ell+1}{2\ell-1}(\tilde{p}_{k}^{(\ell)})^{2}\right)}.
\end{equation}
\end{thm}
\begin{proof}
(i)  For $i\leq k+2$,  taking the inner product with respect to $\omega^{(\ell-1)}$ between the polynomials  $P_{i}^{(\ell-1)}(x;\zeta)$
and $(x^2-1)P_{k}^{(\ell)}(x;\zeta)$ 
gives
\begin{align*}
\frac{2\ell-1}{2\ell+1}&\left((x^2-1)P_{k}^{(\ell)},P_{k+2}^{(\ell-1)}\right)_{\omega^{(\ell-1)}}=
\frac{2\ell-1}{2\ell+1}\left(c_{k}^{(\ell)}x^{k+2},P_{k+2}^{(\ell-1)}\right)_{\omega^{(\ell-1)}}
\\
&=\frac{2\ell-1}{2\ell+1}\frac{c_{k-1}^{(\ell)}}{c_{k+1}^{(\ell-1)}}
\left(P_{k+2}^{(\ell-1)},P_{k+2}^{(\ell-1)}\right)_{\omega^{(\ell-1)}}=r_{k+1},
\\
\frac{2\ell-1}{2\ell+1}&\left((x^2-1)P_{k}^{(\ell)},P_{k+1}^{(\ell-1)}\right)_{\omega^{(\ell-1)}}=
\frac{2\ell-1}{2\ell+1}\left(c_{k}^{(\ell)}\left(x^{k+2}-\sum_{i=1}^{k+2}
x_{i,k+2}^{(\ell-1)}x^{k+1}\right.\right.
\\
&+\left.\left.\left(\sum_{i=1}^{k+2}x_{i,k+2}^{(\ell-1)}-\sum_{i=1}^{k}
x_{i,n}^{(\ell)}\right)x^{k+1}\right),P_{k+1}^{(\ell-1)}\right)_{\omega^{(\ell-1)}}
\\
 &=r_{k+1}^{(\ell)}\left(P_{k+2}^{(\ell-1)},P_{k+1}^{(\ell-1)}\right)_{\omega^{(\ell-1)}}+
 q_{k}^{(\ell)}\left(P_{k+1}^{(\ell-1)},P_{k+1}^{(\ell-1)}\right)_{\omega^{(\ell-1)}}\\
 &=q_{k}^{(\ell)},\\
\frac{2\ell-1}{2\ell+1}&\left((x^2-1)P_{k}^{(\ell)},P_{k+1}^{(\ell-1)}\right)_{\omega^{(\ell-1)}}=\left(P_{k}^{(\ell)},c_{k+1}^{(\ell-1)}
\left(x^{k+1}-\sum_{i=1}^{k+1}x_{i,k+1}^{(\ell)}x^{k}\right.\right.\\
&+\left.\left.
\left(\sum_{i=1}^{k+1}x_{i,k+1}^{(\ell)}-\sum_{i=1}^{k+1}x_{i,k+1}^{(\ell-1)}\right)x^{k}\right)\right)_{\omega^{(\ell)}}\\
 &=p_{k+1}^{(\ell)}\left(P_{k}^{(\ell)},P_{k+1}^{(\ell)}\right)_{\omega^{(\ell)}}+\frac{c_{k+1}^{(\ell-1)}}{c_{k}^{(\ell)}}\sum_{i=1}^{k+1} \left(x_{i,k+1}^{(\ell)}-x_{i,k+1}^{(\ell-1)}\right)\left(P_{k}^{(\ell)},P_{k}^{(\ell)}\right)_{\omega^{(\ell)}}\\
 &=\frac{c_{k+1}^{(\ell-1)}}{c_{k}^{(\ell)}}
\sum_{i=1}^{k+1} \left(x_{i,k+1}^{(\ell)}-x_{i,k+1}^{(\ell-1)}\right)=q_{k}^{(\ell)},
 \\
\frac{2\ell-1}{2\ell+1}&\left((x^2-1)P_{k}^{(\ell)},P_{k}^{(\ell-1)}\right)_{\omega^{(\ell-1)}}
=\left(P_{k}^{(\ell)},P_{k}^{(\ell-1)}\right)_{\omega^{(\ell)}}=
   \left(P_{k}^{(\ell)},c_{k}^{(\ell-1)}x^{k}\right)_{\omega^{(\ell)}}\\
   &=p_{k}^{(\ell)}\left(P_{k}^{(\ell)},P_{k}^{(\ell)}\right)_{\omega^{(\ell)}}=
   p_{k}^{(\ell)},\\
    \frac{2\ell-1}{2\ell+1}&\left((x^2-1)P_{k}^{(\ell)},P_{i}^{(\ell-1)}\right)_{\omega^{(\ell-1)}}
    =\left(P_{k}^{(\ell)},P_{i}^{(\ell-1)}\right)_{\omega^{(\ell)}}=0, \quad i\leq k-1,
\end{align*}
Substituting them into  \eqref{eq:P01orth-3} gives  \eqref{eq:recP01}.

(ii) Taking the inner product with respect to $\omega^{(\ell)}$ between $P_{k+1}^{(\ell-1)} (x;\zeta)$
 and $P_{i}^{(\ell)} (x;\zeta)$  with $i\leq k+1$ 
\begin{align*}
\left(P_{k+1}^{(\ell-1)},P_{k+1}^{(\ell)}\right)_{\omega^{(\ell)}}&= \left(c_{k+1}^{(\ell-1)}x^{k+1},P_{k+1}^{(\ell)}\right)_{\omega^{(\ell)}}
=p_{k+1}\left(P_{k+1}^{(\ell)},P_{k+1}^{(\ell)}\right)_{\omega^{(\ell)}}=p_{k+1}^{(\ell)},\\
\left(P_{k+1}^{(\ell-1)},P_{k}^{(\ell)}\right)_{\omega^{(\ell)}}&= \frac{2\ell-1}{2\ell+1}\left(P_{k+1}^{(\ell-1)},(x^2-1)P_{k}^{(\ell)}\right)_{\omega^{(\ell-1)}}=q_{k}^{(\ell)},
\\
\left(P_{k+1}^{(\ell-1)},P_{k-1}^{(\ell)}\right)_{\omega^{(\ell)}}&=\frac{2\ell-1}{2\ell+1}\left(P_{k+1}^{(\ell-1)},(x^2-1)P_{k-1}^{(\ell)}\right)_{\omega^{(\ell-1)}}=
r_{k}^{(\ell)}\left(P_{k+1}^{(\ell-1)},P_{k+1}^{(\ell-1)}\right)_{\omega^{(\ell-1)}}=r_{k}^{(\ell)},\\
    \left(P_{k+1}^{(\ell-1)},P_{i}^{(\ell)}\right)_{\omega^{(\ell)}}
    &=\frac{2\ell-1}{2\ell+1}\left(P_{k+1}^{(\ell-1)},(x^2-1)P_{i}^{(\ell)}\right)_{\omega^{(\ell-1)}}=0, \quad i\leq k-2.
\end{align*}
Similarly, substituting them into  \eqref{eq:P01orth-3} gives    \eqref{eq:recP10}.

(iii) If using \eqref{eq:recP0P1} to eliminate $P_{k+2}^{(\ell-1)}$ and $P_{k+1}^{(\ell)}$ in \eqref{eq:recP01} and \eqref{eq:recP10} respectively, then one yields
 \begin{align*}
  &  \frac{2\ell-1}{2\ell+1}(x^2-1)P_{k}^{(\ell)}=\tilde{p}_{k}^{(\ell)}
    (x+\tilde{q}_{k}^{(\ell)})P_{k+1}^{(\ell-1)}+\tilde{r}_{k}^{(\ell)}P_{k}^{(\ell-1)},
    \\
   & P_{k+1}^{(\ell-1)}=\frac{2\ell-1}{2\ell+1}\frac{1}{\tilde{\tilde{p}}_{k}^{(\ell)}}(x-\tilde{\tilde{q}}_{k}^{(\ell)})P_{k}^{(\ell)}-\frac{a_{k-1}^{(\ell)}}{a_{k}^{(\ell-1)}}\tilde{\tilde{r}}_{k}^{(\ell)}P_{k-1}^{(\ell)},
\end{align*}
with 
\[
    \tilde{p}_{k}^{(\ell)}=\frac{r_{k+1}^{(\ell)}}{a_{k+1}^{(\ell-1)}}=\frac{2\ell-1}{2\ell+1}\frac{c_{k}^{(\ell)}}{c_{k+1}^{(\ell-1)}}
    =\frac{2\ell-1}{2\ell+1}\frac{a_{k}^{(\ell)}}{p_{k+1}^{(\ell)}}
    =\tilde{\tilde{p}}_{k},
\]
\begin{equation*}
  \tilde{q}_{k}^{(\ell)}=\frac{1}{\tilde{p}_{k}^{(\ell)}}q_{k}^{(\ell)}-b_{k+1}^{(\ell-1)}=\sum_{i=1}^{k+1}x_{i,k+1}^{(\ell-1)}-\sum_{i=1}^{k}x_{i,k}^{(\ell)}
  =b_{k}^{(\ell)}-\frac{2\ell+1}{2\ell-1}\tilde{p}_{k}^{(\ell)}q_{k}^{(\ell)}=\tilde{\tilde{q}}_{k}^{(\ell)},
\end{equation*}
\begin{equation*}
 \tilde{r}_{k}^{(\ell)}=p_{k}^{(\ell)}-\tilde{p}_{k}^{(\ell)}a_{k}^{(\ell-1)}={p_{k}^{(\ell)}\left(1-\frac{2\ell+1}{2\ell-1}(\tilde{p}_{k}^{(\ell)})^{2}\right)}
 =\frac{a_{k}^{(\ell-1)}}{a_{k-1}^{(\ell)}}\left(-r_{k}^{(\ell)}+\frac{2\ell-1}{2\ell+1}\frac{1}{\tilde{p}_{k}^{(\ell)}}a_{k-1}^{(\ell)}\right)=
 \tilde{\tilde{r}}_{k}^{(\ell)}.
\end{equation*}
The proof is completed.
\qed\end{proof}

\subsection{Partial derivatives}\label{subsec:deriva}
This section
calculates the derivatives of the polynomial $P_{k}^{(\ell)} (x;\zeta)$ with respect to $x$ and $\zeta$,
$\ell\in\mathbb{N}$.
\begin{thm}\label{thm:derivezeta}
	For $\ell\in\mathbb{N}$,
	 the first-order derivative of  the polynomial $P_{k+1}^{(\ell)} (x;\zeta)$
	with respect to the parameter $\zeta$ satisfies
\begin{equation}
 \frac{\partial P_{k+1}^{(\ell)}}{\partial \zeta}=a_{k}^{(\ell)}P_{k}^{(\ell)}-\frac{1}{2}\left(G(\zeta)-\zeta^{-1}-b_{k+1}^{(\ell)}\right)P_{k+1}^{(\ell)}.
 \label{eq:partialPnzeta}
 \end{equation}
\end{thm}
\begin{proof}
With the aid of  definition \eqref{eq:bessel} and   recurrence relation \eqref{eq:besselrec} of the second kind modified
Bessel functions, one has
\begin{align*}
\frac{\partial }{\partial \zeta}\omega^{(\ell)}(x;\zeta)
&=\frac{K_{3}(\zeta)+K_{1}(\zeta)+2\zeta^{-1} K_{2}(\zeta)-2xK_{2}(\zeta)}{2K_{2}(\zeta)}\left(\frac{\zeta}{(2\ell+1)K_{2}(\zeta)}(x^2-1)^{\ell-\frac{1}{2}}\exp(-\zeta x)\right)\\
&=\left(G(\zeta)-\zeta^{-1}-x\right)\omega^{(\ell)}(x;\zeta).
\end{align*}
Taking the partial derivative of both sides of identities
\[
    \left(P_{k+1}^{(\ell)},P_{i}^{(\ell)}\right)_{\omega^{(\ell)}}=\delta_{k+1,i}, \ i=0,\cdots,k+1,
\]
 with respect to $\zeta$ and using \eqref{eq:abn} gives
\begin{align*}
    \frac{\partial }{\partial \zeta}
    \left(P_{k+1}^{(\ell)},P_{k+1}^{(\ell)}\right)_{\omega^{(\ell)}}
    =&2\left(\frac{\partial }{\partial \zeta}P_{k+1}^{(\ell)},P_{k+1}^{(\ell)}\right)_{\omega^{(\ell)}}
    +\left(G(\zeta)-\zeta^{-1}\right)\left(P_{k+1}^{(\ell)},P_{k+1}^{(\ell)}\right)_{\omega^{(\ell)}}-\left(xP_{k+1}^{(\ell)},P_{k+1}^{(\ell)}\right)_{\omega^{(\ell)}}\\
    =&2\left(\frac{\partial }{\partial \zeta}P_{k+1}^{(\ell)},P_{k+1}^{(\ell)}\right)_{\omega^{(\ell)}}
    +\left(G(\zeta)-\zeta^{-1}-b_{k+1}^{(\ell)}\right)=0,\\
    \frac{\partial }{\partial \zeta}
    \left(P_{k+1}^{(\ell)},P_{k}^{(\ell)}\right)_{\omega^{(\ell)}}
    =&\left(\frac{\partial }{\partial \zeta}P_{k+1}^{(\ell)},P_{k}^{(\ell)}\right)_{\omega^{(\ell)}}
    +\left(P_{k+1}^{(\ell)},\frac{\partial }{\partial \zeta}P_{k}^{(\ell)}\right)_{\omega^{(\ell)}}\\
    &+\left(G(\zeta)-\zeta^{-1}\right)\left(P_{k+1}^{(\ell)},P_{k}^{(\ell)}\right)_{\omega^{(\ell)}}-\left(xP_{k}^{(\ell)},P_{k+1}^{(\ell)}\right)_{\omega^{(\ell)}}\\
    =&\left(\frac{\partial }{\partial \zeta}P_{k+1}^{(\ell)},P_{k}^{(\ell)}\right)_{\omega^{(\ell)}}
    -a_{k}^{(\ell)}=0,\\
    \frac{\partial }{\partial \zeta}
    \left(P_{k+1}^{(\ell)},P_{i}^{(\ell)}\right)_{\omega^{(\ell)}}
    =&\left(\frac{\partial }{\partial \zeta}P_{k+1}^{(\ell)},P_{i}^{(\ell)}\right)_{\omega^{(\ell)}}
    +\left(P_{k+1}^{(\ell)},\frac{\partial }{\partial \zeta}P_{i}^{(\ell)}\right)_{\omega^{(\ell)}}\\
    &+\left(G(\zeta)-\zeta^{-1}\right)\left(P_{k+1}^{(\ell)},P_{i}^{(\ell)}\right)_{\omega^{(\ell)}}-\left(xP_{i}^{(\ell)},P_{k+1}^{(\ell)}\right)_{\omega^{(\ell)}}\\
    =&\left(\frac{\partial }{\partial \zeta}P_{k+1}^{(\ell)},P_{i}^{(\ell)}\right)_{\omega^{(\ell)}}
    =0,\quad i\leq k-1.
\end{align*}
Thus one has
\begin{align*}
    \left(\frac{\partial }{\partial \zeta}P_{k+1}^{(\ell)},P_{k+1}^{(\ell)}\right)_{\omega^{(\ell)}}
    &=-\frac{1}{2}\left(G(\zeta)-\zeta^{-1}-b_{k+1}^{(\ell)}\right),\\
    \left(\frac{\partial }{\partial \zeta}P_{k+1}^{(\ell)},P_{k}^{(\ell)}\right)_{\omega^{(\ell)}}
    &= a_{k}^{(\ell)},\quad
    \left(\frac{\partial }{\partial \zeta}P_{k+1}^{(\ell)},P_{i}^{(\ell)}\right)_{\omega^{(\ell)}}
    =0,\quad i\leq k-1.
\end{align*}
Because  $\frac{\partial P_{k+1}^{(\ell)}}{\partial \zeta}$ is
a polynomial and its degree is not larger than $ k+1$,
using \eqref{eq:P01orth-3} gives  \eqref{eq:partialPnzeta}.
The proof is completed.
\qed\end{proof}

\begin{thm}
\label{thm:derivex}
For $\ell\in\mathbb{N}^+$,
the first-order derivatives of  the polynomials $\{P_{k}^{(\ell)} (x;\zeta)\}$
with respect to  the independent variable $x$ satisfy
 \begin{align}
     &\frac{\partial P_{k+1}^{(\ell-1)}}{\partial x}=\frac{2\ell-1}{2\ell+1}\frac{k+1}{\tilde{p}_{k}^{(\ell)}}P_{k}^{(\ell)}+
     \zeta r_{k}^{(\ell)}P_{k-1}^{(\ell)}, \label{eq:derivePn0x}\\
     &\frac{2\ell-1}{2\ell+1}(x^2-1)\frac{\partial P_{k}^{(\ell)}}{\partial x}+(2\ell-1)xP_{k}^{(\ell)}=(k+2\ell+1)\tilde{p}_{k}^{(\ell)}P_{k+1}^{(\ell-1)}+
     \zeta p_{k}^{(\ell)} P_{k}^{(\ell-1)}.
     \label{eq:derivePn1x}
\end{align}
\end{thm}
\begin{proof}
Similar to the proof of Theorem \ref{thm:derivezeta}, one has
\[
{\frac{\partial }{\partial x}\omega^{(\ell)}(x;\zeta)
=(2\ell-1)x\omega^{(\ell-1)}(x;\zeta)-\zeta\omega^{(\ell)}(x;\zeta).}
\]
Because the degrees of polynomials
$\frac{\partial P_{k+1}^{(\ell-1)}}{\partial x}$ and
$\frac{2\ell-1}{2\ell+1}(x^2-1)\frac{\partial P_{k}^{(\ell)}}{\partial x}+(2\ell-1)xP_{k}^{(\ell)}$
are not larger than  $ n$ and  $k+1$, respectively,
and
\[
\lim_{x\rightarrow+\infty}P_{i}^{(\ell-1)}(x;\zeta)P_{j}^{(\ell)}(x;\zeta)\omega^{(\ell)}(x;\zeta)=0,\quad \lim_{x\rightarrow1}P_{i}^{(\ell-1)}(x;\zeta)P_{j}^{(\ell)}(x;\zeta)\omega^{(\ell)}(x;\zeta)=0,\quad \forall i,j\in\mathbb{K},
\]
one can calculate the expansion coefficients in  \eqref{eq:P01orth-3} as follows
\begin{align*}
&\left(\frac{\partial }{\partial x}P_{k+1}^{(\ell-1)},P_{k}^{(\ell)}\right)_{\omega^{(\ell)}}=
\left((k+1)c_{k+1}^{(\ell-1)}x^{k},P_{k}^{(\ell)}\right)_{\omega^{(\ell)}}=
\frac{2\ell-1}{2\ell+1}\frac{k+1}{\tilde{p}_{k}^{(\ell)}}\left(P_{k}^{(\ell)},P_{k}^{(\ell)}\right)_{\omega^{(\ell)}}
=\frac{2\ell-1}{2\ell+1}\frac{k+1}{\tilde{p}_{k}^{(\ell)}},\\
&\left(\frac{\partial }{\partial x}P_{k+1}^{(\ell-1)},P_{k-1}^{(\ell)}\right)_{\omega^{(\ell)}}=
\int_{1}^{+\infty}\frac{\partial }{\partial x}\left(P_{k+1}^{(\ell-1)}P_{k-1}^{(\ell)}\omega^{(\ell)}\right)dx
-\frac{2\ell-1}{2\ell+1}\left(P_{k+1}^{(\ell-1)},(x^2-1)\frac{\partial }{\partial x}P_{k-1}^{(\ell)}\right)_{\omega^{(\ell-1)}}\\
&-(2\ell-1)\left(P_{k+1}^{(\ell-1)},xP_{k-1}^{(\ell)}\right)_{\omega^{(\ell-1)}}+\zeta\frac{2\ell-1}{2\ell+1}\left(P_{k+1}^{(\ell-1)},(x^2-1)P_{k-1}^{(\ell)}\right)_{\omega^{(\ell-1)}}
=\zeta r_{k}^{(\ell)},\\
&\left(\frac{\partial }{\partial x}P_{k+1}^{(\ell-1)},P_{i}^{(\ell)}\right)_{\omega^{(\ell)}}=
\int_{1}^{+\infty}\frac{\partial }{\partial x}\left(P_{k+1}^{(\ell-1)}P_{i}^{(\ell)}\omega^{(\ell)}\right)dx
-\frac{2\ell-1}{2\ell+1}\left(P_{k+1}^{(\ell-1)},(x^2-1)\frac{\partial }{\partial x}P_{i}^{(\ell)}\right)_{\omega^{(\ell-1)}}\\
&-(2\ell-1)\left(P_{k+1}^{(\ell-1)},xP_{i}^{(\ell)}\right)_{\omega^{(\ell-1)}}
+\zeta\frac{2\ell-1}{2\ell+1}\left(P_{k+1}^{(\ell-1)},(x^2-1)P_{i}^{(\ell)}\right)_{\omega^{(\ell-1)}}
=0,\ i\leq k-2.
\end{align*}
and
%
\begin{align*}
&\left(\frac{2\ell-1}{2\ell+1}(x^2-1)\frac{\partial }{\partial x}P_{k}^{(\ell)}+(2\ell-1)xP_{k}^{(\ell)},P_{k+1}^{(\ell-1)}\right)_{\omega^{(\ell-1)}}=
\frac{2\ell-1}{2\ell+1}\left((k+2\ell+1)c_{k}^{(\ell)}x^{k+1},P_{k+1}^{(\ell-1)}\right)_{\omega^{(\ell-1)}}
\\
&=(k+2\ell+1)\tilde{p}_{k}^{(\ell)}\left(P_{k+1}^{(\ell-1)},P_{k+1}^{(\ell-1)}\right)_{\omega^{(\ell-1)}}
=(k+2\ell+1)\tilde{p}_{k}^{(\ell)},
\\
&\left(\frac{2\ell-1}{2\ell+1}(x^2-1)\frac{\partial }{\partial x}P_{k}^{(\ell)}+(2\ell-1)xP_{k}^{(\ell)},P_{k}^{(\ell-1)}\right)_{\omega^{(\ell-1)}}=
\int_{1}^{+\infty}\frac{\partial }{\partial x}\left(P_{k}^{(\ell)}P_{k}^{(\ell-1)}\omega^{(\ell)}\right)dx\\
&-\left(P_{k}^{(\ell)},\frac{\partial }{\partial x}P_{k}^{(\ell-1)}\right)_{\omega^{(\ell)}}
+\zeta\left(P_{k}^{(\ell)},P_{k}^{(\ell-1)}\right)_{\omega^{(\ell)}}
=\zeta\left(P_{k}^{(\ell)},P_{k}^{(\ell-1)}\right)_{\omega^{(\ell)}}=\zeta p_{k}^{(\ell)},
\\
&\left(\frac{2\ell-1}{2\ell+1}(x^2-1)\frac{\partial }{\partial x}P_{k}^{(\ell)}+(2\ell-1)xP_{k}^{(\ell)},P_{i}^{(\ell-1)}\right)_{\omega^{(\ell-1)}}=
\int_{1}^{+\infty}\frac{\partial }{\partial x}\left(P_{k}^{(\ell)}P_{i}^{(\ell-1)}\omega^{(\ell)}\right)dx\\
&-\left(P_{k}^{(\ell)},\frac{\partial }{\partial x}P_{i}^{(\ell-1)}\right)_{\omega^{(\ell)}}
+\zeta\left(P_{k}^{(\ell)},P_{i}^{(\ell-1)}\right)_{\omega^{(\ell)}}=0,\ i\leq k-1.
\end{align*}
The proof is completed.
\qed\end{proof}

\subsection{Zeros}
\label{subsec:zeros}
Using the separation theorem of zeros of
general orthogonal polynomials 
 \cite{SP:2011} gives the following conclusion on
the  orthogonal polynomials $\{P_{k}^{(\ell)}(x;\zeta)\}$.

\begin{thm}
\label{thm:zerointerself}
For $\ell\in\mathbb{N}$, the zeros $\{x_{i,k}^{(\ell)}\}_{i=1}^{k}$ of $P_{k}^{(\ell)}(x;\zeta)$ and
$\{x_{i,k+1}^{(\ell)}\}_{i=1}^{k+1}$ of $P_{k+1}^{(\ell)}(x;\zeta)$ satisfy
the  separation property
\[
1<x_{1,k+1}^{(\ell)}<x_{1,k}^{(\ell)}<x_{2,k+1}^{(\ell)}<\cdots<x_{k,k}^{(\ell)}<x_{k+1,k+1}^{(\ell)}.
\]
\end{thm}

There is still another important separation property for the
zeros of  the orthogonal polynomials $\{P_{k}^{(\ell)}(x;\zeta), \ell\in\mathbb{N}\}$.

\begin{thm}
\label{thm:zerointer01}
The $k$ zeros $\{x_{i,k}^{(\ell)}\}_{i=1}^{k}$ of $P_{k}^{(\ell)}$ and
$(k+1)$ zeros of $\{x_{i,k+1}^{(\ell-1)}\}_{i=1}^{k+1}$ of $P_{k+1}^{(\ell-1)}$ satisfy
\[
1<x_{1,k+1}^{(\ell-1)}<x_{1,k}^{(\ell)}<x_{2,k+1}^{(\ell-1)}<\cdots<x_{k,k}^{(\ell)}<x_{k+1,k+1}^{(\ell-1)}.
\]
\end{thm}
\begin{proof}
Substituting $\{x_{i,k+1}^{(\ell-1)}\}_{i=1}^{k+1}$ into \eqref{eq:recP01x} gives
\[
   \frac{2\ell-1}{2\ell+1}\left((x_{i,k+1}^{(\ell-1)})^2-1\right)P_{k}^{(\ell)}(x_{i,k+1}^{(\ell-1)};\zeta)=\tilde{r}_{k}^{(\ell)}P_{k}^{(\ell-1)}(x_{i,k+1}^{(\ell-1)};\zeta).
\]
which implies that $\tilde{r}_{k}^{(\ell)}\neq0$. In fact, if assuming   $\tilde{r}_{k}^{(\ell)}=0$,
then the above identity and  the fact that $(x_{i,k+1}^{(\ell-1)})^2-1>0$
imply $P_{k}^{(\ell)}(x_{i,k+1}^{(\ell-1)};\zeta)=0$,
which contradicts with $P_{k}^{(\ell)}$ being a polynomial of degree $k$.

Using Theorem \ref{thm:zerointerself} gives
\[
{\rm sign}\left(P_{k}^{(\ell)}(x_{i,k+1}^{(\ell-1)};\zeta)P_{k}^{(\ell)}(x_{i+1,k+1}^{(\ell-1)};\zeta)\right)={\rm sign}\left(P_{k}^{(\ell-1)}(x_{i,k+1}^{(\ell-1)};\zeta)P_{k}^{(\ell-1)}(x_{i+1,k+1}^{(\ell-1)};\zeta)\right)<0.
\]
Thus the number of zero of  the polynomial $P_{k}^{(\ell)}$ is not less than one in each subinterval
$(x_{i,k+1}^{(\ell-1)},x_{i+1,k+1}^{(\ell-1)})$.
The proof is completed.
\qed\end{proof}

According to Theorems \ref{thm:zerointerself} and \ref{thm:zerointer01}, one can
know  the sign of the coefficients of the recurrence relations in Theorem \ref{thm:huxiang}.
\begin{corollary}
\label{col:sign}
All quantities $p_k, q_k, r_k$ in \eqref{eq:pqrn} and
$\tilde{p}_k, \tilde{q}_k, \tilde{r}_k$ in
\eqref{eq:pqrnx} are positive.
\end{corollary}

\begin{proof}
It is obvious that
\[
p_{k}^{(\ell)}=\frac{c_{k}^{(\ell-1)}}{c_{k}^{(\ell)}}>0,\quad r_{k}^{(\ell)}=\frac{2\ell-1}{2\ell+1}\frac{c_{k-1}^{(\ell)}}{c_{k+1}^{(\ell-1)}}>0,\quad
\tilde{p}_{k}^{(\ell)}=\frac{2\ell-1}{2\ell+1}\frac{c_{k}^{(\ell)}}{c_{k+1}^{(\ell-1)}}>0.
\]
Using  Theorems \ref{thm:rec} and \ref{thm:zerointer01} gives
{\begin{align*}
\tilde{q}_{k}^{(\ell)}&=\sum_{i=1}^{k+1}x_{i,k+1}^{(\ell-1)}-\sum_{i=1}^{k}x_{i,k}^{(\ell)}=\sum_{i=1}^{k}\left(x_{i+1,k+1}^{(\ell-1)}-x_{i,k}^{(\ell)}\right)
+x_{1,k+1}^{(\ell-1)}>0,\\
q_{k}^{(\ell)}&=\tilde{p}_{k}^{(\ell)}\left(\sum_{i=1}^{k+2}x_{i,k+2}^{(\ell-1)}-\sum_{i=1}^{k}x_{i,k}^{(\ell)}\right)=\tilde{p}_{k}^{(\ell)}\left(b_{k+1}^{(\ell-1)}+\tilde{q}_{k}^{(\ell)}\right)>0,
\end{align*}}
which imply $q_{k}^{(\ell)}>0$ and $\tilde{q}_{k}^{(\ell)}>0$.

Comparing the coefficients of the term of order $n$
at two sides of \eqref{eq:recP01x} gives
\begin{align*}
\tilde{r}_{k}=&\frac{2\ell-1}{2\ell+1}(p_{k}^{(\ell)})^{-1}
\left(\sum_{i=1}^{k}\sum_{j=i+1}^{k}x_{i,k}^{(\ell)}x_{j,k}^{(\ell)}-1-
    \sum_{i=1}^{k+1}\sum_{j=i+1}^{k+1}x_{i,k+1}^{(\ell-1)}x_{j,k+1}^{(\ell-1)}+
    \left(\sum_{i=1}^{k+1}x_{i,k+1}^{(\ell-1)}-\sum_{i=1}^{k}x_{i,k}^{(\ell)}\right)\sum_{i=1}^{k+1}x_{i,k+1}^{(\ell-1)}\right)\\
=& \frac{2\ell-1}{2\ell+1}(p_{k}^{(\ell)})^{-1}
\left(\sum_{i=1}^{k}\sum_{j=i+1}^{k}x_{i,k}^{(\ell)}x_{j,k}^{(\ell)}+\sum_{i=1}^{k+1}\sum_{j=i}^{k+1}x_{i,k+1}^{(\ell-1)}x_{j,k+1}^{(\ell-1)}
    -\sum_{i=1}^{k}x_{i,k}^{(\ell)}\sum_{i=1}^{k+1}x_{i,k+1}^{(\ell-1)}-1\right)\\
=&\frac{2\ell-1}{2\ell+1}(p_{k}^{(\ell)})^{-1}
\left(\sum_{i=1}^{k}x_{i+1,k+1}^{(\ell-1)}(x_{i+1,k+1}^{(\ell-1)}-x_{i,k}^{(\ell)})+(x_{1,k+1}^{(\ell-1)})^{2}-1\right.\\   &\left.+\sum_{i=0}^{k}\sum_{j=i+1}^{k}(x_{i+1,k+1}^{(\ell-1)}-x_{i,k}^{(\ell)})(x_{j+1,k+1}^{(\ell-1)}-x_{j,k}^{(\ell)})\right),
\end{align*}
where $x_{0,k}^{(\ell)}=0$.

Using Theorem \ref{thm:zerointer01} gives $\tilde{r}_{k}^{(\ell)}>0$.
The proof is completed.
\qed\end{proof}

Using  Corollary \ref{col:sign}, $\tilde{r}_{k}^{(\ell)}={p_{k}^{(\ell)}(1-\frac{2\ell+1}{2\ell-1}(\tilde{p}_{k}^{(\ell)})^{2})}$, and $\tilde{p}_{k}^{(\ell)}=\frac{2\ell-1}{2\ell+1}(c_{k+1}^{(\ell-1)})^{-1}c_{k}^{(\ell)}$
can give the following corollary.
\begin{corollary}
\label{corollary3.2}
	The leading coefficient of  $P_{k+1}^{(\ell-1)}$ is
	$\sqrt{\frac{2\ell+1}{2\ell-1}}$ times larger than that of $P_{k}^{(\ell)}$, i.e. $c_{k+1}^{(\ell-1)}>\sqrt{\frac{2\ell+1}{2\ell-1}}c_{k}^{(\ell)}$.
\end{corollary}

According to Theorems \ref{thm:derivezeta} and \ref{thm:zerointerself},
one has   the following further conclusion.
\begin{corollary}
\label{col:varx}
The zeros $\{x_{i,k}^{(\ell)}\}_{i=1}^{k}$ of $P_{k}^{(\ell)}$ strictly decrease   with respect to
$\zeta$, i.e.
\[
    \frac{\partial x_{i,k}^{(\ell)}}{\partial \zeta}<0.
\]
\end{corollary}
\begin{proof}
Taking partial derivative of $P_{k}^{(\ell)}(x_{i,k}^{(\ell)};\zeta)$ with respect to $\zeta$
and using  Theorem \ref{thm:derivezeta}
gives
\[
    \frac{\partial x_{i,k}^{(\ell)}}{\partial \zeta}=-\left(\frac{\partial P_{k}^{(\ell)}}{\partial x}(x_{i,k}^{(\ell)};\zeta)\right)^{-1}
    \left(\frac{\partial P_{k}^{(\ell)}}{\partial \zeta}(x_{i,k}^{(\ell)};\zeta)\right)=-a_{k-1}^{(\ell)}\left(\frac{\partial P_{k}^{(\ell)}}{\partial x}(x_{i,k}^{(\ell)};\zeta)\right)^{-1}P_{k-1}^{(\ell)}(x_{i,k}^{(\ell)};\zeta).
\]
Due to Theorem \ref{thm:zerointerself}, one has
\[
{\rm sign}(P_{k-1}^{(\ell)}(x_{i,k}^{(\ell)};\zeta))=(-1)^{k+i}={\rm sign}\left(\frac{\partial P_{k}^{(\ell)}}{\partial x}(x_{i,k}^{(\ell)};\zeta)\right).
\]
Combining them  completes the proof.
\qed\end{proof}

\section{{Some properties} of real spherical harmonics}
\label{sec:SH}
The real spherical harmonics are known as tesseral spherical harmonics \cite{sh:1927}. The harmonics with $m > 0$ are said to be of cosine type, and those with $m < 0$ are of sine type. The reason for that can be seen by writing the spherical harmonic functions in terms of the associated Legendre polynomials
\begin{equation}
\label{eq:sh}
Y_{\ell,m}(y,\phi)=
  \begin{cases}
    \sqrt{2}\check{P}_{\ell}^{|m|}(y)\sin(|m|\phi), & m<0, \\
    \check{P}_{\ell}^{m}(y), & m=0, \\
    \sqrt{2}\check{P}_{\ell}^{m}(y)\cos(m\phi), & m>0,
  \end{cases}
\end{equation}
where
\[
\check{P}_{\ell}^{m}(y)=\sqrt{\frac{(\ell-m)!}{(\ell+m)!}}\hat{P}_{\ell}^{m}(y),\quad |m|\leq\ell,
\ \ \ell\in\mathbb{N},
\]
and
$\{\hat{P}_{\ell}^{m}(y),|m|\leq\ell,\ell\in\mathbb{N}\}$ are the associated Legendre polynomials defined by
\begin{equation}
\label{eq:alp}
\hat{P}_{\ell}^{m}(y)=(-1)^{m}(1-y^2)^{\frac{m}{2}}\frac{1}{2^{\ell}\ell!}\frac{d^{\ell+m}}{dy^{\ell+m}}(y^2-1)^{\ell}.
\end{equation}

Moreover, it holds that
$\check{P}_{\ell}^{m}(y)=0$ for $|m|>\ell$; $\check{P}_{\ell}^{-m}(y)=(-1)^{m}\check{P}_{\ell}^{m}(y)$, and $\check{P}_{-1}^{m}(y)=0$.

Thus   the following orthogonality properties hold
\begin{equation}
\label{eq:orthsh}
\int_{-1}^{1}\int_{0}^{2\pi}Y_{\ell,m}Y_{\ell',m'} d\phi dy =\frac{4\pi}{2\ell+1}\delta_{\ell,\ell'}\delta_{m,m'}.
\end{equation}

\subsection{Recurrence relations and Partial derivatives}
\label{subsec:recurrsh}

\begin{lemma}[\cite{handbook:331-339 1974}]
\label{thm:recderalp}
The associated Legendre polynomials $\{\hat{P}_{\ell}^{m}(y)\}$ have the following properties:
\begin{align}
&y\hat{P}_{\ell}^{m}=\frac{1}{2\ell+1}\left((\ell-m+1)\hat{P}_{\ell+1}^{m}+(\ell+m)\hat{P}_{\ell-1}^{m}\right),\label{eq:rec1}\\
&y(1-y^2)^{-\frac{1}{2}}\hat{P}_{\ell}^{m}=-\frac{1}{2m}\left(\hat{P}_{\ell}^{m+1}+(\ell+m)(\ell-m+1)\hat{P}_{\ell}^{m-1}\right),\label{eq:rec2}\\
&(1-y^2)^{-\frac{1}{2}}\hat{P}_{\ell}^{m}=-\frac{1}{2m}\left(\hat{P}_{\ell-1}^{m+1}+(\ell+m-1)(\ell+m)\hat{P}_{\ell-1}^{m-1}\right),\label{eq:rec3}\\
&(1-y^2)^{\frac{1}{2}}\hat{P}_{\ell}^{m}=\frac{1}{2\ell+1}\left((\ell-m+1)(\ell-m+2)\hat{P}_{\ell+1}^{m-1}-(\ell+m-1)(\ell+m)\hat{P}_{\ell-1}^{m-1}\right),\label{eq:rec4}\\
&(1-y^2)^{\frac{1}{2}}\hat{P}_{\ell}^{m}=\frac{1}{2\ell+1}\left(-\hat{P}_{\ell+1}^{m+1}+\hat{P}_{\ell-1}^{m+1}\right),\label{eq:rec5}\\
&(1-y^2)^{\frac{1}{2}}\hat{P}_{\ell}^{m+1}=(\ell-m+1)\hat{P}_{\ell+1}^{m}-(\ell+m+1)y\hat{P}_{\ell}^{m},\label{eq:rec6}\\
&(y^2-1)\frac{d}{dy}\hat{P}_{\ell}^{m}=\ell y\hat{P}_{\ell}^{m}-(\ell+m)\hat{P}_{\ell-1}^{m},\label{eq:der1}\\
&(y^2-1)\frac{d}{dy}\hat{P}_{\ell}^{m}=\sqrt{1-y^2}\hat{P}_{\ell}^{m+1}+my\hat{P}_{\ell}^{m}.\label{eq:der2}
\end{align}
\end{lemma}
Using Lemma \ref{thm:recderalp}
can give the following recurrence relations and partial derivatives of
the real spherical harmonics.

\begin{thm}
\label{lem:recder}
\begin{itemize}
  \item[(1)] The real spherical harmonics satisfy the recurrence relations
\begin{align}
I_{0}Y_{\ell,m}=&\frac{1}{2\ell+1}\left(h_{\ell+1,m}Y_{\ell+1,m}+h_{\ell,m}Y_{\ell-1,m}\right),\label{eq:recc0}\\
\nonumber I_{1}Y_{\ell,m}=&\frac{1}{2(2\ell+1)}
\left(\check{s}_{m}(\tilde{h}_{\ell+1,-m}Y_{\ell+1,m-1}-\tilde{h}_{\ell-1,m}Y_{\ell-1,m-1})\right.\\
&-\left.\hat{s}_{m}(\tilde{h}_{\ell+1,m}Y_{\ell+1,m+1}-\tilde{h}_{\ell-1,-m}Y_{\ell-1,m+1})\right),\label{eq:recc1}\\
\nonumber I_{2}Y_{\ell,m}=&\frac{1}{2(2\ell+1)}
\left(s_{m}(\tilde{h}_{\ell-1,-m}Y_{\ell-1,-m-1}-\tilde{h}_{\ell+1,m}Y_{\ell+1,-m-1})\right.\\
&+\left.\tilde{s}_{m}(\tilde{h}_{\ell-1,m}Y_{\ell-1,-m+1}-\tilde{h}_{\ell+1,-m}Y_{\ell+1,-m+1})\right), \label{eq:recc2}
\end{align}
where
$I_{0}(y,\phi):=y$,
$I_{1}(y,\phi):=\sqrt{1-y^2}\cos\phi$,
$I_{2}(y,\phi):=\sqrt{1-y^2}\sin\phi$,
$
h_{\ell,m}:=\sqrt{(\ell+m)(\ell-m)}$,
$\tilde{h}_{\ell-1,m}:=\sqrt{(\ell+m-1)(\ell+m)}$,
and
\begin{align*}
&s_{m:}={\rm {sign}(m)}\sqrt{\delta_{m,-1}+\delta_{m,0}+1},\quad \tilde{s}_{m}:={\rm {sign}(m)}(1-\delta_{m,0})(1-\delta_{m,1}),\\
&\hat{s}_{m}:={\rm {sign}(m)}(1-\delta_{{m},-1})\sqrt{\delta_{{m},0}+1},\quad \check{s}_{m}:={\rm {sign}(m)}(1-\delta_{m,0})\sqrt{\delta_{{m},1}+1},
\end{align*}
with
\[
{\rm sign}(m)=
               \begin{cases}
                 1, & m\geq0, \\
                 -1, & m<0.
               \end{cases}
\]

  \item[(2)]  The partial derivatives of  real spherical harmonics
 satisfy
\begin{align}
\tilde{I}_{0}\frac{d}{dy}Y_{\ell,m}-\ell I_{0}Y_{\ell,m}&=-h_{\ell,m}Y_{\ell-1,m},\label{eq:reccc0}\\
\tilde{I}_{0}\frac{d}{dy}Y_{\ell,m}+(\ell+1)I_{0}Y_{\ell,m}&=h_{\ell+1,m}Y_{\ell+1,m},\label{eq:reccc01}
\end{align}
\begin{align}
\tilde{I}_{1}\frac{d}{dy}& Y_{\ell,m}-\ell I_{1}Y_{\ell,m}-m\hat{I}_{1}Y_{\ell,-m}
=\frac{1}{2}\left(\check{s}_{m}\tilde{h}_{\ell-1,m}Y_{\ell-1,m-1}\right.
-\left.\hat{s}_{m}\tilde{h}_{\ell-1,-m}Y_{\ell-1,m+1}\right),
\label{eq:reccc1}\\
\tilde{I}_{2}\frac{d}{dy}& Y_{\ell,m}-\ell I_{2}Y_{\ell,m}-m\hat{I}_{2}Y_{\ell,-m}
 =-\frac{1}{2}\left(\tilde{s}_{m}\tilde{h}_{\ell-1,m}Y_{\ell-1,-m+1}\right.
+\left.s_{m}\tilde{h}_{\ell-1,-m}Y_{\ell-1,-m-1}\right),\label{eq:reccc2}\\
\tilde{I}_{1}\frac{d}{dy} & Y_{\ell,m}+(\ell+1)I_{1}Y_{\ell,m}-m\hat{I}_{1}Y_{\ell,-m}
=\frac{1}{2}\left(\check{s}_{m}\tilde{h}_{\ell+1,-m}Y_{\ell+1,m-1}\right.
-\left.\hat{s}_{m}\tilde{h}_{\ell+1,m}Y_{\ell+1,m+1}\right),
\label{eq:reccc11}\\
\tilde{I}_{2}\frac{d}{dy}& Y_{\ell,m}+(\ell+1)I_{2}Y_{\ell,m}-m\hat{I}_{2}Y_{\ell,-m}
=-\frac{1}{2}\left(\tilde{s}_{m}\tilde{h}_{\ell+1,-m}Y_{\ell+1,-m+1}\right.
+\left.s_{m}\tilde{h}_{\ell+1,m}Y_{\ell+1,-m-1}\right),\label{eq:reccc21}
\end{align}
\begin{align}
I_{1}\frac{d}{dy}Y_{\ell,m}-m\hat{I}_{3}Y_{\ell,-m}
=&\frac{1}{2}\left(\check{s}_{m}\hat{h}_{\ell,m}Y_{\ell,m-1}
-\hat{s}_{m}\hat{h}_{\ell,-m}Y_{\ell,m+1}\right),\label{eq:reccc3}\\
I_{2}\frac{d}{dy}Y_{\ell,m}-m\hat{I}_{4}Y_{\ell,-m}
=&-\frac{1}{2}\left(\tilde{s}_{m}\hat{h}_{\ell,m}Y_{\ell,-m+1}
+s_{m}\hat{h}_{\ell,-m}Y_{\ell,-m-1}\right).\label{eq:reccc4}
\end{align}
where
\begin{align*}
&\tilde{I}_{0}(y,\phi):=y^2-1, \ \
\tilde{I}_{1}(y,\phi):=y\sqrt{1-y^2}\cos\phi,\\
&\tilde{I}_{2}(y,\phi):=y\sqrt{1-y^2}\sin\phi,\ \
\hat{I}_{1}(y,\phi):=(1-y^2)^{-\frac{1}{2}}\sin\phi,\\
&\hat{I}_{2}(y,\phi):=-(1-y^2)^{-\frac{1}{2}}\cos\phi,
\ \
\hat{I}_{3}(y,\phi):=y(1-y^2)^{-\frac{1}{2}}\sin\phi, \\
&\hat{I}_{4}(y,\phi):=-y(1-y^2)^{-\frac{1}{2}}\cos\phi,
\ \
 \hat{h}_{\ell,m}:=\sqrt{(\ell+m)(\ell-m+1)}.
\end{align*}

\end{itemize}

\end{thm}
\begin{proof}
Only the case of $m>1$ is taken as an example because other cases can be obtained similarly.

(1) Thanks to \eqref{eq:rec1}, the identity \eqref{eq:recc0} is obvious.
Due to \eqref{eq:rec4} and \eqref{eq:rec5},  one has
\begin{align*}
I_{1}&\hat{P}_{\ell}^{m}\cos((m\phi)=\frac{1}{2}\sqrt{1-y^2}\hat{P}_{\ell}^{m}\cos((m+1)\phi)
+\frac{1}{2}\sqrt{1-y^2}\hat{P}_{\ell}^{m}\cos((m-1)\phi)\\
=&\frac{1}{2\ell+1}\left(-\hat{P}_{\ell+1}^{m+1}+\hat{P}_{\ell-1}^{m+1}\right)\cos((m+1)\phi)\\
+&\frac{1}{2\ell+1}\left((\ell-m+1)(\ell-m+2)\hat{P}_{\ell+1}^{m-1}-(\ell+m-1)(\ell+m)\hat{P}_{\ell-1}^{m-1}\right)\cos((m-1)\phi).
\end{align*}
Thus \eqref{eq:recc1} and \eqref{eq:recc2} can be gotten.

(2) Due to \eqref{eq:der1}, one has
\[
\tilde{I}_{0}\frac{d}{dy}\hat{P}_{\ell}^{m}-\ell I_{0}\hat{P}_{\ell}^{m}=-(\ell+m)\hat{P}_{\ell-1}^{m},
\]
which deduces \eqref{eq:reccc0}. Combining it with \eqref{eq:recc0} gives \eqref{eq:reccc01}.

(3) Due to \eqref{eq:rec2}, \eqref{eq:rec6} and \eqref{eq:der1}, one has
\begin{align*}
&\left(\tilde{I}_{1}\frac{d}{dy}\hat{P}_{\ell}^{m}-\ell I_{1}\hat{P}_{\ell}^{m}\right)\cos(m\phi)-m\hat{I}_{1}\hat{P}_{\ell}^{m}\sin(m\phi)\\
=&\frac{\sqrt{1-y^2}}{y^2-1}\left(\ell y^2\hat{P}_{\ell}^{m}-(\ell+m)y\hat{P}_{\ell-1}^{m}-\ell(y^2-1)\hat{P}_{\ell}^{m}\right)
\cos\phi\cos(m\phi)-m(1-y^2)^{-\frac{1}{2}}\hat{P}_{\ell}^{m}\sin\phi\sin(m\phi)\\
=&-(1-y^2)^{-\frac{1}{2}}\left(\left(\ell \hat{P}_{\ell}^{m}-(\ell+m)y\hat{P}_{\ell-1}^{m}\right)
\cos\phi\cos(m\phi)+m\hat{P}_{\ell}^{m}\sin\phi\sin(m\phi)\right)\\
=&-m(1-y^2)^{-\frac{1}{2}}\hat{P}_{\ell}^{m}\cos((m-1)\phi)-\hat{P}_{\ell-1}^{m+1}\cos\phi\cos(m\phi)\\
=&\frac{1}{2}\left(\hat{P}_{\ell-1}^{m+1}+(\ell+m-1)(\ell+m)\hat{P}_{\ell-1}^{m-1}\right)\cos((m-1)\phi)
-\frac{1}{2}\hat{P}_{\ell-1}^{m+1}\left(\cos((m-1)\phi)+\cos((m+1)\phi)\right)\\
=&\frac{1}{2}(\ell+m-1)(\ell+m)\hat{P}_{\ell-1}^{m-1}\cos((m-1)\phi)
-\frac{1}{2}\hat{P}_{\ell-1}^{m+1}\cos((m+1)\phi).
\end{align*}
Thus \eqref{eq:reccc1} and \eqref{eq:reccc2} can be gotten. Combining  them with \eqref{eq:recc1} and \eqref{eq:recc2} obtains \eqref{eq:reccc11} and \eqref{eq:reccc21}.

(4) Due to \eqref{eq:rec2} and \eqref{eq:der2}, one has
\begin{align*}
&I_{1}\frac{d}{dy}\hat{P}_{\ell}^{m}\cos(m\phi)-m\hat{I}_{3}\hat{P}_{\ell}^{m}\sin(m\phi)\\
=&\frac{\sqrt{1-y^2}}{y^2-1}\left(\sqrt{1-y^2}\hat{P}_{\ell}^{m+1}+my\hat{P}_{\ell}^{m}\right)\cos\phi\cos(m\phi)
-my(1-y^2)^{-\frac{1}{2}}\hat{P}_{\ell}^{m}\sin\phi\sin(m\phi)\\
=&-my(1-y^2)^{-\frac{1}{2}}\hat{P}_{\ell}^{m}\cos((m-1)\phi)-\hat{P}_{\ell}^{m+1}\cos\phi\cos(m\phi)\\
=&\frac{1}{2}\left(\hat{P}_{\ell}^{m+1}+(\ell+m)(\ell-m+1)\hat{P}_{\ell}^{m-1}\right)\cos((m-1)\phi)-\frac{1}{2}\hat{P}_{\ell}^{m+1}\left(\cos((m-1)\phi)+\cos((m+1)\phi)\right)\\
=&\frac{1}{2}(\ell+m)(\ell-m+1)\hat{P}_{\ell}^{m-1}\cos((m-1)\phi)-\frac{1}{2}\hat{P}_{\ell}^{m+1}\cos((m+1)\phi).
\end{align*}
Thus  \eqref{eq:reccc3} and \eqref{eq:reccc4} are gotten.
The proof is completed.
\qed\end{proof}
\label{subsec:zerosh}

\section{Moment method by operator projection} 
\label{sec:moment}
This section begins to extend the moment method by operator projection  \cite{MR:2014,Kuang:2017}
to the special relativistic Boltzmann equation \eqref{eq:Boltz} and derive its arbitrary order  hyperbolic moment model. For the sake of convenience,  units in which both the speed of light $c$  and  rest mass $m$ of particle are equal to one will be used hereafter.

Similar to \cite{GRADP:1974}, the momentum  $p^{\alpha}$ at a point is decomposed  as
\begin{equation}
\label{eq:pdec}
p^{\alpha}=U^{\alpha}E+\sqrt{E^2-1}l^{\alpha},
\end{equation}
where $l^{\alpha}$ is an unit spacelike vector orthogonal to $U^{\alpha}$, i.e.
\[
l^{\alpha}l_{\alpha}=-1,\quad l^{\alpha}U_{\alpha}=0.
\]
We introduce an orthogonal tetrad $n_{i}^{\alpha}$ ($i=1,2,3$) orthogonal to $U^{\alpha}$ so that
\[
U_{\alpha}n_{i}^{\alpha}=0,\quad g_{\alpha\beta}n_{i}^{\alpha}n_{j}^{\beta}=-\delta_{i,j}.
\]
Using the Lorentz transformation \cite{SR:1961} to the local rest frame where $(U^{\alpha})=(1,0,0,0)$, and taking $n_{i}^{\alpha}=\delta_{i,\alpha}$ ($i=1,2,3$)
 can obtain
\begin{equation}
\label{eq:ni}
n_{i}^{0}=U^{i},\quad n_{i}^{j}=((U^{0})^{2}-1)^{-1}U^{i}U^{j}(U^{0}-1)+\delta_{i,j},\quad i,j=1,2,3.
\end{equation}

Thus  $l^{\alpha}$ can be expressed as
\[
l^{\alpha}=\sin\xi\cos\phi n_{1}^{\alpha}+\sin\xi\sin\phi n_{2}^{\alpha}+\cos\xi n_{3}^{\alpha},
\]
where
\[
(\sin\xi\cos\phi,\sin\xi\sin\phi,\cos\xi):=-(E^2-1)^{-\frac{1}{2}}\left(n_{1}^{\alpha}p_{\alpha},n_{2}^{\alpha}p_{\alpha},n_{3}^{\alpha}p_{\alpha}\right),
\]
and $\xi\in[0,\pi]$, $\phi\in[0,2\pi)$.
It is easy to prove that it is well-defined, that is to say, the above vector is an unit vector.
If denoting $y:=\cos\xi$, then one has
\begin{equation}
\label{eq:yphi}
(\sqrt{1-y^2}\cos\phi,\sqrt{1-y^2}\sin\phi,y):=-(E^2-1)^{-\frac{1}{2}}\left(n_{1}^{\alpha}p_{\alpha},n_{2}^{\alpha}p_{\alpha},n_{3}^{\alpha}p_{\alpha}\right).
\end{equation}

\subsection{Weighted polynomial space}
\label{subsec:expand}
In order to use the moment method by the operator projection
to derive the hyperbolic moment model of the kinetic equation,
one should define weighted polynomial spaces and  norms
 as well as the projection operator.
 Thanks to the equilibrium distribution $f^{(0)} $ in \eqref{eq:equm1},
 the weight function is chosen as $g^{(0)}$, which will be
replaced with  the new notation $g^{(0)}_{[\vec{u},\theta]}$, considering
the dependence of $g^{(0)}$  on the  macroscopic fluid velocity $\vec u$
and $\theta=k_BT/mc^2=\zeta^{-1}$, that is
\begin{equation}\label{eq:equm1-zzzzz}
g^{(0)}_{[\vec{u},\theta]}=\frac{\zeta}{4\pi K_{2}(\zeta)}\exp\left(-\frac{E}{\theta}\right),\ E=U_{\alpha}p^{\alpha}.
\end{equation}

Associated with the weight function $g^{(0)}_{[\vec{u},\theta]}$,  our weighted
polynomial space is defined by
    \[
    \mathbb{H}^{g^{(0)}_{[\vec{u},\theta]}}:
    ={\rm span}\left\{p^{\mu_{1}}p^{\mu_{2}}\cdots p^{\mu_{\ell}}g^{(0)}_{[\vec{u},\theta]}:
     \ \ \mu_{i}=0,1,2,3 ,\quad  \ell\in\mathbb{N}\right\},
    \]
which is an infinite-dimensional linear space equipped with the inner product
\[
   <f,g>_{g^{(0)}_{[\vec{u},\theta]}}:=\int_{\mathbb{R}^{3}}\frac{1}{g^{(0)}_{[\vec{u},\theta]}}f(\vec{p})g(\vec{p})\frac{d^{3}\vec{p}}{p^{0}}, \quad f,g\in \mathbb{H}^{g^{(0)}_{[\vec{u},\theta]}}.
\]
Similarly, for a finite positive integer $M\in\mathbb{N}$,   a finite-dimensional  weighted
polynomial space   can  be defined by
\[
    \mathbb{H}^{g^{(0)}_{[\vec{u},\theta]}}_{M}:={\rm span}\left\{p^{\mu_{1}}p^{\mu_{2}}\cdots p^{\mu_{\ell}}g^{(0)}_{[\vec{u},\theta]}: \ \ \mu_{i}=0,1,2,3 ,\quad \ell=0,1,\cdots, M\right\},
    \]
    which is  a closed subspace of $\mathbb{H}^{g^{(0)}_{[\vec{u},\theta]}}$ obviously.

Unlike the  one dimensional case \cite{Kuang:2017}, the basis of $\mathbb{H}^{g^{(0)}_{[\vec{u},\theta]}}_{M}$ cannot be obtained easily.
People usually use the following weighted polynomials \cite{DM:2012,RK:1980}
\begin{equation}\label{nobasis}
g^{(0)}_{[\vec{u},\theta]}
P_{k}^{(\ell)}(E{;}\theta^{-1})p^{<\mu_{1}\cdots}p^{\mu_{\ell}>}
\end{equation}
to span the  spaces $\mathbb{H}^{g^{(0)}_{[\vec{u},\theta]}}$ and $\mathbb{H}^{g^{(0)}_{[\vec{u},\theta]}}_{M}$,
where
$
p^{<\mu_{1}\cdots}p^{\mu_{\ell}>}:=\Delta_{\nu_{1}\cdots\mu_{\ell}}^{\mu_{1}\cdots\mu_{\ell}}p^{\nu_{1}}\cdots p^{\nu_{\ell}}$ denotes the $\ell$th order irreducible tensor, here
\[
\Delta^{\alpha}_{\beta}:=g^{\alpha}_{\beta}-\frac{1}{c^2}U^{\alpha}U_{\beta},\quad
\Delta^{\alpha\beta}_{\mu\nu}:= \frac{1}{2}\left( \Delta_{\mu}^{\alpha}\Delta_{\nu}^{\beta}+ \Delta_{\mu}^{\beta}\Delta_{\nu}^{\alpha}-\frac{2}{3} \Delta_{\mu\nu}\Delta^{\alpha\beta} \right),
\]
and the higher order ones could be found in \cite{RK:1980} and are
symmetric and traceless.
Moreover, the irreducible tensors satisfy
\[
\int_{\mathbb{R}^3}  p^{<\mu_{1}\cdots}p^{\mu_{m}>}p_{<\nu_{1}\cdots}p_{\nu_{k}>}F(E)\frac{d^3\vec p}{p^{0}}
= \frac{(-1)^{\ell}m!\delta_{mk}}{(2m+1)!!} \Delta^{\mu_{1}\cdots\mu_{m}}_{\nu_{1}\cdots\nu_{m}} \int F(E)(E^2-1)^{\ell}\frac{d^3\vec p}{p^{0}},
\]
for any function $F(E)$ of $E$.
Unfortunately, the weighted polynomials in \eqref{nobasis} are linearly dependent. It is obvious for  the case of $\ell=1$, so they cannot be the basis of $\mathbb{H}^{g^{(0)}_{[\vec{u},\theta]}}$ or $\mathbb{H}^{g^{(0)}_{[\vec{u},\theta]}}_{M}$.
In the one dimensional case, one can simply delete $p^{<0>}$ to obtain the basis, because the irreducible tensors are zeros for $\ell\geq2$.
 But in the three dimensional case, it is complicate.
This work uses the real  spherical harmonics to replace the irreducible tensors
and {presents} the basis of $\mathbb{H}^{g^{(0)}_{[\vec{u},\theta]}}$ or $\mathbb{H}^{g^{(0)}_{[\vec{u},\theta]}}_{M}$.

Thanks to Theorem \ref{lem:admissible1},
for all physically admissible $\vec{u}$ and $\theta$ satisfying $|\vec{u}|<1$ and $\theta>0$,
introduce two  notations
\begin{align}
\label{eq:basis}
\nonumber\vec{\mathcal{P}}_{\infty}[\vec{u},\theta]:=&(\tilde{P}_{0,0}^{(0)}[\vec{u},\theta],\cdots,\tilde{P}_{M,0}^{(0)}[\vec{u},\theta],
\tilde{P}_{M-1,-1}^{(1)}[\vec{u},\theta],
\tilde{P}_{M-1,0}^{(1)}[\vec{u},\theta],\\
&\tilde{P}_{M-1,1}^{(1)}[\vec{u},\theta],\cdots,\tilde{P}_{0,-M}^{(M)},\cdots,\tilde{P}_{0,M}^{(M)},\cdots)^{T},
\\
\label{eq:huap}
\nonumber\vec{\mathcal{P}}_{M}[\vec{u},\theta]:=&(\tilde{P}_{0,0}^{(0)}[\vec{u},\theta],\tilde{P}_{1,0}^{(0)}[\vec{u},\theta],\tilde{P}_{0,-1}^{(1)}[\vec{u},\theta],
\tilde{P}_{0,0}^{(1)}[\vec{u},\theta],\tilde{P}_{0,1}^{(1)}[\vec{u},\theta],
\cdots,\tilde{P}_{M,0}^{(0)}[\vec{u},\theta],\\
&\tilde{P}_{M-1,-1}^{(1)}[\vec{u},\theta],
\tilde{P}_{M-1,0}^{(1)}[\vec{u},\theta],\tilde{P}_{M-1,1}^{(1)}[\vec{u},\theta],\cdots,\tilde{P}_{0,-M}^{(M)},\cdots,\tilde{P}_{0,M}^{(M)} )^{T},
\end{align}
 where $\tilde{P}_{k,m}^{(\ell)}[\vec{u},\theta]:=g^{(0)}_{[\vec{u},\theta]}{P_{k}^{(\ell)}(E{;}
 \theta^{-1})\tilde{Y}_{\ell,m}}$,
$\tilde{Y}_{\ell,m}:=(E^2-1)^{\frac{\ell}{2}}{Y_{\ell,m}(y,\phi)}$, $|m|\leq\ell$, $\ell\in\mathbb{N}$,
and the length of the vector \eqref{eq:huap} is equal to
\[
\sum_{\ell=0}^{M}(2\ell+1)(M+1-\ell)=(M+1)\left(\frac{M(2M+1)}{6}+1\right)=:N_{M}.
\]

\begin{lemma}
\label{lem:basis}
The set of all components of $\vec{\mathcal{P}}_{\infty}[\vec{u},\theta]$ (resp. $\vec{\mathcal{P}}_{M}[\vec{u},\theta]$)
form a standard orthogonal basis of $\mathbb{H}^{g^{(0)}_{[\vec{u},\theta]}}$
 (resp. $\mathbb{H}^{g^{(0)}_{[\vec{u},\theta]}}_{M}$).
\end{lemma}
\begin{proof}
(i)  Using \eqref{eq:sh} and \eqref{eq:alp} gives
\begin{align*}
\tilde{Y}_{\ell,m}&=(E^2-1)^{\frac{\ell}{2}}\kappa_{\ell,m}^{1}(1-y^2)^{\frac{m}{2}}
\frac{d^{m}}{dy^{m}}\hat{P}_{\ell}(y)
\begin{cases}
\sin(|m|\phi), & m<0, \\
\cos({m}\phi), & m\geq0,
\end{cases}
\\
=&(E^2-1)^{\frac{\ell}{2}}(1-y^2)^{\frac{m}{2}}\sum_{k=0}^{m}\kappa_{\ell,m,k}^{2}\sin^{k}\phi\cos^{m-k}\phi\frac{d^{\ell+m}}{dy^{\ell+m}}(y^2-1)^{\ell}\\
=&(E^2-1)^{\frac{\ell}{2}}(1-y^2)^{\frac{m}{2}}\sum_{k=0}^{m}\kappa_{\ell,m,k}^{2}\sin^{k}\phi\cos^{m-k}\phi\sum_{k=0}^{\lfloor\frac{\ell-m}{2}\rfloor}\tilde{\kappa}_{\ell,m,k}y^{2k+2\{\frac{\ell-m}{2}\}}\\
=&\sum_{\chi_{j}}\hat{\kappa}_{\ell,m,\chi_{j},\vec{u}}(p^{0})^{\chi_{0}}(p^{1})^{\chi_{1}}(p^{2})^{\chi_{2}}E^{\chi_{3}}
\sum_{k=0}^{\lfloor\frac{\ell-m}{2}\rfloor}(E^2-1)^{\lfloor\frac{\ell-m}{2}\rfloor-k}\sum_{\hat{\chi}_{j}}\check{\kappa}_{\ell,m,\hat{\chi}_{j},\vec{u}}(p^{0})^{\hat{\chi}_{0}}(p^{3})^{\hat{\chi}_{1}}E^{\hat{\chi}_{2}},
\end{align*}
where $\kappa_{\ell,m}^{1}$, $\kappa_{\ell,m,k}^{2}$, $\hat{\kappa}_{\ell,m,\chi_{j},\vec{u}}$,
$\tilde{\kappa}_{\ell,m,k}$, $\check{\kappa}_{\ell,m,\hat{\chi}_{j},\vec{u}}$ are {constant}
 and independent on $\vec{p}$
and $\chi_{j}$, $\hat{\chi}_{j}\in\mathbb{N}$, $\sum_{j=0}^{3}\chi_{j}=m$, $\sum_{j=0}^{2}\hat{\chi}_{j}=2k+2\{\frac{\ell-m}{2}\}$,
and $\lfloor\cdot\rfloor$ and $\{\cdot\}$ denote the
integer part and the decimal part respectively.

Thus it holds
\[
g^{(0)}_{[\vec{u},\theta]}\tilde{Y}_{\ell,m}\in \mathbb{H}^{g^{(0)}_{[\vec{u},\theta]}}_{\ell},
\]
so each component of $\vec{\mathcal{P}}_{\infty}[\vec{u},\theta]$
(resp.  $\vec{\mathcal{P}}_{M}[\vec{u},\theta]$)
belongs to $\mathbb{H}^{g^{(0)}_{[\vec{u},\theta]}}$  (resp. $\mathbb{H}^{g^{(0)}_{[\vec{u},\theta]}}_{M}$).

(ii) The mathematical induction is  used to prove
that any element in the space $\mathbb{H}^{g^{(0)}_{[\vec{u},\theta]}}$  (resp. $\mathbb{H}^{g^{(0)}_{[\vec{u},\theta]}}_{M}$)
can be written into a linear combination  of
vectors in $\vec{\mathcal{P}}_{\infty}[\vec{u},\theta]$
(resp.  $\vec{\mathcal{P}}_{M}[\vec{u},\theta]$).
For $M=1$, it is clear to have the linear combination
\begin{align*}
p^{\alpha}g^{(0)}_{[\vec{u},\theta]}
=&\left(U^{\alpha}E+\sqrt{E^2-1}(\sqrt{1-y^2}\cos\phi n_{1}^{\alpha}+\sqrt{1-y^2}\sin\phi n_{2}^{\alpha}+y n_{3}^{\alpha})\right)g^{(0)}_{[\vec{u},\theta]}\\
=&\left(U^{\alpha}E-(E^2-1)^{\frac{1}{2}}(Y_{1,1} n_{1}^{\alpha}+Y_{1,-1} n_{2}^{\alpha}+Y_{1,0} n_{3}^{\alpha})\right)g^{(0)}_{[\vec{u},\theta]}\\
=&(c_{1}^{(0)})^{-1}U^{\alpha}\tilde{P}_{1,0}^{(0)}+(c_{0}^{(0)})^{-1}U^{\alpha}x_{1,1}^{(0)}\tilde{P}_{0,0}^{(0)}\\
&-(c_{0}^{(1)})^{-1}n_{1}^{\alpha}\tilde{P}_{0,1}^{(1)}
-(c_{0}^{(1)})^{-1}n_{2}^{\alpha}\tilde{P}_{0,-1}^{(1)}
-(c_{0}^{(1)})^{-1}n_{2}^{\alpha}\tilde{P}_{0,0}^{(1)},
\end{align*}
where  the decomposition of the particle velocity vector \eqref{eq:pdec} has been used.

Assume that the  linear combination
\[
p^{\mu_{1}}p^{\mu_{2}}\cdots p^{\mu_{M}}g^{(0)}_{[\vec{u},\theta]}=\sum_{\ell=0}^{M}\sum_{m=-\ell}^{\ell}\sum_{i=0}^{M-\ell}c_{i,m,\ell}^{\mu_{1},\cdots,\mu_{M}}\tilde{P}_{i,m}^{(\ell)},\quad \mu_{i}=0,1,2,3,\quad c_{i,m,\ell}^{\mu_{1},\cdots,\mu_{M}}\in \mathbb{R},
\]
 holds.  One has to show that $p^{\mu_{1}}p^{\mu_{2}}\cdots p^{\mu_{M+1}}g^{(0)}_{[\vec{u},\theta]}$
 may be written into a linear combination of components of $\vec{\mathcal{P}}_{M+1}[\vec{u},\theta]$.
 Because
\begin{align*}
&p^{\mu_{1}}p^{\mu_{2}}\cdots p^{\mu_{M+1}}g^{(0)}_{[\vec{u},\theta]}\\
=&\left(\sum_{\ell=0}^{M}\sum_{m=-\ell}^{\ell}\sum_{i=0}^{M-\ell}c_{i,m,\ell}^{\mu_{1},\cdots,\mu_{M}}\tilde{P}_{i,m}^{(\ell)}\right)\\
&\left(U^{\mu_{M+1}}E+\sqrt{E^2-1}(\sqrt{1-y^2}\cos\phi n_{1}^{\mu_{M+1}}+\sqrt{1-y^2}\sin\phi n_{2}^{\mu_{M+1}}+y n_{3}^{\mu_{M+1}})\right)\\
=&\sum_{\ell=0}^{M}\sum_{m=-\ell}^{\ell}\sum_{i=0}^{M-\ell}\left\{U^{\mu_{M+1}}c_{i,m,\ell}^{\mu_{1},\cdots,\mu_{M}}E\tilde{P}_{i,m}^{(\ell)}
+n_{3}^{\mu_{M+1}}c_{i,m,\ell}^{\mu_{1},\cdots,\mu_{M}}\sqrt{E^2-1}y\tilde{P}_{i,m}^{(\ell)}\right.\\
&\left.+n_{1}^{\mu_{M+1}}c_{i,m,\ell}^{\mu_{1},\cdots,\mu_{M}}\sqrt{E^2-1}\sqrt{1-y^2}\cos\phi\tilde{P}_{i,m}^{(\ell)}
+n_{2}^{\mu_{M+1}}c_{i,m,\ell}^{\mu_{1},\cdots,\mu_{M}}\sqrt{E^2-1}\sqrt{1-y^2}\sin\phi\tilde{P}_{i,m}^{(\ell)}\right\}\\
=&g^{(0)}_{[\vec{u},\theta]}
\sum_{\ell=0}^{M}\sum_{m=-\ell}^{\ell}\sum_{i=0}^{M-\ell}\left\{U^{\mu_{M+1}}c_{i,m,\ell}^{\mu_{1},\cdots,\mu_{M}}
EP_{i}^{(\ell)}\tilde{Y}_{\ell,m}\right.\\
&\left.+\frac{n_{3}^{\mu_{M+1}}c_{i,m,\ell}^{\mu_{1},\cdots,\mu_{M}}}{2\ell+1}
\left(h_{\ell+1,m}P_{i}^{(\ell)}\tilde{Y}_{\ell+1,m}
+h_{\ell,m}(E^2-1)P_{i}^{(\ell)}\tilde{Y}_{\ell-1,m}\right)\right.\\
&\left.+\frac{n_{1}^{\mu_{M+1}}c_{i,m,\ell}^{\mu_{1},\cdots,\mu_{M}}}{2(2\ell+1)}P_{i}^{(\ell)}\left(\check{s}_{m}(\tilde{h}_{\ell+1,-m}\tilde{Y}_{\ell+1,m-1}-\tilde{h}_{\ell-1,m}(E^2-1)\tilde{Y}_{\ell-1,m-1})\right.\right.\\
&-\left.\left.\hat{s}_{m}(\tilde{h}_{\ell+1,m}\tilde{Y}_{\ell+1,m+1}-\tilde{h}_{\ell-1,-m}(E^2-1)\tilde{Y}_{\ell-1,m+1})\right)\right.\\
&\left.+\frac{n_{2}^{\mu_{M+1}}c_{i,m,\ell}^{\mu_{1},\cdots,\mu_{M}}}{2(2\ell+1)}P_{i}^{(\ell)}\left(s_{m}(\tilde{h}_{\ell-1,-m}(E^2-1)\tilde{Y}_{\ell-1,-m-1}-\tilde{h}_{\ell+1,m}\tilde{Y}_{\ell+1,-m-1})\right.\right.\\
&+\left.\left.\tilde{s}_{m}(\tilde{h}_{\ell-1,m}(E^2-1)\tilde{Y}_{\ell-1,-m+1}-\tilde{h}_{\ell+1,-m}\tilde{Y}_{\ell+1,-m+1})\right)\right\},
\end{align*}
where  \eqref{eq:recc0}-\eqref{eq:recc2} have been used.

By using    the three-term recurrence relations \eqref{eq:recP0P1},  \eqref{eq:recP01}, and \eqref{eq:recP10} for
 the orthogonal polynomials $\{P_{k}^{(\ell)}(x;\zeta), \ell\in\mathbb{N}\}$, one has
\begin{align*}
&p^{\mu_{1}}p^{\mu_{2}}\cdots p^{\mu_{M+1}}g^{(0)}_{[\vec{u},\theta]}\\
=&\sum_{\ell=0}^{M}\sum_{m=-\ell}^{\ell}\sum_{i=0}^{M-\ell}U^{\mu_{M+1}}c_{i,m,\ell}^{\mu_{1},\cdots,\mu_{M}}
\left(a_{i-1}^{(\ell)}\tilde{P}_{i-1,m}^{(\ell)}+b_{i}^{(\ell)}\tilde{P}_{i,m}^{(\ell)}+a_{i}^{(\ell)}\tilde{P}_{i+1,m}^{(\ell)}\right)\\
&+\sum_{\ell=0}^{M}\sum_{m=-\ell}^{\ell}\sum_{i=0}^{M-\ell}n_{3}^{\mu_{M+1}}c_{i,m,\ell}^{\mu_{1},\cdots,\mu_{M}}\\
&\left(\frac{1}{2\ell+1}h_{\ell+1,m}
\left(r_{i-1}^{(\ell+1)}\tilde{P}_{i-2,m}^{(\ell+1)}+q_{i-1}^{(\ell+1)}\tilde{P}_{i-1,m}^{(\ell+1)}+p_{i}^{(\ell+1)}\tilde{P}_{i,m}^{(\ell+1)}\right)\right.\\
&\left.+\frac{1}{2\ell-1}h_{\ell,m}
\left(p_{i}^{(\ell)}\tilde{P}_{i,m}^{(\ell-1)}+q_{i}^{(\ell)}\tilde{P}_{i+1,m}^{(\ell-1)}+r_{i+1}^{(\ell)}\tilde{P}_{i+2,m}^{(\ell-1)}\right)\right)\\
&+\sum_{\ell=0}^{M}\sum_{m=-\ell}^{\ell}\sum_{i=0}^{M-\ell}\frac{n_{1}^{\mu_{M+1}}c_{i,m,\ell}^{\mu_{1},\cdots,\mu_{M}}}{2}\\
&\left(
\frac{1}{2\ell+1}\check{s}_{m}\tilde{h}_{\ell+1,-m}\left(r_{i-1}^{(\ell+1)}\tilde{P}_{i-2,m-1}^{(\ell+1)}+q_{i-1}^{(\ell+1)}\tilde{P}_{i-1,m-1}^{(\ell+1)}+p_{i}^{(\ell+1)}\tilde{P}_{i,m-1}^{(\ell+1)}\right)\right.\\
&\left.-\frac{1}{2\ell+1}\hat{s}_{m}\tilde{h}_{\ell+1,m}\left(r_{i-1}^{(\ell+1)}\tilde{P}_{i-2,m+1}^{(\ell+1)}+q_{i-1}^{(\ell+1)}\tilde{P}_{i-1,m+1}^{(\ell+1)}+p_{i}^{(\ell+1)}\tilde{P}_{i,m+1}^{(\ell+1)}\right)
\right.\\
&\left.+\frac{1}{2\ell-1}
\hat{s}_{m}\tilde{h}_{\ell-1,-m}\left(p_{i}^{(\ell)}\tilde{P}_{i,m+1}^{(\ell-1)}+q_{i}^{(\ell)}\tilde{P}_{i+1,m+1}^{(\ell-1)}+r_{i+1}^{(\ell)}\tilde{P}_{i+2,m+1}^{(\ell-1)}\right)\right.\\
&\left.-\frac{1}{2\ell-1}
\check{s}_{m}\tilde{h}_{\ell-1,-m}\left(p_{i}^{(\ell)}\tilde{P}_{i,m-1}^{(\ell-1)}+q_{i}^{(\ell)}\tilde{P}_{i+1,m-1}^{(\ell-1)}+r_{i+1}^{(\ell)}\tilde{P}_{i+2,m-1}^{(\ell-1)}\right)\right)\\
&+\sum_{\ell=0}^{M}\sum_{m=-\ell}^{\ell}\sum_{i=0}^{M-\ell}\frac{n_{2}^{\mu_{M+1}}c_{i,m,\ell}^{\mu_{1},\cdots,\mu_{M}}}{2}\\
&\left(-\frac{1}{2\ell+1}
\tilde{s}_{m}(\tilde{h}_{\ell+1,-m}\left(r_{i-1}^{(\ell+1)}\tilde{P}_{i-2,-m+1}^{(\ell+1)}+q_{i-1}^{(\ell+1)}\tilde{P}_{i-1,-m+1}^{(\ell+1)}+p_{i}^{(\ell+1)}\tilde{P}_{i,-m+1}^{(\ell+1)}\right)\right.\\
&\left.-\frac{1}{2\ell+1}s_{m}(\tilde{h}_{\ell+1,m}\left(r_{i-1}^{(\ell+1)}\tilde{P}_{i-2,-m-1}^{(\ell+1)}+q_{i-1}^{(\ell+1)}\tilde{P}_{i-1,-m-1}^{(\ell+1)}+p_{i}^{(\ell+1)}\tilde{P}_{i,-m-1}^{(\ell+1)}\right)
\right.\\
&\left.+\frac{1}{2\ell-1}
s_{m}\tilde{h}_{\ell-1,-m}\left(p_{i}^{(\ell)}\tilde{P}_{i,-m-1}^{(\ell-1)}+q_{i}^{(\ell)}\tilde{P}_{i+1,-m-1}^{(\ell-1)}+r_{i+1}^{(\ell)}\tilde{P}_{i+2,-m-1}^{(\ell-1)}\right)\right.\\
&\left.+\frac{1}{2\ell-1}
\tilde{s}_{m}\tilde{h}_{\ell-1,m}\left(p_{i}^{(\ell)}\tilde{P}_{i,-m+1}^{(\ell-1)}+q_{i}^{(\ell)}\tilde{P}_{i+1,-m+1}^{(\ell-1)}+r_{i+1}^{(\ell)}\tilde{P}_{i+2,-m+1}^{(\ell-1)}\right)\right)\\
=&:\sum_{\ell=0}^{M+1}\sum_{m=-\ell}^{\ell}\sum_{i=0}^{M+1-\ell}c_{i,m,\ell}^{\mu_{1},
\cdots,\mu_{M+1}}\tilde{P}_{i,m}^{(\ell)}.
\end{align*}


(iii)
Because of \eqref{eq:pdec}, one has
\[
\frac{d^3\vec{p}}{p^{0}}=\frac{\left|\det(\frac{\partial(p1,p2,p3)}{\partial(E,y,\phi)})\right|}{p^{0}}d\phi dydE=\sqrt{E^2-1}d\phi dydE,
\]
where
\begin{equation*}
\frac{\partial{(p^{1},p^{2},p^{3})}}{\partial(E,y,\phi)}=
\begin{pmatrix}
\begin{smallmatrix}
U^{1}+\frac{E}{\sqrt{E^2-1}}\ell^{1} ,& -\sqrt{E^2-1}\left(\frac{y}{\sqrt{1-y^2}}(\cos\phi n_{1}^{1}+\sin\phi n_{2}^{1})+n_{3}^{1}\right) ,&
\sqrt{E^2-1}\left(\sqrt{1-y^2}(-\sin\phi n_{1}^{1}+\cos\phi n_{2}^{1})\right)\\
U^{2}+\frac{E}{\sqrt{E^2-1}}\ell^{2} ,& -\sqrt{E^2-1}\left(\frac{y}{\sqrt{1-y^2}}(\cos\phi n_{1}^{2}+\sin\phi n_{2}^{2})+n_{3}^{2}\right) ,&
\sqrt{E^2-1}\left(\sqrt{1-y^2}(-\sin\phi n_{1}^{2}+\cos\phi n_{2}^{2})\right)\\
U^{3}+\frac{E}{\sqrt{E^2-1}}\ell^{3} ,& -\sqrt{E^2-1}\left(\frac{y}{\sqrt{1-y^2}}(\cos\phi n_{1}^{3}+\sin\phi n_{2}^{3})+n_{3}^{3}\right) ,&
\sqrt{E^2-1}\left(\sqrt{1-y^2}(-\sin\phi n_{1}^{3}+\cos\phi n_{2}^{3})\right)
\end{smallmatrix}
\end{pmatrix}
\end{equation*}
is the Jacobi martrix.

 Using  \eqref{eq:P01orth} gives
\begin{align}
\nonumber <\tilde{P}_{i,m}^{(\ell)},\tilde{P}_{j,m'}^{(\ell')}>_{g^{(0)}_{[\vec{u},\theta]}}
=&\int_{\mathbb{R}^{3}}P_{i}^{(\ell)}P_{j}^{(\ell')}(E^2-1)^{\frac{\ell+\ell'}{2}}Y_{\ell,m}Y_{\ell',m'}g^{(0)}_{[\vec{u},\theta]}\frac{d^{3}\vec{p}}{p^{0}}\\
\nonumber=&\int_{1}^{+\infty}\int_{0}^{\pi}\int_{0}^{2\pi}P_{i}^{(\ell)}P_{j}^{(\ell')}(E^2-1)^{\frac{\ell+\ell'+1}{2}}g^{(0)}_{[\vec{u},\theta]}Y_{\ell,m}Y_{\ell',m'}d\phi dydE\\
\nonumber=&\frac{4\pi}{2\ell+1}\int_{1}^{+\infty}P_{i}^{(\ell)}P_{j}^{(\ell)}(E^2-1)^{\ell+\frac{1}{2}}g^{(0)}_{[\vec{u},\theta]}dE\delta_{\ell,\ell'}\delta_{m,m'}\\
\nonumber=&\left(P_{i}^{(\ell)},P_{j}^{(\ell)}\right)_{\omega^{(\ell)}}\delta_{\ell,\ell'}\delta_{m,m'}
=\delta_{i,j}\delta_{\ell,\ell'}\delta_{m,m'},\quad \ell\in\mathbb{N}, |m|\leq\ell.\label{eq:orth}
\end{align}

Combining (i) and (ii) with (iii) completes
the proof.
\qed\end{proof}

Since $\mathbb{H}^{g^{(0)}_{[\vec{u},\theta]}}_{{M}}$ is a subspace of $\mathbb{H}^{g^{(0)}_{[\vec{u},\theta]}}_{N}$ when $ M<N<+\infty$,
there exists a matrix $P_{M,N}\in \mathbb{R}^{(N_{M})\times(N_{N})}$ with full row rank
such that
$\vec{\mathcal{P}}_{M}[\vec{u},\theta]=P_{M,N}\vec{\mathcal{P}}_{N}[\vec{u},\theta]$, where
\[
\vec{P}_{M,N}:={\rm diag}\{\vec{I}_{N_{M},N_{M}},\vec{O}_{N_{M},N_{N}-N_{M}}\}.
\]
Using the properties of the orthogonal polynomials  $\{P_{n}^{(\ell)}(x;\zeta), \ell=0,1, n\geq 0\}$ in Section  \ref{sec:orth}
can further give calculation of the partial derivatives and recurrence relations of the basis functions
$ \{\tilde{P}_{k}^{(\ell)}[\vec{u},\theta]\}$.
\begin{lemma}[Derivative relations]
\label{lem:derive}
The partial derivatives of basis functions can  be calculated by
\begin{align*}
\frac{\partial \tilde{P}_{k,m}^{(\ell)}[\vec{u},\theta]}{\partial s}
=&-\frac{\partial \theta}{\partial s}\zeta^2\left(\frac{1}{2}\left(G(\zeta)-\zeta^{-1}-b_{k}^{(\ell)}\right)\tilde{P}_{k,m}^{(\ell)}[\vec{u},\theta]
-a_{k}^{(\ell)}\tilde{P}_{k+1,m}^{(\ell)}[\vec{u},\theta]\right)\\
&-\frac{\partial u_{i}}{\partial s}U^{0}g^{(0)}_{[\vec{u},\theta]}\left[\left(\frac{2\ell+1}{2\ell+3}\frac{k}{\tilde{p}_{k-1}^{(\ell+1)}}-\zeta q_{k-1}^{(\ell+1)}\right)\right.\\
&\cdot\left.\left(n_{3}^{i}h_{\ell+1,m}\tilde{P}_{k-1,m}^{(\ell+1)}+n_{1}^{i}\frac{1}{2}\left(\check{s}_{m}\tilde{h}_{\ell+1,-m}\tilde{P}_{k-1,m-1}^{(\ell+1)}
-\hat{s}_{m}\tilde{h}_{\ell+1,m}\tilde{P}_{k-1,m+1}^{(\ell+1)}\right)\right.\right.\\
&-\left.\left.n_{2}^{i}\frac{1}{2}\left(\tilde{s}_{m}\tilde{h}_{\ell+1,-m}\tilde{P}_{k-1,-m+1}^{(\ell+1)}+s_{m}\tilde{h}_{\ell+1,m}\tilde{P}_{k-1,-m-1}^{(\ell+1)}\right)\right)\right.\\
&-\left.\zeta p_{k}^{(\ell+1)}\left(n_{3}^{i}h_{\ell+1,m}\tilde{P}_{k,m}^{(\ell+1)}+n_{1}^{i}\frac{1}{2}\left(\check{s}_{m}\tilde{h}_{\ell+1,-m}\tilde{P}_{k,m-1}^{(\ell+1)}
-\hat{s}_{m}\tilde{h}_{\ell+1,m}\tilde{P}_{k,m+1}^{(\ell+1)}\right)\right.\right.\\
&-\left.\left.n_{2}^{i}\frac{1}{2}\left(\tilde{s}_{m}\tilde{h}_{\ell+1,-m}\tilde{P}_{k,-m+1}^{(\ell+1)}+s_{m}\tilde{h}_{\ell+1,m}\tilde{P}_{k,-m-1}^{(\ell+1)}\right)\right)\right.\\
&-\left.\frac{2\ell+1}{2\ell-1}\left((k+2\ell+1)\tilde{p}_{k}^{(\ell)}-\zeta q_{k}^{(\ell)}\right)\right.\\
&\left.\left(n_{3}^{i}h_{\ell,m}\tilde{P}_{k+1,m}^{(\ell-1)}-n_{1}^{i}\left(\check{s}_{m}\tilde{h}_{\ell-1,m}\tilde{P}_{k+1,m-1}^{(\ell-1)}
-\hat{s}_{m}\tilde{h}_{\ell-1,-m}\tilde{P}_{k+1,m+1}^{(\ell-1)}\right)\right.\right.\\
&+\left.\left.n_{2}^{i}\left(\tilde{s}_{m}\tilde{h}_{\ell-1,m}\tilde{P}_{k+1,-m+1}^{(\ell-1)}
+s_{m}\tilde{h}_{\ell-1,-m}\tilde{P}_{k+1,-m-1}^{(\ell-1)}\right)\right)\right.\\
&+\frac{2\ell+1}{2\ell-1}\left.\zeta r_{k+1}^{(\ell)}
\left(n_{3}^{i}h_{\ell,m}\tilde{P}_{k+2,m}^{(\ell-1)}-n_{1}^{i}\left(\check{s}_{m}\tilde{h}_{\ell-1,m}\tilde{P}_{k+2,m-1}^{(\ell-1)}
-\hat{s}_{m}\tilde{h}_{\ell-1,-m}\tilde{P}_{k+2,m+1}^{(\ell-1)}\right)\right.\right.\\
&+\left.\left.n_{2}^{i}\left(\tilde{s}_{m}\tilde{h}_{\ell-1,m}\tilde{P}_{k+2,-m+1}^{(\ell-1)}
+s_{m}\tilde{h}_{\ell-1,-m}\tilde{P}_{k+2,-m-1}^{(\ell-1)}\right)\right)\right]\\
&+\frac{\partial u_{i}}{\partial s}U^{0}(U^{0}+1)^{-1}\left[-m(U^{2}\delta_{1,i}-U^{1}\delta_{2,i})\tilde{P}_{k,-m}^{(\ell)}\right.\\
&\left.+\frac{1}{2}(\delta_{1,i}U^{3}-\delta_{3,i}U^{1})
\left(\check{s}_{m}\hat{h}_{\ell,m}\tilde{P}_{k,m-1}^{(\ell)}[\vec{u},\theta]-\hat{s}_{m}\hat{h}_{\ell,-m}\tilde{P}_{k,m+1}^{(\ell)}\right)\right.\\
&\left.-\frac{1}{2}(\delta_{2,i}U^{3}-\delta_{3,i}U^{2})
\left(\tilde{s}_{m}\hat{h}_{\ell,m}\tilde{P}_{k,-m+1}^{(\ell)}[\vec{u},\theta]+s_{m}\hat{h}_{\ell,-m}\tilde{P}_{k,-m-1}^{(\ell)}\right)\right].
\end{align*}
for $s=x^{\alpha}$.
It indicates that $\frac{\partial \tilde{P}_{k,m}^{(\ell)}[\vec{u},\theta]}{\partial s}$
$\in \mathbb{H}^{g^{(0)}_{[\vec{u},\theta]}}_{k+\ell+1}$.
\end{lemma}
\begin{proof}
	For $s=t$ and $x$,
it is clear to have
\[
\frac{\partial U^{\alpha}}{\partial s}=\frac{\partial u_{i}}{\partial s}U^{0}(U^{i}U^{\alpha}+\delta_{i,\alpha}),
\quad \frac{\partial U_{\alpha}}{\partial s}U^{\alpha}=0,\quad
\frac{\partial n_{i}^{\alpha}}{\partial s}n_{i}^{\beta}g_{\alpha\beta}=0,\quad,
\frac{\partial n_{i}^{\alpha}}{\partial s}U_{\alpha}+\frac{\partial U_{\alpha}}{\partial s}n_{i}^{\alpha}=0.
\]
Thus one has
\begin{align*}
\frac{\partial E}{\partial s}=&
\frac{\partial u_{i}}{\partial s}U^{0}(U^{i}U_{\alpha}-\delta_{i,\alpha})p^{\alpha}
=-\frac{\partial u_{i}}{\partial s}U^{0}\sqrt{E^2-1}\left(n_{1}^{i}I_{1}+n_{2}^{i}I_{2}
+n_{3}^{i}I_{0}\right),\\
\frac{\partial y}{\partial s}=&-yE(E^2-1)^{-1}\frac{\partial E}{\partial s}-\frac{\partial n_{3}^{\alpha}}{\partial s}p_{\alpha}(E^2-1)^{-\frac{1}{2}}\\
=&\frac{\partial u_{i}}{\partial s}U^{0}E(E^2-1)^{-\frac{1}{2}}\left(n_{1}^{i}\tilde{I}_{1}+n_{2}^{i}\tilde{I}_{2}
+n_{3}^{i}\tilde{I}_{0}\right)-\frac{\partial n_{3}^{\alpha}}{\partial s}g_{\alpha\beta}(n_{1}^{\beta}I_{1}+n_{2}^{\beta}I_{2})\\
=&\frac{\partial u_{i}}{\partial s}U^{0}
\left(\frac{E}{\sqrt{E^2-1}}(n_{1}^{i}\tilde{I}_{1}+n_{2}^{i}\tilde{I}_{2}+n_{3}^{i}\tilde{I}_{0})
+\frac{(\delta_{1,i}U^{3}-\delta_{3,i}U^{1})I_{1}+
(\delta_{2,i}U^{3}-\delta_{3,i}U^{2})I_{2}}{U^{0}+1}\right),\\
\frac{\partial \phi}{\partial s}=&-(\frac{\partial n_{2}^{\alpha}}{\partial s}p_{\alpha}\frac{\cos\phi}{\sqrt{1-y^2}\sqrt{E^2-1}}-\frac{\partial n_{1}^{\alpha}}{\partial s}p_{\alpha}\frac{\sin\phi}{\sqrt{1-y^2}\sqrt{E^2-1}})\\
=&\frac{\partial u_{i}}{\partial s}U^{0}\left(\frac{E}{
\sqrt{E^2-1}}(n_{1}^{i}\hat{I}_{1}+n_{2}^{i}\hat{I}_{2})\right)
+g_{\alpha\beta}\left(\frac{\partial n_{2}^{\alpha}}{\partial s}(n_{3}^{\beta}\hat{I}_{4}-n_{1}^{\beta})
+\frac{\partial n_{1}^{\alpha}}{\partial s}n_{3}^{\beta}\hat{I}_{3}\right)\\
=&\frac{\partial u_{i}}{\partial s}
U^{0}\left(\frac{E}{
\sqrt{E^2-1}}(n_{1}^{i}\hat{I}_{1}+n_{2}^{i}\hat{I}_{2})+\frac{U^{2}\delta_{1,i}-U^{1}\delta_{2,i}+
(\delta_{1,i}U^{3}-\delta_{3,i}U^{1})\hat{I}_{3}+(\delta_{2,i}U^{3}-\delta_{3,i}U^{2})\hat{I}_{4}}{U^{0}+1}\right).
\end{align*}

Using  those above identities and  \eqref{eq:equm1-zzzzz}  gives
\[
\frac{\partial g^{(0)}_{[\vec{u},\theta]}}{\partial s}=-\left(\frac{\partial \theta}{\partial s}\zeta^2\left(G(\zeta)-\zeta^{-1}-E\right)+\zeta\frac{\partial E}{\partial s}\right) g^{(0)}_{[\vec{u},\theta]}.
\]
Using Lemma \ref{lem:recder} gives
\begin{align*}
\frac{\partial \tilde{Y}_{\ell,m}}{\partial s}
=&\frac{\partial y}{\partial s}\frac{d}{dy}Y_{\ell,m}(E^2-1)^{\frac{\ell}{2}}-mY_{\ell,-m}(E^2-1)^{\frac{\ell}{2}}\frac{\partial \phi}{\partial s}+\ell EY_{\ell,m}(E^2-1)^{\frac{\ell-2}{2}}\frac{\partial E}{\partial s},\\
=&\frac{\partial u_{i}}{\partial s}U^{0}E(E^2-1)^{\frac{\ell-1}{2}}n_{3}^{i}\left(\tilde{I}_{0}\frac{d}{dy}Y_{\ell,m}-
\ell I_{0}Y_{\ell,m}\right)\\
+&\frac{\partial u_{i}}{\partial s}U^{0}E(E^2-1)^{\frac{\ell-1}{2}}n_{1}^{i}\left(\tilde{I}_{1}\frac{d}{dy}Y_{\ell,m}-\ell I_{1}Y_{\ell,m}-m\hat{I}_{1}Y_{\ell,-m}\right)\\
+&\frac{\partial u_{i}}{\partial s}U^{0}E(E^2-1)^{\frac{\ell-1}{2}}n_{2}^{i}\left(\tilde{I}_{2}\frac{d}{dy}Y_{\ell,m}-\ell I_{2}Y_{\ell,m}-m\hat{I}_{2}Y_{\ell,-m}\right)\\
+&\frac{\partial u_{i}}{\partial s}U^{0}\frac{\delta_{1,i}U^{3}-\delta_{3,i}U^{1}}{U^{0}+1}
(E^2-1)^{\frac{\ell}{2}}\left(I_{1}\frac{d}{dy}Y_{\ell,m}-m\hat{I}_{3}Y_{\ell,-m}\right)\\
+&\frac{\partial u_{i}}{\partial s}U^{0}\frac{\delta_{2,i}U^{3}-\delta_{3,i}U^{2}}{U^{0}+1}
(E^2-1)^{\frac{\ell}{2}}\left(I_{2}\frac{d}{dy}Y_{\ell,m}-m\hat{I}_{4}Y_{\ell,-m}\right)\\
-&m\frac{\partial u_{i}}{\partial s}U^{0}\frac{U^{2}\delta_{1,i}-U^{1}\delta_{2,i}}{U^{0}+1}\tilde{Y}_{\ell,-m}\\
=&-\frac{\partial u_{i}}{\partial s}U^{0}n_{3}^{i}h_{\ell,m}E\tilde{Y}_{\ell-1,m}\\
&+\frac{1}{2}\frac{\partial u_{i}}{\partial s}U^{0}n_{1}^{i}E\left(\check{s}_{m}\tilde{h}_{\ell-1,m}\tilde{Y}_{\ell-1,m-1}
-\hat{s}_{m}\tilde{h}_{\ell-1,-m}\tilde{Y}_{\ell-1,m+1}\right)\\
&-\frac{1}{2}\frac{\partial u_{i}}{\partial s}U^{0}n_{2}^{i}E\left(\tilde{s}_{m}\tilde{h}_{\ell-1,m}\tilde{Y}_{\ell-1,-m+1}
+s_{m}\tilde{h}_{\ell-1,-m}\tilde{Y}_{\ell-1,-m-1}\right)\\
&+\frac{1}{2}\frac{\partial u_{i}}{\partial s}U^{0}\frac{\delta_{1,i}U^{3}-\delta_{3,i}U^{1}}{U^{0}+1}
\left(\check{s}_{m}\hat{h}_{\ell,m}\tilde{Y}_{\ell,m-1}-\hat{s}_{m}\hat{h}_{\ell,-m}\tilde{Y}_{\ell,m+1}\right)\\
&-\frac{1}{2}\frac{\partial u_{i}}{\partial s}U^{0}\frac{\delta_{2,i}U^{3}-\delta_{3,i}U^{2}}{U^{0}+1}
\left(\tilde{s}_{m}\hat{h}_{\ell,m}\tilde{Y}_{\ell,-m+1}+s_{m}\hat{h}_{\ell,-m}\tilde{Y}_{\ell,-m-1}\right)\\
&-m\frac{\partial u_{i}}{\partial s}U^{0}\frac{U^{2}\delta_{1,i}-U^{1}\delta_{2,i}}{U^{0}+1}\tilde{Y}_{\ell,-m},
\end{align*}
and
\begin{align*}
\frac{\partial P_{k}^{(\ell)}}{\partial s}=&-\zeta^2\frac{\partial P_{k}^{(\ell)}}{\partial \zeta}\frac{\partial \theta}{\partial s}+\frac{\partial P_{k}^{(\ell)}}{\partial E}\frac{\partial E}{\partial s}\\
=&-\zeta^2\left(a_{k-1}^{(\ell)}P_{k-1}^{(\ell)}-\frac{1}{2}\left(G(\zeta)-\zeta^{-1}-b_{k}^{(\ell)}\right)P_{k}^{(\ell)}\right)\frac{\partial \theta}{\partial s}\\
&+\left(\frac{2\ell+1}{2\ell+3}\frac{k}{\tilde{p}_{k-1}^{(\ell+1)}}P_{k-1}^{(\ell+1)}+\zeta r_{k-1}^{(\ell+1)}P_{k-2}^{(\ell+1)}\right)\frac{\partial E}{\partial s}\\
=&-\zeta^2\left(a_{k-1}^{(\ell)}P_{k-1}^{(\ell)}-\frac{1}{2}\left(G(\zeta)-\zeta^{-1}-b_{k}^{(\ell)}\right)P_{k}^{(\ell)}\right)\frac{\partial \theta}{\partial s}\\
&+(E^2-1)^{-1}\left(\frac{2\ell+1}{2\ell-1}\left((k+2\ell+1)\tilde{p}_{k}^{(\ell)}P_{k+1}^{(\ell-1)}+
     \zeta p_{k}^{(\ell)} P_{k}^{(\ell-1)}\right)-(2\ell+1)EP_{k}^{(\ell)}\right)\frac{\partial E}{\partial s}.
\end{align*}
Combining \eqref{eq:recP0P1} and \eqref{eq:derivePn0x} with \eqref{eq:derivePn1x}
yields
\begin{align*}
\frac{\partial P_{k}^{(\ell)}g^{(0)}_{[\vec{u},\theta]}}{\partial s}\tilde{Y}_{\ell,m}
=&-\frac{\partial \theta}{\partial s}\zeta^2\left(\frac{1}{2}\left(G(\zeta)-\zeta^{-1}-b_{k}^{(\ell)}\right)\tilde{P}_{k,m}^{(\ell)}[\vec{u},\theta]
-a_{k}^{(\ell)}\tilde{P}_{k+1,m}^{(\ell)}[\vec{u},\theta]\right)\\
&+\frac{\partial E}{\partial s}g^{(0)}_{[\vec{u},\theta]}(E^2-1)^{-1}\left(\frac{2\ell+1}{2\ell-1}\left(\left((k+2\ell+1)\tilde{p}_{k}^{(\ell)}-\zeta q_{k}^{(\ell)}\right)P_{k+1}^{(\ell-1)}
     -\zeta r_{k+1}^{(\ell)} P_{k+2}^{(\ell-1)}\right)\right.\\
&\left.-(2\ell+1)EP_{k}^{(\ell)}\right)\tilde{Y}_{\ell,m}\\
=&-\frac{\partial \theta}{\partial s}\zeta^2\left(\left(G(\zeta)-\zeta^{-1}-b_{k}^{(\ell)}\right)\tilde{P}_{k,m}^{(\ell)}[\vec{u},\theta]
-a_{k}^{(\ell)}\tilde{P}_{k+1,m}^{(\ell)}[\vec{u},\theta]\right)\\
&+\frac{\partial E}{\partial s}g^{(0)}_{[\vec{u},\theta]}\left(\left(\frac{2\ell+1}{2\ell+3}\frac{k}{\tilde{p}_{k-1}^{(\ell+1)}}-\zeta q_{k-1}^{(\ell+1)}\right)P_{k-1}^{(\ell+1)}-\zeta p_{k}^{(\ell+1)}P_{k}^{(\ell+1)}\right)\tilde{Y}_{\ell,m}.
\end{align*}
The derivation rule of compound function gives
\[
\frac{\partial \tilde{P}_{k,m}^{(\ell)}[\vec{u},\theta]}{\partial s}=\frac{\partial P_{k}^{(\ell)}g^{(0)}_{[\vec{u},\theta]}}{\partial s}\tilde{Y}_{\ell,m}+\frac{\partial \tilde{Y}_{\ell,m}}{\partial s}P_{k}^{(\ell)}g^{(0)}_{[\vec{u},\theta]}.
\]
Combining the above and using Lemma \ref{lem:recder} again can complete the proof.
\qed\end{proof}

\begin{lemma}[Recurrence relations]
\label{lem:rec}
The basis functions $\{\tilde{P}_{k,m}^{(\ell)}[\vec{u},\theta], \ell, m, k\in\mathbb{N},|m|\leq\ell,{k\leq M-\ell} \}$ satisfy the following recurrence relations
\begin{equation}
    \label{eq:rec}
    \begin{aligned}
      p^{\alpha}\vec{\mathcal{P}}_{M}[\vec{u},\theta]=&\vec{M}_{M}^{\alpha}\vec{\mathcal{P}}_{M}[\vec{u},\theta]+
U^{\alpha}\vec{P}_{M}^{p}\vec{e}_{M}^{0}+n_{1}^{\alpha}\vec{P}_{M}^{p}\vec{e}_{M}^{1}+n_{2}^{\alpha}\vec{P}_{M}^{p}\vec{e}_{M}^{2}
++n_{3}^{\alpha}\vec{P}_{M}^{p}\vec{e}_{M}^{3},
    \end{aligned}
    \end{equation}
where
     \begin{equation}
    \label{eq:MtMx}
    \vec{M}_{M}^{\alpha}:=n_{1}^{\alpha}\vec{P}_{M}^{p}\vec{A}_{M}^{1}(\vec{P}_{M}^{p})^{T}+ n_{2}^{\alpha}\vec{P}_{M}^{p}\vec{A}_{M}^{2}(\vec{P}_{M}^{p})^{T}+n_{3}^{\alpha}\vec{P}_{M}^{p}\vec{A}_{M}^{3}(\vec{P}_{M}^{p})^{T}+ U^{\alpha}\vec{P}_{M}^{p}\vec{A}_{M}^{0}(\vec{P}_{M}^{p})^{T} ,
    \end{equation}
in which   $\vec{P}_{M}^{p}$ is a permutation matrix  satisfying
     \begin{equation}
     \label{eq:MtMx-22222}
    \vec{P}_{M}^{p}\vec{\mathcal{\tilde{P}}}_{M}[\vec{u},\theta]=\vec{\mathcal{P}}_{M}[\vec{u},\theta], \quad {\vec{P}}_{M}^{p}({\vec{P}}_{M}^{p})^{T}=({\vec{P}}_{M}^{p})^{T}{\vec{P}}_{M}^{p}
    ={\vec{I}},
    \end{equation}
with
\begin{align*}
\vec{\mathcal{\tilde{P}}}_{M}[\vec{u},\theta]:&=(\tilde{P}_{0,0}^{(0)},\cdots,\tilde{P}_{M,0}^{(0)},
\tilde{P}_{0,-1}^{(1)},\cdots,\tilde{P}_{M-1,-1}^{(1)},
\cdots,\tilde{P}_{0,1}^{(1)},
\cdots,\tilde{P}_{M-1,1}^{(1)},
\cdots,\tilde{P}_{0,-M}^{(M)},\cdots,\tilde{P}_{0,M}^{(M)} )^{T}.
\end{align*}
Moreover, $\vec{A}_{M}^{\alpha}$ is a partitioned matrix such as
\[
\vec{A}_{M}^{\alpha}=\begin{pmatrix}
\begin{smallmatrix}
\vec{A}_{M}^{\alpha}[(0,0),(0,0)]& \vec{A}_{M}^{\alpha}[(0,0),(-1,1)]&\vec{A}_{M}^{\alpha}[(0,0),(0,1)]&\vec{A}_{M}^{\alpha}[(0,0),(1,1)]
&\cdots&\vec{A}_{M}^{\alpha}[(0,0),(-M,M)]&\cdots&\vec{A}_{M}^{\alpha}[(0,0),(M,M)]\\
\vec{A}_{M}^{\alpha}[(-1,1),(0,0)]& \vdots&\vdots&\vdots
&\vdots&\vdots&\cdots&\vec{A}_{M}^{\alpha}[(-1,1),(M,M)]\\
\vdots& \vdots&\vdots&\vdots
&\vdots&\vdots&\cdots&\vdots\\
\vec{A}_{M}^{\alpha}[(M,M),(0,0)]& \cdots&\cdots&\cdots
&\cdots&\cdots&\cdots&\vec{A}_{M}^{\alpha}[(M,M),(M,M)]
\end{smallmatrix}
\end{pmatrix},
\]
where
\[
\vec{A}_{M}^{0}[(m,\ell),(m,\ell)]=\vec{J}_{M-\ell}^{(\ell)},
\]
and
\begin{align*}
\vec{A}_{M}^{1}[(m,\ell-1),(m-1,\ell)]=\frac{(\vec{\tilde{J}}_{M-\ell}^{(\ell)})^{T}\check{s}_{m}\tilde{h}_{\ell,-m}}{2(2\ell-1)}, &\quad
\vec{A}_{M}^{1}[(m,\ell-1),(m+1,\ell)]=\frac{(\vec{\tilde{J}}_{M-\ell}^{(\ell)})^{T}\hat{s}_{m}\tilde{h}_{\ell,m}}{-2(2\ell-1)}, \\
\vec{A}_{M}^{1}[(m,\ell),(m-1,\ell-1)]=\frac{\vec{\tilde{J}}_{M-\ell}^{(\ell)}\check{s}_{m}\tilde{h}_{\ell-1,m}}{-2(2\ell-1)},&\quad
\vec{A}_{M}^{1}[(m,\ell),(m+1,\ell-1)]=\frac{\vec{\tilde{J}}_{M-\ell}^{(\ell)}\hat{s}_{m}\tilde{h}_{\ell-1,-m}}{2(2\ell-1)},\\
\vec{A}_{M}^{2}[(m,\ell-1),(-m-1,\ell)]=\frac{(\vec{\tilde{J}}_{M-\ell}^{(\ell)})^{T}_{m}s_{m}\tilde{h}_{\ell,-m}}{-2(2\ell-1)}, &\quad
\vec{A}_{M}^{2}[(m,\ell-1),(-m+1,\ell)]=\frac{(\vec{\tilde{J}}_{M-\ell}^{(\ell)})^{T}\tilde{s}_{m}\tilde{h}_{\ell,-m}}{-2(2\ell-1)}, \\
\vec{A}_{M}^{2}[(m,\ell),(-m-1,\ell-1)]=\frac{\vec{\tilde{J}}_{M-\ell}^{(\ell)}s_{m}\tilde{h}_{\ell-1,-m}}{2(2\ell-1)},&\quad
\vec{A}_{M}^{2}[(m,\ell),(-m+1,\ell-1)]=\frac{\vec{\tilde{J}}_{M-\ell}^{(\ell)}\tilde{s}_{m}\tilde{h}_{\ell-1,m}}{2(2\ell-1)},\\
\vec{A}_{M}^{3}[(m,\ell-1),(m,\ell)]=\frac{(\vec{\tilde{J}}_{M-\ell}^{(\ell)})^{T}h_{\ell,m}}{2\ell-1},& \quad
\vec{A}_{M}^{3}[(m,\ell),(m,\ell-1)]=\frac{\vec{\tilde{J}}_{M-\ell}^{(\ell)}h_{\ell,m}}{2\ell-1},
\end{align*}
with other blocks are zero matrices.

The vector $\vec{e}_{M}^{\alpha}$ is   partitioned as follows
\[
\vec{e}_{M}^{\alpha}=\left(\vec{e}_{M}^{\alpha}(0,0)^{T},\cdots,\vec{e}_{M}^{\alpha}(-M,M),\cdots,\vec{e}_{M}^{\alpha}(M,M)\right)^{T},
\]
where
\[
\vec{e}_{M}^{0}(m,\ell)=a_{M-\ell}^{(\ell)}\tilde{P}_{M-\ell+1,m}^{(\ell)}\vec{e}_{M-\ell+1},
\]
\begin{align*}
\vec{e}_{M}^{1}(m,\ell)=&\frac{p_{M-\ell}^{(\ell+1)}(\check{s}_{m}\tilde{h}_{\ell+1,-m}\tilde{P}_{M-\ell,m-1}^{(\ell+1)}
-\hat{s}_{m}\tilde{h}_{\ell+1,m}\tilde{P}_{M-\ell,m+1}^{(\ell+1)})}{2(2\ell+1)}\vec{e}_{M-\ell+1}\\
&+\frac{r_{M-\ell+1}^{(\ell)}(\hat{s}_{m}\tilde{h}_{\ell-1,-m}\tilde{P}_{M-\ell+2,m+1}^{(\ell-1)}
-\check{s}_{m}\tilde{h}_{\ell-1,m}\tilde{P}_{M-\ell+2,m-1}^{(\ell-1)})
}{2(2\ell+1)}\vec{e}_{M-\ell+1},\\
\vec{e}_{M}^{2}(m,\ell)=&-\frac{p_{M-\ell}^{(\ell+1)}(\tilde{s}_{m}\tilde{h}_{\ell+1,-m}\tilde{P}_{M-\ell,-m+1}^{(\ell+1)}
+s_{m}\tilde{h}_{\ell+1,m}\tilde{P}_{M-\ell,-m-1}^{(\ell+1)})}{2(2\ell+1)}\vec{e}_{M-\ell+1}\\
&+\frac{r_{M-\ell+1}^{(\ell)}(\hat{s}_{m}\tilde{h}_{\ell-1,-m}\tilde{P}_{M-\ell+2,-m-1}^{(\ell-1)}
+\check{s}_{m}\tilde{h}_{\ell-1,m}\tilde{P}_{M-\ell+2,-m+1}^{(\ell-1)})
}{2(2\ell+1)}\vec{e}_{M-\ell+1},\\
\vec{e}_{M}^{3}(m,\ell)=&\frac{p_{M-\ell}^{(\ell+1)}h_{\ell+1,m}\tilde{P}_{M-\ell,m}^{(\ell+1)}
+r_{M-\ell+1}^{(\ell)}h_{\ell,m}\tilde{P}_{M-\ell+2,m}^{(\ell-1)}}{2\ell+1}\vec{e}_{M-\ell+1},
\end{align*}
and
$\vec{e}_{M-\ell+1}$ {is} the last column of the identity matrix of order $(M-\ell+1)$.
\end{lemma}
\begin{proof}
Using the three-term recurrence relations
 \eqref{eq:recP0P1mat}, \eqref{eq:recPQ}, \eqref{eq:recQP}, \eqref{eq:recc0}-\eqref{eq:recc2}, and
 \eqref{eq:pdec} can complete the proof.
\qed\end{proof}

For a finite integer $M\geq1$,  define an operator
$\Pi_{M}[\vec{u},\theta]: \mathbb{H}^{g^{(0)}_{[\vec{u},\theta]}}\rightarrow \mathbb{H}^{g^{(0)}_{[\vec{u},\theta]}}_{M}$
 by
\begin{equation}
\label{EQ-projection-aaaa}
\Pi_{M}[\vec{u},\theta]f:=\sum_{\ell=0}^{M}\sum_{m=-\ell}^{\ell}\sum_{i=0}^{M-\ell}f_{i,m}^{(\ell)}\tilde{P}_{i,m}^{(\ell)}[\vec{u},\theta],
\end{equation}
or  in a compact form
\begin{equation}\label{EQ-projection-aaaa2}
\Pi_{M}[\vec{u},\theta]f=[\vec{\mathcal{P}}_{M}[\vec{u},\theta],\vec{f}_{M}]_{M},
\end{equation}
where
\begin{align}
f_{i,m}^{(\ell)}&=<f,\tilde{P}_{i,m}^{(\ell)}[\vec{u},\theta]>_{g^{(0)}_{[\vec{u},\theta]}},\quad |m|\leq\ell,\quad i\leq M-\ell, \label{eq:deff1}\\
\vec{f}_{M}&=(f_{0,0}^{(0)},f_{1,0}^{(0)},f_{0,-1}^{(1)},
f_{0,0}^{(1)},f_{0,1}^{(1)},\cdots,f_{M,0}^{(0)},f_{M-1,-1}^{(1)},
f_{M-1,0}^{(1)},f_{M-1,1}^{(1)},\cdots,f_{0,-M}^{(M)},\cdots,f_{0,M}^{(M)} )^{T}.
\label{EQ-projection-bbbb}
\end{align}
and the symbol $[\cdot,\cdot]_{M}$ denotes the common inner product of two $N_{M}$-dimensional vectors.

\begin{lemma}
\label{lem:project}
The operator	$\Pi_{M}[\vec{u},\theta]$ is
 linear bounded and projection operator in sense that
\begin{description}
	\item[(i)]   $\Pi_{M}[\vec{u},\theta]f  \in\mathbb{H}^{g^{(0)}_{[\vec{u},\theta]}}_{M}$ for all
	$f\in\mathbb{H}^{g^{(0)}_{[\vec{u},\theta]}}$,
   \item[(ii)] $\Pi_{M}[\vec{u},\theta]f=f$ for all $f\in\mathbb{H}^{g^{(0)}_{[\vec{u},\theta]}}_{M}$.
\end{description}
\end{lemma}
\begin{proof}
It is obvious that $\Pi_{M}[\vec{u},\theta]$ is a linear bounded operator and
 $\Pi_{M}[\vec{u},\theta]f\in\mathbb{H}^{g^{(0)}_{[\vec{u},\theta]}}_{M}$ for all $f\in\mathbb{H}^{g^{(0)}_{[\vec{u},\theta]}}$.

For each $f\in\mathbb{H}^{g^{(0)}_{[\vec{u},\theta]}}_{M}$, 
besides \eqref{EQ-projection-aaaa},  one has by using Lemma \ref{lem:basis}
\[
f=\sum_{\ell=0}^{M}\sum_{m=-\ell}^{\ell}\sum_{i=0}^{M-\ell}\tilde{f}_{i,m}^{(\ell)}\tilde{P}_{i,m}^{(\ell)}[\vec{u},\theta].
\]
Taking respectively the inner product with $\tilde{P}_{i,m}^{(\ell)}[\vec{u},\theta]$ from both sides
of the last equation gives
\[
    f_{i,m}^{(\ell)}=<f,\tilde{P}_{i,m}^{(\ell)}[\vec{u},\theta]>_{g^{(0)}_{[\vec{u},\theta]}}=\tilde{f}_{i,m}^{(\ell)}.
\]
The proof is completed.
\qed\end{proof}

\begin{remark}
	 The so-called  Grad type expansion is to expand the distribution function $f(\vec x,\vec p,t)$
	 in the weighted polynomial space $\mathbb{H}^{g^{(0)}_{[\vec{u},\theta]}}$ as follows
	  \[
	  f(\vec x,\vec p,t)=\left[\vec{\mathcal{P}}_{\infty}[\vec{u},\theta],\vec{f}_{\infty}\right]_{\infty},
	  \]
	  where the symbol $[\cdot,\cdot]_{\infty}$ denotes the common inner product of two infinite-dimensional vectors,
	  and $\vec{f}_{\infty}=(f_{0,0}^{(0)},\cdots,f_{M,0}^{(0)},
f_{M-1,-1}^{(1)},f_{M-1,0}^{(1)},f_{M-1,1}^{(1)},\cdots,f_{0,-M}^{(M)},\cdots,f_{0,M}^{(M)},\cdots)^{T}$.
\end{remark}

\subsection{Derivation of the moment model}
\label{subsec:deduction}

Based on the weighted polynomial spaces $\mathbb{H}^{g^{(0)}_{[\vec{u},\theta]}}$
and $\mathbb{H}^{g^{(0)}_{[\vec{u},\theta]}}_M$ in Section \ref{subsec:expand}
and the projection operator  $\Pi_{M}[\vec{u},\theta]$ defined in \eqref{EQ-projection-aaaa},
the moment method by the operator projection \cite{MR:2014,Kuang:2017} may be
implemented  for the  3D special relativistic Boltzmann equation \eqref{eq:Boltz}.
In view of the fact that  the  variables $\{n, \vec{u}, \theta, \Pi\}$
are several physical quantities of practical interest
and the first three are required in calculating the equilibrium distribution $f^{(0)}$,
the $N_{M}$-dimensional vector
 \[
\vec{W_{M}}=(n,\vec{u},\theta,\Pi,\tilde{f}_{0,-1}^{(1)},\tilde{f}_{0,0}^{(1)},\tilde{f}_{0,1}^{(1)},\cdots,f_{M,0}^{(0)},f_{M-1,-1}^{(1)},
f_{M-1,0}^{(1)},f_{M-1,1}^{(1)},\cdots,f_{0,-M}^{(M)},\cdots,f_{0,M}^{(M)} )^{T},
\]
will be considered as the dependent variable vector, instead of $\vec{f}_{M}$ defined in \eqref{EQ-projection-bbbb},
where $\tilde{f}_{0,m}^{(1)}:=f_{0,m}^{(1)}(c_{0}^{(1)})^{-1}$   satisfying
\[
n^{\alpha}=-n_{1}^{\alpha}\tilde{f}_{0,1}^{(1)}-n_{2}^{\alpha}\tilde{f}_{0,-1}^{(1)}+n_{3}^{\alpha}\tilde{f}_{0,0}^{(1)}.
\]
Thanks to \eqref{eq:condition-222} and \eqref{eq:condition},
the relations between $\vec{W_{M}}$ and $\vec{f_{M}}$ is
\begin{equation}
\label{eq:constraint}
\vec{f}_{M}=\vec{D}_{M}^{W}\vec{W}_{M},
\end{equation}
where  the square matrix $\vec{D}_{M}^{W}$ depends on $\theta$ and  is of the following explicit form
\[
\vec{D}_{1}^{W}=
\begin{pmatrix}
(c_{0}^{(0)})^{-1}& \ \\
\                 &\vec{O}_{4\times 4}
\end{pmatrix},\quad
\vec{D}_{2}^{W}=
\begin{pmatrix}
\vec{D}_{1}^{W} &\vec{D}_{5\times 4}^{12}    &\  \\    
\ &\vec{D}_{4\times 4}^{22} &\  \\
\ &\    &\vec{I}_{5\times5}
\end{pmatrix},
\]
where
\[
\vec{D}_{5\times 4}^{12}=
\begin{pmatrix}
& -3c_{0}^{(0)}   & 0    & 0     & 0      \\    
& -3\zeta c_{1}^{(0)}x_{1,1}^{(0)} & 0    & 0     & 0    \\
& 0  & c_{0}^{(1)}  & 0     & 0      \\
& 0 & 0     & c_{0}^{(1)}    & 0              \\
& 0 & 0     & 0     & c_{0}^{(1)}    &\   \\
\end{pmatrix},\quad
\vec{D}_{4\times 4}^{22}=\begin{pmatrix}
& -3c_{2}^{(0)}x_{1,2}^{(0)}   & 0    & 0     & 0      \\    
& 0 & -c_{1}^{(1)}x_{1,1}^{(1)}    & 0     & 0    \\
& 0  &  0 & -c_{1}^{(1)}x_{1,1}^{(1)}    & 0      \\
& 0 & 0     & 0   &-c_{1}^{(1)}x_{1,1}^{(1)}           \\
\end{pmatrix},
\]
and  $\vec{D}_{M}^{W}={\rm diag}\{\vec{D}_{2}^{W},\vec{I}_{N_{M}-N_{2}}\}$ for $M\geq3$.
\begin{figure}[h!]
      \centering
      \includegraphics[width=14cm,height=5.2cm]{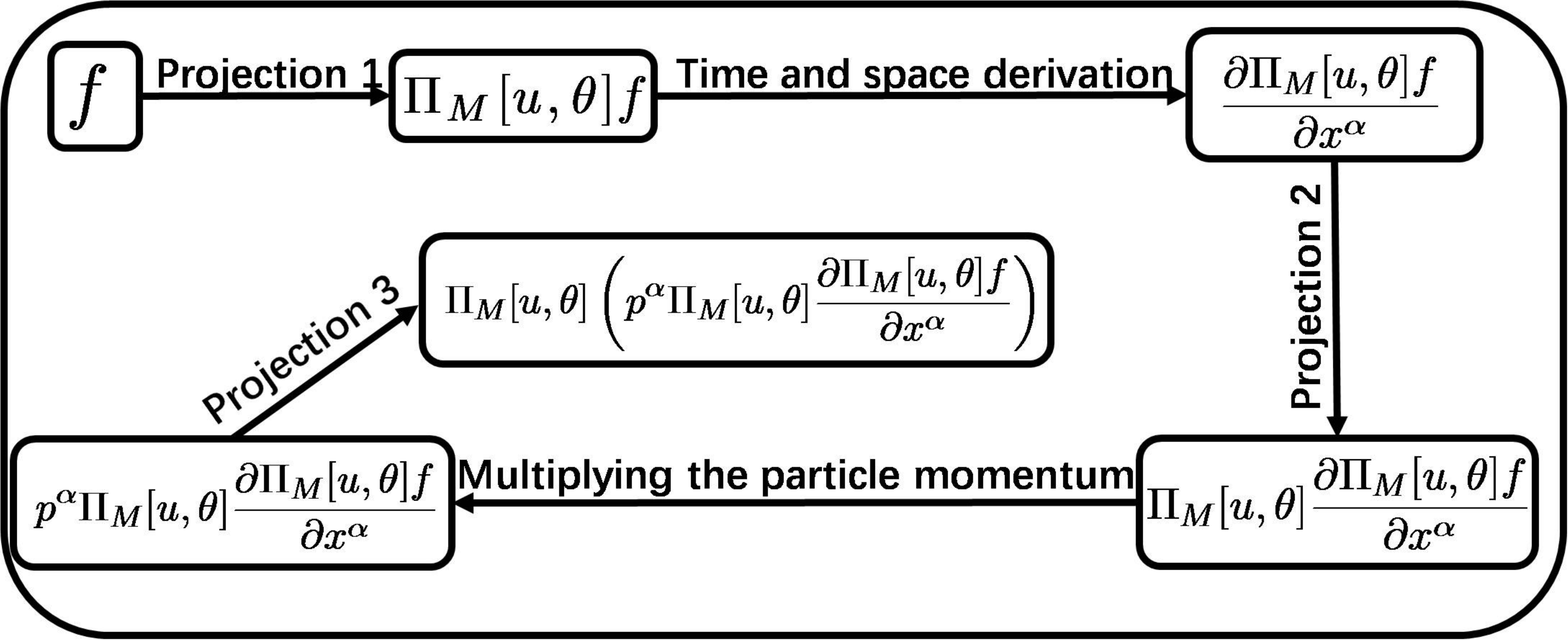}
  \caption{\small Schematic diagram of the moment method by the operator projection
  for the  special relativistic Boltzmann equation.}
    \label{fig:step}
\end{figure}
Referring to the schematic diagram shown in  Fig. \ref{fig:step},
the arbitrary order moment system for the  Boltzmann equation \eqref{eq:Boltz}
may be derived by the operator projection as follows:

\noindent
{\tt Step 1 (Projection 1):}
Projecting the distribution function $f$
  into space $\mathbb{H}^{g^{(0)}_{[\vec{u},\theta]}}_{M}$  by the operator $  \Pi_{M}[\vec{u},\theta]$ defined in \eqref{EQ-projection-aaaa2}.
\noindent

{\tt Step 2: }
Calculating the partial derivatives in time and space provides
    \begin{align}
    \nonumber \frac{\partial \Pi_{M}[\vec{u},\theta]f}{\partial s}&=\left[\frac{\partial \vec{\mathcal{P}}_{M}[\vec{u},\theta]}{\partial s},\vec{f}_{M}\right]_{M}
    + \left[\vec{\mathcal{P}}_{M}[\vec{u},\theta],\frac{\partial \vec{f}_{M}}{\partial s}\right]_{M}\\
    &=\left[\vec{C}_{M+1}\vec{P}_{M,M+1}^{T}\vec{\mathcal{P}}_{M}[\vec{u},\theta],\vec{P}_{M,M+1}^{T}\vec{f}_{M}\right]_{M+1}
    + \left[\vec{\mathcal{P}}_{M}[\vec{u},\theta],\frac{\partial \vec{f}_{M}}{\partial s}\right]_{M},\label{eq:diff}
    \end{align}
   for $s=x^{\alpha}$,
   where ${\vec{C}}_{M+1}$ is a square matrix of order $N_{M+1}$ and directly derived  with the aid of the derivative relations of the
    basis functions in Lemma \ref{lem:derive}.

\noindent
{\tt Step 3 (Projection 2):}
Projecting the partial derivatives in \eqref{eq:diff} into
  the  space $\mathbb{H}^{g^{(0)}_{[\vec{u},\theta]}}_{M}$ gives
    \begin{align}
    \nonumber \Pi_{M}[\vec{u},\theta]\frac{\partial \Pi_{M}[\vec{u},\theta]f}{\partial s}&=
    \left[\vec{\mathcal{P}}_{M}[\vec{u},\theta],\vec{C}_{M}^{T}\vec{f}_{M}\right]_{M}
    + \left[\vec{\mathcal{P}}_{M}[\vec{u},\theta],\frac{\partial \vec{f}_{M}}{\partial s}\right]_{M}\\
    \nonumber &=\left[\vec{\mathcal{P}}_{M}[\vec{u},\theta],\vec{C}_{M}^{T}\vec{D}_{M}^{W}\vec{W}_{M}+\frac{\partial \left(\vec{D}_{M}^{W}\vec{W}_{M}\right)}{\partial s}\right]_{M}\\
    &=:\left[\vec{\mathcal{P}}_{M}[\vec{u},\theta],\vec{D}_{M}\frac{\partial \vec{W}_{M}}{\partial s}\right]_{M},
    \label{eq:proj1}
    \end{align}
    where the $N_{M}$-by-$N_{M}$  matrix  $\vec{D}_{M}$  may be obtained from $\vec{C}_{M}$ and $\vec{D}_{M}^{W}$
    and is   of the following form
        \begin{equation}
        \label{eq:Dmatrix}
        \vec{D}_{M}=\begin{pmatrix}
        &\ &\ &\vec{D}_{2} &\ &\ & O &\ &\ \\
        &0 &* &* &0 &\ &\ &\ &\  \\
        &\vdots&\vdots&\vdots&\ &\  &\vec{I}_{N_{M}-N_{2}}&\ &\\\
        &0 &* &* &0 &\ &\  &\ &\
        \end{pmatrix}, \      M\geq3,\quad
\vec{D}_{2}=\begin{pmatrix}
\vec{D}_{5\times5}^{11} &\vec{D}_{5\times4}^{12} & \vec{O}\\
\vec{D}_{4\times5}^{21} & \vec{D}_{4\times4}^{22} & \vec{O}\\
\vec{O} & \vec{O} & \vec{I}_{5\times5}
\end{pmatrix},
        \end{equation}
and
        \begin{align*}
     \vec{D}_{1}=&\begin{pmatrix}
(c_{0}^{(0)})^{-1}& 0  & 0 & 0  & -n\zeta^2(c_{1}^{(0)})^{-2}(c_{0}^{(0)})^{-1} \\
0                  & 0  & 0 & 0  & n\zeta^2(c_{1}^{(0)})^{-1} \\
0                  & -n U^{0}n_{2}^{1}c_{0}^{(1)}&-n U^{0}n_{2}^{2}c_{0}^{(1)}&-n U^{0}n_{2}^{3}c_{0}^{(1)} & 0 \\
0                  & n U^{0}n_{3}^{1}c_{0}^{(1)}  & n U^{0}n_{3}^{2}c_{0}^{(1)} & n U^{0}n_{3}^{3}c_{0}^{(1)}   & 0 \\
0                  & -n U^{0}n_{1}^{1}c_{0}^{(1)} & -n U^{0}n_{1}^{2}c_{0}^{(1)}&  -n U^{0}n_{1}^{3}c_{0}^{(1)}& 0
    \end{pmatrix},\\
\vec{D}_{5\times5}^{11} =&\begin{pmatrix}
(c_{0}^{(0)})^{-1}& 0  & 0 & 0  & -n\zeta^2(c_{1}^{(0)})^{-2}(c_{0}^{(0)})^{-1} \\
0                  & 0  & 0 & 0  & n\zeta^2(c_{1}^{(0)})^{-1} \\
0                  & -n U^{0}n_{2}^{1}c_{0}^{(1)}&-n U^{0}n_{2}^{2}c_{0}^{(1)}&-n U^{0}n_{2}^{3}c_{0}^{(1)}+(1-\vec u^2)^{-\frac{3}{2}}c_{0}^{(1)}\tilde{f}_{0,-1}^{(1)} & 0 \\
0                  & n U^{0}n_{3}^{1}c_{0}^{(1)}  & n U^{0}n_{3}^{2}c_{0}^{(1)} & n U^{0}n_{3}^{3}c_{0}^{(1)} + (1-\vec u^2)^{-\frac{3}{2}}c_{0}^{(1)}\tilde{f}_{0,0}^{(1)} & 0 \\
0                  & -n U^{0}n_{1}^{1}c_{0}^{(1)} & -n U^{0}n_{1}^{2}c_{0}^{(1)}&  -n U^{0}n_{1}^{3}c_{0}^{(1)}+(1-\vec u^2)^{-\frac{3}{2}}c_{0}^{(1)}\tilde{f}_{0,1}^{(1)}& 0
    \end{pmatrix},\\
    \vec{D}_{4\times5}^{21} =&\begin{pmatrix}
0                  &-\frac{c_{2}^{(0)}\tilde{f}_{0,-1}^{(1)}(x_{1,2}^{(0)}+x_{2,2}^{(0)})}{(1-\vec u^2)^{\frac{3}{2}}}
&-\frac{c_{2}^{(0)}\tilde{f}_{0,0}^{(1)}(x_{1,2}^{(0)}+x_{2,2}^{(0)})}{(1-\vec u^2)^{\frac{3}{2}}}
&-\frac{c_{2}^{(0)}\tilde{f}_{0,1}^{(1)}(x_{1,2}^{(0)}+x_{2,2}^{(0)})}{(1-\vec u^2)^{\frac{3}{2}}} & 0\\
0                  &-U^{0}n_{2}^{1}c_{1}^{(1)}\Pi
&-U^{0}n_{2}^{2}c_{1}^{(1)}\Pi
&-U^{0}n_{2}^{3}c_{1}^{(1)}\Pi & 0\\
0                  &U^{0}n_{3}^{1}c_{1}^{(1)}\Pi
&U^{0}n_{3}^{2}c_{1}^{(1)}\Pi
&U^{0}n_{3}^{3}c_{1}^{(1)}\Pi & 0\\
0                  &-U^{0}n_{1}^{1}c_{1}^{(1)}\Pi
&-U^{0}n_{1}^{2}c_{1}^{(1)}\Pi
&-U^{0}n_{1}^{3}c_{1}^{(1)}\Pi & 0
    \end{pmatrix}.
\end{align*}

 \noindent
 {\tt Step 4:}
 Multiplying  \eqref{eq:proj1} by the particle momentum $(p^\alpha)$ yields
    \begin{align}
    \nonumber p^{\alpha}\Pi_{M}[\vec{u},\theta]\frac{\partial \Pi_{M}[\vec{u},\theta]f}{\partial x^{\alpha}}
    &:=[p^{\alpha}\vec{\mathcal{P}}_{M}[\vec{u},\theta],\vec{D}_{M}\frac{\partial \vec{W}_{M}}{\partial x^{\alpha}}]_{M}\\
    &=[\vec{M}_{M+1}^{\alpha}\vec{P}_{M,M+1}^{T}\vec{\mathcal{P}}_{M}[\vec{u},\theta],\vec{P}_{M,M+1}^{T}\vec{D}_{M}\frac{\partial \vec{W}_{M}}{\partial x^{\alpha}}]_{M+1}.    \label{eq:MtREC}
    \end{align}

\noindent
{\tt Step 5 (Projection 3):}
Projecting \eqref{eq:MtREC} into the space $\mathbb{H}^{g^{(0)}_{[\vec{u},\theta]}}_{M}$ gives
      \begin{equation}
    \label{eq:pMtREC}
    \Pi_{M}[\vec{u},\theta]\left(p^{\alpha}\Pi_{M}[\vec{u},\theta]\frac{\partial \Pi_{M}[\vec{u},\theta]f}{\partial x^{\alpha}}\right)=[\vec{\mathcal{P}}_{M}[\vec{u},\theta],\vec{M}_{M}^{\alpha}\vec{D}_{M}\frac{\partial \vec{W}_{M}}{\partial x^{\alpha}}]_{M}.
    \end{equation}

\noindent
{\tt Step 6:}
   Substituting them into  the special relativistic Boltzmann equation \eqref{eq:Boltz} derives
   the abstract form of the moment system
 \begin{equation}
 \Pi_{M}[\vec{u},\theta]\left(p^{\alpha}\Pi_{M}[\vec{u},\theta]\left( \frac{\partial \Pi_{M}[\vec{u},\theta]f}{\partial x^{\alpha}}\right)\right)
        =\Pi_{M}[\vec{u},\theta]Q(\Pi_{M}[\vec{u},\theta]f,\Pi_{M}[\vec{u},\theta]f),
        \label{eq:moment0000}
        \end{equation}
and then       matching the coefficients in front of the basis functions $\{ \tilde{P}_{k{,m}}^{(\ell)}[\vec{u},\theta]\}$
       leads to   an ``explicit'' matrix-vector form of the moment system
         \begin{equation}
          \vec{B}_{M}^{\alpha}\frac{\partial \vec{W}_{M}}{\partial x^{\alpha}}=\vec{S}(\vec{W}_{M}), \label{eq:moment1}
        \end{equation}
        which consists of  $N_{M}$ equations,
        where $ \vec{B}_{M}^{\alpha}=\vec{M}^{\alpha}_{M}\vec{D}_{M}$.
 For a general collision term $Q(f,f)$,  it is difficult to get an explicit expression of
    the source term  $\vec{S}(\vec{W}_{M})$ in \eqref{eq:moment1}.
  For the Anderson-Witting model  \eqref{eq:colAW},
  the right-hand side of \eqref{eq:moment0000} becomes
    \begin{align*}
    \frac{1}{\tau}&\Pi_{M}[\vec{u},\theta]Q(\Pi_{M}[\vec{u},\theta]f,\Pi_{M}[\vec{u},\theta]f)=-\frac{1}{\tau}\Pi_{M}[\vec{u},\theta]E\Pi_{M}[\vec{u},\theta]{(f-f^{(0)})}\\
    &=-\frac{1}{\tau}\Pi_{M}[\vec{u},\theta][\vec{P}_{M+1}^{p}\vec{A}_{M+1}^{0}(\vec{P}_{M+1}^{p})^{T}\vec{P}_{M,M+1}^{T}\vec{\mathcal{P}}_{M}[\vec{u},\theta],
     \vec{P}_{M,M+1}^{T}(\vec{f}_{M}-\vec{f}^{(0)}_{M})]_{M+1}\\
    &=-\frac{1}{\tau}[\vec{\mathcal{P}}_{M}[\vec{u},\theta],\vec{P}_{M}^{p}\vec{A}_{M}^{0}(\vec{P}_{M}^{p})^{T}{\tilde{\vec{D}}_{M}^{W}\vec{W}_{M}}]_{M},
    \end{align*}
which implies that    the source term  $\vec{S}(\vec{W}_{M})$  can be explicitly given by
    \begin{equation}
    \label{eq:colAWfinal}
    \vec{S}(\vec{W}_{M})=-\frac{1}{\tau}\vec{P}_{M}^{p}\vec{A}_{M}^{0}(\vec{P}_{M}^{p})^{T}{\tilde{\vec{D}}_{M}^{W}\vec{W}_{M}},
    \end{equation}
where  $\vec{f}^{(0)}_{M}=\left(n\sqrt{G(\zeta)-{4}\zeta^{-1}},0,\cdots,0\right)^{T}$,
and  the matrix $\tilde{\vec{D}}_{M}^{W}$ is the same as $\vec{D}_{M}^{W}$ except that the component of the upper left  corner
 is zero.
 It is worth noting that  the first {five} components
 of $\vec{S}(\vec{W}_{M})$ are zero due to  \eqref{eq:condition-222} and  \eqref{eq:condition}.

\begin{remark}
	When $M=1$,  it is obvious to show  $\Pi_M [\vec{u},\theta] f=f^{(0)}$ by the definition of $\vec{D}_{M}^{W}$
	such that the moment system \eqref{eq:moment1} gives
	the macroscopic RHD equations \eqref{eq:conser}.
\end{remark}

\subsection{The limited case}
\label{subsec:limit}
This section  discusses two important special cases: the non-relativistic
limit (small temperature and velocities) for a cool gas and the ultra-relativistic limit (zero rest mass of the particles).

\subsubsection{The non-relativistic limit}
In the {non}-relativistic limit,
one has that $\zeta\rightarrow+\infty$,
$K_{\nu}(\zeta)\rightarrow \sqrt{\frac{\pi}{2\zeta}}\exp(-\zeta)$,
$\vec{u}c^{-1}\rightarrow\vec0$, and $\vec{p}c^{-1}\rightarrow\vec0$.
Thus it holds
\begin{align*}
E=\sqrt{c^2+\vec{u}^2}\sqrt{m^2c^2+\vec{p}^2}-\vec{u}\cdot\vec{p}=&
(mc+\frac{\vec{p}^2}{2mc})(c+\frac{\vec{u}^2}{2c})-\vec{u}\cdot\vec{p}
\\
 \rightarrow & mc^2+\frac{m\vec{u}^2}{2}+\frac{\vec{p}^2}{2m}-\vec{u}\cdot\vec{p}
= mc^2+\frac{m|\vec\xi-\vec u|^2}{2},
\end{align*}
where  $\vec\xi=\vec{p}m^{-1}$,
and the Maxwell-J$\ddot{\text{u}}$ttner distribution
 in \eqref{eq:equm1} reduces to
  \[
 f^{(0)}=ng^{(0)},\quad g^{(0)}=
 \frac{1}{(2\pi mkT)^{\frac{3}{2}}}\exp\left(-\frac{m|\vec\xi-\vec u|^2}{2kT}\right),
 \]
which means that the weight function in \eqref{eq:equm1-zzzzz} reduces to
 \[
 g^{(0)}_{[\vec{u},\theta]}=\frac{1}{(2\pi \theta)^{\frac{3}{2}}}\exp\left(-\frac{|\vec\xi-\vec u|^2}{2\theta}\right).
 \]
 Corresponding to such limited weight function $g^{(0)}_{[\vec{u},\theta]}$,
 the previous standard orthogonal bases of the weighted polynomial spaces and
 the moment system will respectively reduce to those in \cite{MR:2014} for the non-relativistic case.

\subsubsection{The ultra-relativistic limit}
In the ultra-relativistic limit ($m\rightarrow0$), one has
$K_2(\zeta)\rightarrow \frac{2}{\zeta^2}$, $p^{0}\rightarrow \vec{p}$, and the Maxwell-J$\ddot{\text{u}}$ttner distribution
 in \eqref{eq:equm1} reduces to
 \[
 f^{(0)}=ng^{(0)},\quad g^{(0)}=
 \frac{ c^3}{8\pi k^3T^3}\exp\left(-\frac{U_{\alpha}p^{\alpha}}{kT}\right),
 \]
 so that the weight function in \eqref{eq:equm1-zzzzz} tends to
 \[
 g^{(0)}_{[\vec{u},\theta]}= \frac{ 1}{8\pi \theta^3}\exp\left(-\frac{E}{\theta}\right).
 \]
For this weight function $g^{(0)}_{[\vec{u},\theta]}$,
 the previous standard orthogonal bases of the weighted polynomial spaces
become the followings.

 \begin{thm}
 The basis of the weighted polynomial space $\mathbb{H}_{M}^{g^{(0)}_{[\vec u,\theta]}}$ in
 \eqref{eq:basis} reduces to
 \begin{align}
\label{eq:basisultra}
\nonumber\vec{\mathcal{P}}_{\infty}[\vec{u},\theta]:=&(\tilde{P}_{0,0}^{(0)}[\vec{u},\theta],\cdots,\tilde{P}_{M,0}^{(0)}[\vec{u},\theta],
\tilde{P}_{M-1,-1}^{(1)}[\vec{u},\theta],
\tilde{P}_{M-1,0}^{(1)}[\vec{u},\theta],\\
&\tilde{P}_{M-1,1}^{(1)}[\vec{u},\theta],\cdots,\tilde{P}_{0,-M}^{(M)},\cdots,\tilde{P}_{0,M}^{(M)},\cdots)^{T},
\\
\label{eq:huapultra}
\nonumber\vec{\mathcal{P}}_{M}[\vec{u},\theta]:=&(\tilde{P}_{0,0}^{(0)}[\vec{u},\theta],\tilde{P}_{1,0}^{(0)}[\vec{u},\theta],\tilde{P}_{0,-1}^{(1)}[\vec{u},\theta],
\tilde{P}_{0,0}^{(1)}[\vec{u},\theta],\tilde{P}_{0,1}^{(1)}[\vec{u},\theta],
\cdots,\tilde{P}_{M,0}^{(0)}[\vec{u},\theta],\\
&\tilde{P}_{M-1,-1}^{(1)}[\vec{u},\theta],
\tilde{P}_{M-1,0}^{(1)}[\vec{u},\theta],\tilde{P}_{M-1,1}^{(1)}[\vec{u},\theta],\cdots,\tilde{P}_{0,-M}^{(M)},\cdots,\tilde{P}_{0,M}^{(M)} )^{T},
\end{align}
 where $\tilde{P}_{k,m}^{(\ell)}[\vec{u},\theta]:=g^{(0)}_{[\vec{u},\theta]}P_{k}^{(\ell)}\tilde{Y}_{\ell,m}$,
$\tilde{Y}_{\ell,m}:=E^{\ell}Y_{\ell,m}$, $|m|\leq\ell$, $\ell\in\mathbb{N}$,
and $\{P_{k}^{(\ell)}\}$ reduce to the generalized Laguerre polynomials of the variable $E/\theta$ with
the parameter $2\ell+1$.
 \end{thm}
\begin{proof}
Due to the ultra-relativistic limit of \eqref{eq:equm1-zzzzz},
the weight function in \eqref{eq:omegal} reduces to
\[
  \omega^{(\ell)}(x;\zeta)=\frac{\zeta^3 x^{2\ell+1}}{(2\ell+1)}\exp\left(-\zeta x\right),\ell\in\mathbb{N}.
\]
If using the notation $\tilde{x}=\zeta x$, then one has
\[
P_{k}^{(\ell)}(x,\zeta)=C_{\ell,k}\bar{P}_{k}^{(2\ell+1)}(\tilde{x}), \ x>0,
\]
where $\{\bar{P}_{k}^{(2\ell+1)}\}$ are the
generalized Laguerre polynomials with the parameter $2\ell+1$,
satisfying
\[
\int_{0}^{\infty}\tilde{x}^{2\ell+1}\exp(-\tilde{x})\bar{P}_{k}^{(2\ell+1)}
\bar{P}_{m}^{(2\ell+1)}d\tilde{x}=\frac{\Gamma(k+2\ell+2)}{k!}\delta_{k,m},
\]
and
\[
C_{\ell,k}=\sqrt{\frac{(2\ell+1)k!}{\zeta^2\Gamma(k+2\ell+2)}}.
\]
Thus, it holds
\[
\int_{0}^{\infty}P_{k}^{(\ell)}(x,\zeta)P_{m}^{(\ell)}(x,\zeta)\omega^{(\ell)}(x;\zeta)dx
=\delta_{k,m}.
\]
 The proof is completed.
\qed\end{proof}
Based on the above discussion, one can  derive the moment system in the case of
ultra-relativistic  limit by  the procedure in Sec. \ref{subsec:deduction}.

\subsection{Quasi-1D case}
\label{subsec:reduction}
In practice, only one spatial coordinate $x\in\mathbb{R}$ may be considered for some problems, e.g. the shock tube or spherical symmetric problems, but  the particle momentum   should still be 3D, i.e. $\vec{p}\in\mathbb{R}^{3}$.

Take $x=x_3\in\mathbb{R}$ as an example,
and assume that the distribution $f(\vec{x},\vec{p},t)$
is symmetric in $p^{1}$ and $p^{2}$ directions, and constant in $x_{1}$ and $x_{2}$,
that is,
$
f(\vec{x},\vec{p},t)=f(x_{3},\vec{p},t)$, $f(\vec{x},-p^{1},p^{2},p^{3},t)=f(\vec{x},p^{1},p^{2},p^{3},t)$,
$f(\vec{x},p^{1},-p^{2},p^{3},t)=f(\vec{x},p^{1},p^{2},p^{3},t)$.
The energy-momentum tensor becomes
\[
T^{\alpha\beta}=\begin{pmatrix}
\int_{\mathbb{R}^{3}}p^{0}p^{0}f\frac{d^{3}\vec{p}}{p^{0}}&0&0&\int_{\mathbb{R}^{3}}p^{0}p^{3}f\frac{d^{3}\vec{p}}{p^{0}}\\
0&\int_{\mathbb{R}^{3}}p^{1}p^{1}f\frac{d^{3}\vec{p}}{p^{0}}&0&0\\
0&0&\int_{\mathbb{R}^{3}}p^{2}p^{2}f\frac{d^{3}\vec{p}}{p^{0}}&0\\
\int_{\mathbb{R}^{3}}p^{0}p^{3}f\frac{d^{3}\vec{p}}{p^{0}}&0&0&\int_{\mathbb{R}^{3}}p^{3}p^{3}f\frac{d^{3}\vec{p}}{p^{0}}
\end{pmatrix},
\]
thus one has  $U^{1}=U^{2}=0$, thanks to  Theorem \ref{thm:admissible}.

According to \eqref{eq:ni} and \eqref{eq:yphi}, one gets
\[
n_{1}^{\alpha}=(0,1,0,0),\quad n_{2}^{\alpha}=(0,0,1,0),\quad n_{3}^{\alpha}=(U^{3},0,0,U^{0}),
\]
\[
y=(E^2-1)^{-\frac{1}{2}}(U^{0}p^{3}-U^{3}p^{0}), \ \cos\phi=\frac{p^1}{\sqrt{1-y^2}},\
\sin\phi=\frac{p^2}{\sqrt{1-y^2}},
\]
\[
p^{0}=U^{0}E+\sqrt{E^2-1}U^{3}y,\ p^{1}=\sqrt{E^2-1}\sqrt{1-y^2}\cos\phi,
\ p^{2}=\sqrt{E^2-1}\sqrt{1-y^2}\sin\phi.
\]
Using  Lemma \ref{lem:derive} yields
\begin{align*}
&\frac{\partial \tilde{P}_{k,m}^{(\ell)}[\vec{u},\theta]}{\partial s}
=-\frac{\partial \theta}{\partial s}\zeta^2\left(\frac{1}{2}\left(G(\zeta)-\zeta^{-1}-b_{k}^{(\ell)}\right)\tilde{P}_{k,m}^{(\ell)}
-a_{k}^{(\ell)}\tilde{P}_{k+1,m}^{(\ell)}\right)\\
&-\frac{\partial u_{3}}{\partial s}U^{0}g^{(0)}_{[\vec{u},\theta]}\left[\left(\frac{2\ell+1}{2\ell+3}\frac{k}{\tilde{p}_{k-1}^{(\ell+1)}}-\zeta q_{k-1}^{(\ell+1)}\right)
n_{3}^{i}h_{\ell+1,m}\tilde{P}_{k-1,m}^{(\ell+1)}
-\zeta p_{k}^{(\ell+1)}n_{3}^{i}h_{\ell+1,m}\tilde{P}_{k,m}^{(\ell+1)}\right.\\
&-\left.\frac{2\ell+1}{2\ell-1}\left((k+2\ell+1)\tilde{p}_{k}^{(\ell)}-\zeta q_{k}^{(\ell)}\right)
n_{3}^{i}h_{\ell,m}\tilde{P}_{k+1,m}^{(\ell-1)}
+\frac{2\ell+1}{2\ell-1}\zeta r_{k+1}^{(\ell)}
n_{3}^{i}h_{\ell,m}\tilde{P}_{k+2,m}^{(\ell-1)}\right],
\end{align*}
and $d\phi=0$.
Further using  Lemmas \ref{lem:rec} and \eqref{eq:Dmatrix} derives
\[
\vec{M}_{M}^{0}=U^{0}\vec{P}_{M}^{p}\vec{A}_{M}^{0}(\vec{P}_{M}^{p})^{T}+U^{3}\vec{P}_{M}^{p}\vec{A}_{M}^{3}(\vec{P}_{M}^{p})^{T}, \
\vec{M}_{M}^{3}=U^{3}\vec{P}_{M}^{p}\vec{A}_{M}^{0}(\vec{P}_{M}^{p})^{T}+U^{0}\vec{P}_{M}^{p}\vec{A}_{M}^{3}(\vec{P}_{M}^{p})^{T},
\]
and the second to fourth columns of $\vec{D}_{M}$ have only one nonzero component in each row, where the column for $u^{3}$  only corresponds to
the row for $\tilde{P}_{k,0}^{(\ell)}$.

On the other hand,  the physical quantities of practical interest  $(n,u_{3},\theta,\Pi,n^{3})$ can be expressed as
$(f_{0,0}^{(0)},f_{1,0}^{(0)},f_{0,0}^{(1)},f_{2,0}^{(0)})$ due to \eqref{eq:constraint},
and the expansion of $f$  can be reduced as follows
\[
\acute{\Pi}_{M}[\vec{u},\theta]f
=\sum_{\ell=0}^{M}\sum_{i=0}^{M-\ell}f_{i,0}^{(\ell)}\tilde{P}_{i,0}^{(\ell)}[\vec{u},\theta], \
\vec{\acute{W}_{M}}=(n,u,\theta,\Pi,\tilde{f}_{0,0}^{(1)},\cdots,f_{M,0}^{(0)},
f_{M-1,0}^{(1)},f_{0,0}^{(M)} )^{T},
\]
where $u=u^3$.
Thus if eliminating corresponding rows and columns of $\vec{D}_{M}$, $\vec{M}^{\alpha}$ and $\vec{S}(\vec{W}_{M})$ and then denoting them by the notations
$\vec{\acute{D}}_{M}$, $\vec{\acute{M}}^{\alpha}$, and $\vec{\acute{S}}(\vec{\acute{W}}_{M})$,
then the moment system \eqref{eq:moment1} becomes
\begin{equation}
  \vec{\acute{B}}_{M}^{0}\frac{\partial \vec{\acute{W}}_{M}}{\partial t}+
\vec{\acute{B}}_{M}^{3}\frac{\partial \vec{\acute{W}}_{M}}{\partial x}=\vec{\acute{S}}(\vec{W}_{M}), \label{eq:moment1D}
\end{equation}
where $\vec{\acute{B}}_{M}^{\alpha}=\vec{\acute{M}}^{\alpha}\vec{\acute{D}}_{M}$ and $\acute{N}_{M}=\frac{(M+1)(M+2)}{2}$.

\section{Properties of the moment system}
\label{sec:prop}
This section studies some properties of  moment system \eqref{eq:moment0000} or \eqref{eq:moment1}.

\subsection{Hyperbolicity}
\label{subsec:Hyper}	
In order to prove the hyperbolicity of the moment system \eqref{eq:moment1}, one has to verify that $\vec{B}_{M}^{0}$ to be invertible and  $\vec{B}_{M}:=(\vec{B}_{M}^{0})^{-1}\sum_{i=1}^{3}\tilde{n}_{i}\vec{B}_{M}^{i}$ to be real diagonalizable for a given unit vector $\vec{\tilde{n}}=(\tilde{n}_{1},\tilde{n}_{2},\tilde{n}_{3})$. In the following,  assume that
the first {five} components of  $\vec{W}_{M}$ satisfy three inequalities in \eqref{eq:positivity}.

\begin{lemma}
 \label{lem:Dinv}
The matrix $\vec{D}_{M}$ is invertible for $M\geq1$.
 \end{lemma}
 \begin{proof}
 It is obvious that for $M=1$, the matrix ${\vec{D}}_{M}$ is invertible
 because $\det(\vec{D}_M)=n^{4}\zeta^2(c_0^{(1)})^{3}(c_0^{(0)}c_1^{(0)}(U^{0}+1))^{-1}((U^{0})^2+U^{0}+1)(U^{0})^{3}<0$.
{For $M\geq2$,} according to the form of $\vec{D}_M$ in Section \ref{subsec:deduction},
 one has
\[
\det(\vec{D}_{M})=\det(\vec{D}_{2})=\det(\vec{D}_{2})
=3\zeta^3c_{2}^{(0)}c_{1}^{(1)}(x_{1,2}^{(0)}+x_{2,2}^{(0)})({n} G(\zeta)+\Pi){n}c_{1}^{(0)}c_{0}^{(0)}(U^{0})^6.
\]
Using  Lemma \ref{lem:admissible1} gives
\[
\det(\vec{D}_{M})>\zeta^3c_{2}^{(0)}c_{1}^{(1)}(x_{1,2}^{(0)}+x_{2,2}^{(0)}){n}^2( G(\zeta)-\zeta^{-1})c_{1}^{(0)}c_{0}^{(0)}{(U^{0})^6}>0.
\]
The proof is completed.
\qed\end{proof}

\begin{thm}[Hyperbolicity]
\label{thm:hyper}
The moment system  \eqref{eq:moment1} is strictly hyperbolic and the spectral radius of $\vec{B}_{M}$ is less than one.
\end{thm}
\begin{proof}
Due to Lemma \ref{lem:Dinv}, $\vec{D}_{M}$ is invertible.
According to \eqref{eq:MtMx}, one knows that
\[
\vec{M}_{M}^{\alpha}=<p^{\alpha}\vec{\mathcal{P}}_{M}[\vec{u},\theta],\vec{\mathcal{P}}_{M}[\vec{u},\theta]^{T}>_{g^{(0)}_{[\vec{u},\theta]}}.
\]
It is obvious that $\vec{M}_{M}^{\alpha}$ is symmetric, and for a given $N_M$-dimensional nonzero vector $\vec{q}$,
\begin{align*}
\vec{q}^{T}\vec{M}_{M}^{0}\vec{q}&=
<p^{\alpha}\vec{q}^{T}\vec{\mathcal{P}}_{M}[\vec{u},\theta],\vec{\mathcal{P}}_{M}[\vec{u},\theta]^{T}\vec{q}>_{g^{(0)}_{[\vec{u},\theta]}}\\
&=<p^{\alpha}\vec{q}^{T}\vec{\mathcal{P}}_{M}[\vec{u},\theta],\vec{q}^{T}\vec{\mathcal{P}}_{M}[\vec{u},\theta]>_{g^{(0)}_{[\vec{u},\theta]}}>0,
\end{align*}
thus $\vec{M}_{M}^{0}$ is positive definite. Therefore $\vec{B}_{M}^{0}$ is invertible.

Furthermore, one has
\begin{align*}
\vec{B}_{M}&=\vec{D}_{M}^{-1}(\vec{M}_{M}^{0})^{-1}\sum_{i=1}^{3}\tilde{n}_{i}\vec{M}_{M}^{i}\vec{D}_{M}\\
&=\left((\vec{M}_{M}^{0})^{\frac{1}{2}}\vec{D}_{M}\right)^{-1}
(\vec{M}_{M}^{0})^{-\frac{1}{2}}\sum_{i=1}^{3}\tilde{n}_{i}\vec{M}_{M}^{i}
(\vec{M}_{M}^{0})^{-\frac{1}{2}}\left((\vec{M}_{M}^{0})^{\frac{1}{2}}\vec{D}_{M}\right).
\end{align*}
Thus the matrix
$\vec{B}_{M}$ is real diagonalizable since $(\vec{M}_{M}^{0})^{-\frac{1}{2}}\sum_{i=1}^{3}\tilde{n}_{i}\vec{M}_{M}^{i}(\vec{M}_{M}^{0})^{-\frac{1}{2}}$ is symmetric,
and the spectral radius of $\vec{B}_{M}$ is equal to that of
$(\vec{M}_{M}^{0})^{-1}\sum_{i=1}^{3}\tilde{n}_{i}\vec{M}_{M}^{i}$ and
satisfies
\[
\lambda \vec{M}_{M}^{0} -
 \sum_{i=1}^{3}\tilde{n}_{i}\vec{M}_{M}^{i}=
<(\lambda p^{0}-\sum_{i=1}^{3}p^{i})\vec{\mathcal{P}}_{M}[\vec{u},\theta],\vec{\mathcal{P}}_{M}[\vec{u},\theta]^{T}>_{g^{(0)}_{[\vec{u},\theta]}},
\]
for $|\lambda|>1$. Moreover, for a given $N_M$-dimensional nonzero vector $\vec{q}$, one has
\begin{align*}
\vec{q}^{T}&\left(\lambda \vec{M}_{M}^{0} - \sum_{i=1}^{3}\tilde{n}_{i}\vec{M}_{M}^{i}\right)\vec{q}=
<p^{\alpha}\vec{q}^{T}\vec{\mathcal{P}}_{M}[\vec{u},\theta],\vec{\mathcal{P}}_{M}[\vec{u},\theta]^{T}\vec{q}>_{g^{(0)}_{[\vec{u},\theta]}}\\
=&\left<(\lambda p^{0}-\sum_{i=1}^{3}p^{i})\vec{q}^{T}\vec{\mathcal{P}}_{M}[\vec{u},\theta],\vec{q}^{T}
\vec{\mathcal{P}}_{M}[\vec{u},\theta]\right>_{g^{(0)}_{[\vec{u},\theta]}}
   \begin{cases}
     >0, &  \lambda>1, \\
     <0, & \lambda<-1,
   \end{cases}
\end{align*}
thus the matrix
\[
\lambda \vec{M}_{M}^{0} - \sum_{i=1}^{3}\tilde{n}_{i}\vec{M}_{M}^{i}
\]
is positive (resp. negative) definite for $\lambda>1$ (resp. $\lambda<-1$).
Therefore the spectral radius of $(\vec{M}_{M}^{0})^{-1}\sum_{i=1}^{3}\tilde{n}_{i}\vec{M}_{M}^{i}$ is less than one, and the proof is completed.
\qed\end{proof}
\subsection{Linear stability}
\label{subsec:LinearS}
It is obvious that the moment system \eqref{eq:moment1} has
 the local equilibrium solution $\vec{W}_{M}^{(0)}=(n_{0},\vec{u}_{0},\theta_{0},0,\cdots,0)^{T}$, where $n_{0}$, $\vec{u}_{0}$,
 and $\theta_{0}$ satisfy three inequalities in \eqref{eq:positivity}.
 Similar to the procedure in \cite{LS:2016,Kuang:2017},   the moment system \eqref{eq:moment1}--\eqref{eq:colAWfinal} is  linearized at $\vec{W}_{M}^{(0)}$.
 If assuming that $\vec{W}_{M}=\vec{W}_{M}^{(0)}(1+\vec{\bar{W}}_{M})$ and each component  of $\vec{\bar{W}}_{M}$ is small, then the linearized moment system becomes
\begin{equation}
\label{eq:LME}
\vec{B}_{M}^{\alpha}\big|_{\vec{W}_{M}^{(0)}}\frac{\partial \bar{\vec{W}}_{M}}{\partial x^{\alpha}}=\vec{Q}_{M}\big|_{\vec{W}_{M}^{0}}\bar{\vec{W}}_{M},
\end{equation}
where
\[
\vec{Q}_{M}=-\frac{1}{\tau}U_{\alpha}\vec{M}_{M}^{\alpha}\tilde{\vec{D}}_{M}^{W}.
\]

Assume that the solution $\bar{\vec{W}}_{M}$ is of the plane wave form
\[
\bar{\vec{W}}_{M}=\tilde{\vec{W}}_{M}\exp(i(\omega t-\vec{k}\vec{x})),
\]
where $i$ is the imaginary unit, $\tilde{\vec{W}}_{M}$ is the nonzero amplitude, and $\omega$ and ${ \vec k}$ denote
  the frequency and  wave number, respectively.
Substituting the above plane wave into \eqref{eq:LME} gives
\[
\left(i\omega\vec{B}_{M}^{0}-ik_{j}\vec{B}_{M}^{j}-\vec{Q}_{M}\right)\big|_{\vec{W}_{M}^{(0)}}\tilde{\vec{W}}_{M}=0.
\]
Because the amplitude $\tilde{\vec{W}}_{M}$ is nonzero,  the above coefficient matrix  is singular, i.e.
\begin{equation}
\label{eq:LS}
\det\left(i\omega\vec{B}_{M}^{0}-ik_{j}\vec{B}_{M}^{j}-\vec{Q}_{M}\right)\big|_{\vec{W}_{M}^{(0)}}=0,
\end{equation}
which implies   the dispersion relation between $\omega$ and ${\vec k}$.

\begin{thm}
\label{thm:LS}
The  moment system \eqref{eq:moment1} with the source term \eqref{eq:colAWfinal}
is linearly stable in time at the local equilibrium, that is,
the linearized moment system \eqref{eq:LME} is stable in the sense that
$Im(\omega(\vec k))\geq0$ for each $\vec k\in\mathbb{R}^{3}$.
\end{thm}
\begin{proof}
 Because
  the matrix $\vec{D}_M$ in \eqref{eq:Dmatrix} at  $\vec{W}_{M}=\vec{W}_{M}^{(0)}$
{can} be reformed as follows
\[
\vec{D}_{M}=\begin{pmatrix}
\vec{D}_{5\times5}^{11} &\vec{D}_{5\times4}^{12} & \vec{O}\\
\vec{O} & \vec{D}_{4\times4}^{22} & \vec{O}\\
\vec{O} & \vec{O} & \vec{I}_{N_{M}-9}
\end{pmatrix},
\]
and its inverse  is given by
\[
\vec{D}_{M}^{-1}=\begin{pmatrix}
\left(\vec{D}_{5\times5}^{11}\right)^{-1} & -\left(\vec{D}_{5\times5}^{11}\right)^{-1}\vec{D}_{5\times4}^{12}\left(\vec{D}_{4\times4}^{22}\right)^{-1} & \vec{O}\\
\vec{O} & \left(\vec{D}_{4\times4}^{22}\right)^{-1} & \vec{O}\\
\vec{O} & \vec{O} & \vec{I}_{N_{M}-9}
\end{pmatrix},
\]
as well as
\[
   \tilde{\vec{D}}_{M}^{W}=\begin{pmatrix}
   \vec{O} &\vec{D}_{5\times4}^{12} & \vec{O}\\
   \vec{O} & \vec{D}_{4\times4}^{22} & \vec{O}\\
   \vec{O} & \vec{O} & \vec{I}_{N_{M}-9}
   \end{pmatrix},
   \]
 the product of $\tilde{\vec{D}}_{M}^{W}$ and $\vec{D}_{M}^{-1}$
 is of the  form
\[
\tilde{\vec{D}}_{M}^{W}\vec{D}_{M}^{-1}=\begin{pmatrix}
\vec{O}_{5\times5}& \vec{D}_{5\times4}^{12}\left(\vec{D}_{4\times4}^{22}\right)^{-1} & \vec{O}_{5\times(N_{M}-9)}\\
\vec{O}_{4\times5}& \vec{I}_{4}& \vec{O}_{4\times(N_{M}-9)}\\
\vec{O}_{(N_{M}-9)\times3}&\vec{O}_{(N_{M}-9)\times2}&\vec{I}_{N_{M}-9}
\end{pmatrix},
\]
where $\vec{D}_{5\times5}^{11}$ is the $5\times 5$ subblock of the $9\times 9$ upper left subblock of $\vec{D}_{2}$ in the upper left corner,
$\vec{D}_{5\times4}^{12}$ denotes the
$5\times 4$ subblock of the $9\times 9$ upper left subblock of $\vec{D}_{2}$ in the upper right corner,
and $\vec{D}_{4\times4}^{22}$ is $4\times 4$ subblock of the $9\times 9$ upper left subblock of $\vec{D}_{2}$ in the bottom right corner.
{It is obvious that each eigenvalue of $-\tilde{\vec{D}}_{M}^{W}\vec{D}_{M}^{-1}$ is non-positive, so does the matrix
\[
\bar{\vec{Q}}_{M}:=-\frac{1}{\tau}\left(U_{\alpha}\vec{M}_{M}^{\alpha}\right)^{\frac{1}{2}}\tilde{\vec{D}}_{M}^{W}\vec{D}_{M}^{-1}\left(U_{\alpha}\vec{M}_{M}^{\alpha}\right)^{-\frac{1}{2}}.
\]

The matrix $ U^{\alpha}\vec{M}_{M}^{\alpha}$ can be written as follows
\[
\begin{pmatrix}
\vec{M}_{5\times5}^{11}&\vec{M}_{5\times4}^{12}&\vec{O}_{5\times(N_M-9)}\\
(\vec{M}_{5\times4}^{12})^{T}&\vec{M}_{4\times4}^{22}&\vec{M}_{4\times(N_M-9)}^{23}\\
\vec{O}_{(N_M-9)\times5}&(\vec{M}_{4\times(N_M-9)}^{23})^{T}&\vec{M}_{(N_M-9)\times(N_M-9)}^{33}
\end{pmatrix},
\]
where $\vec{M}_{5\times5}^{11}$ is the $5\times 5$ subblock of $\vec{P}_{2}^{p}\vec{A}_{2}^{0}(\vec{P}_{2}^{p})^{T}$ in the upper left corner, $\vec{M}_{5\times4}^{12}$ denotes the
$5\times 4$ subblock of $\vec{P}_{2}^{p}\vec{A}_{2}^{0}(\vec{P}_{2}^{p})^{T}$ in the upper right corner,
and $\vec{M}_{4\times4}^{22}$ is $4\times 4$ subblock of $\vec{P}_{2}^{p}\vec{A}_{2}^{0}(\vec{P}_{2}^{p})^{T}$ in the bottom right corner,
the rest subblocks form the $(N_M-9)\times (N_M-9)$ bottom right corner of $\vec{P}_{2}^{p}\vec{A}_{M}^{0}(\vec{P}_{M}^{p})^{T}$.
Thus one has
\[
\vec{M}_{D}:=(\vec{M}_{5\times4}^{12})^{T}\vec{D}_{5\times4}^{12}\left(\vec{D}_{4\times4}^{22}\right)^{-1}
=-\left(\vec{D}_{5\times4}^{12}\left(\vec{D}_{4\times4}^{22}\right)^{-1}\right)^{T}\vec{M}_{5\times5}^{11}
\left(\vec{D}_{5\times4}^{12}\left(\vec{D}_{4\times4}^{22}\right)^{-1}\right),
\]
which is symmetric because
$
\vec{M}_{5\times5}^{11}\vec{D}_{5\times4}^{12}\left(\vec{D}_{4\times4}^{22}\right)^{-1}+
\vec{M}_{5\times4}^{12}=\vec{O}_{5\times4}.
$

On  the other hands, because the first {five} components of $\vec{S}(\vec{W}_M)$ are zero,
all elements in  the first five rows and  the first five columns of the matrix
\[
\vec{Q}_{M}=-\frac{1}{\tau}U_{\alpha}\vec{M}_{M}^{\alpha}\tilde{\vec{D}}_{M}^{W}\vec{D}_{M}^{-1},
\]
are zero, and the matrix $\vec{Q}_{M}$ is of form
\[
\vec{Q}_{M}=-\frac{1}{\tau}
\begin{pmatrix}
\vec{O}_{5\times5}&\vec{O}_{5\times4}&\vec{O}_{5\times(N_{M}-9)}\\
\vec{O}_{4\times5}&\vec{M}_{4\times4}^{22}+\vec{M}_{D}&\vec{M}_{4\times(N_{M}-9)}^{23}\\
\vec{O}_{(N_{M}-9)\times5}&(\vec{M}_{4\times(N_{M}-9)}^{23})^{T}&\vec{M}_{(N_{M}-9)\times(N_{M}-9)}^{33}
\end{pmatrix}.
\]
Hence the matrix $\vec{Q}_{M}$  is symmetric.
It is obvious that $\vec{Q}_{M}$ is congruent with $\bar{\vec{Q}}_{M}$,
so it is negative semi-definite.

Because both matrices $\vec{D}_{M}$ and $\vec{M}_{M}^{0}$ are invertible and
$\vec{M}_{M}^{0}$ is positive definite,
\eqref{eq:LS} is equivalent to
\begin{align}
 \det\left(i\omega \vec{I}-i\vec{M}_{M}-\hat{\vec{Q}}_{M}\right)=0,
\end{align}
where
\begin{align*}
\hat{\vec{Q}}_{M}:=& \left(\vec{M}_{M}^{0}\right)^{-\frac{1}{2}}\vec{Q}_{M}\left(\vec{M}_{M}^{0}\right)^{-\frac{1}{2}},
\end{align*}
and
$$\vec{M}_{M}:=\left(\vec{M}_{M}^{0}\right)^{-\frac{1}{2}}k_{i}\vec{M}_{M}^{i}
\left(\vec{M}_{M}^{0}\right)^{-\frac{1}{2}}.$$
It means that the matrix $\hat{\vec{Q}}_{M}$ is congruent with $\vec{Q}_{M}$ and negative semi-definite, and $\vec{M}_{M}$ is symmetric. Using Lemmas 1 and  2 in \cite{LS:2016} completes the
proof.}
\qed\end{proof}

%

\subsection{Lorentz covariance for  quasi 1D case}
\label{subsec:Lorentz}
In physics, the  Lorentz covariance is a key property of space-time following from the special theory of relativity, see e.g. \cite{SR:1961}.
This section studies the Lorentz covariance
of the  quasi 1D  moment system  
in Section \ref{subsec:reduction}, where $x=x_3$.



Some   Lorentz covariant quantities are first pointed out below.

\begin{lemma}
	\label{lem:lorentz}
(i)	Each component of $\vec{\acute{D}}_{M}^{u} d\vec{\acute{W}}_{M}$ is  Lorentz invariant,
where
	$d\vec{\acute{W}}_{M}$ denotes the total differential of $\vec{\acute{W}}_{M}$
and $\vec{\acute{D}}_{M}^{u}:={\rm diag}\{1,U^{0}n_{3}^{3},1,\cdots,1\}$.
(ii)	The matrices $\vec{\acute{A}}_{M}^{0}$, $\vec{\acute{A}}_{M}^{3}$ and source term $\vec{\acute{S}}(\vec{\acute{W}}_{M})$  defined in \eqref{eq:colAWfinal}  are Lorentz invariant.
\end{lemma}
\begin{proof}
(i) Under the given Lorentz boost ($x$ direction)
\[
t'=\gamma(v) (t-vx),\ x'=\gamma(v)(x-vt),\ \gamma(v)=(1-v^2)^{-\frac{1}{2}},
\]
where $v$ is the relative velocity between frames in the $x$-direction,
one has
	\begin{align*}
	&(p^{0})'=\gamma(v)(p^{0}-p^{3}v),\quad (p^{3})'=\gamma(v)(p^{3}-p^{0}v), \\
	&(U^{0})'=\gamma(v)(U^{0}-U^{3}v),\quad (U^{3})'=\gamma(v)(U^{3}-U^{0}v).
	\end{align*}
Thus one further gets
	\[
	E'= (U^{0})'(p^{0})'-(U^{3})'(p^{ 3})'=U^{0}p^{0}-U^{3}p^{3}=E,
	\]
	and
	\begin{align*}
	y'=& ((E^2-1)^{-\frac{1}{2}}(U^{0}p^{3}-U^{3}p^{0}))'
=(E^2-1)^{-\frac{1}{2}}((U^{0})'(p^{3})'-(U^{3})'(p^{0})')\\
 =& (E^2-1)^{-\frac{1}{2}}(U^{0}p^{3}-U^{3}p^{0})=y,
	\\
	\left(\frac{d^{3}\vec{p}}{p^{0}}\right)'=& \frac{(dp^{3}dp^{2}dp^{1})'}{(p^{0})'}=
=\frac{(1-(p^{0})^{-1}p^{3}v)dp^{3}dp^{2}dp^{1}}{p^{0}-p^{3}v}=\frac{d^{3}\vec{p}}{p^{0}}.
		\end{align*}
Combining them with  \eqref{eq:deff1} gives that each component of $\vec{f}_{M}$ is Lorentz invariant, such that
 the last $(\acute{N}_{M}-3)$ components of $\vec{\acute{W}}_{M}$ are also Lorentz invariant.

From \eqref{eq:variable0},
it is not difficult to prove that $n$ and $\theta$ are Lorentz invariant.

Moreover, one has
\[
\left(duU^{0}n_{3}^{3}\right)'=\frac{d(U^{3})'}{(U^{0})'}=\frac{dU^{3}-dU^{0}v}{U^{0}-U^{3}v}
=\frac{dU^{3}-(U^{0})^{-1}U^{3}dU^{3}v}{U^{0}-U^{3}v}=\frac{dU^{3}}{U^{0}}=duU^{0}n_{3}^{3}.
\]
Using the above results completes the proof of the first part.

(ii)  Because   $\vec{\acute{A}}_{M}^{0}$ and  $\vec{\acute{A}}_{M}^{3}$  only depend on $\theta$,
they are Lorentz invariant.
The source term $\vec{\acute{S}}(\vec{\acute{W}}_{M})$ in \eqref{eq:colAWfinal}	can be rewritten into
	\[
	\vec{\acute{S}}(\vec{\acute{W}}_{M})=-\frac{1}{\tau}\vec{\acute{P}}_{M}^{p}\vec{\acute{A}}_{M}^{0}(\vec{\acute{P}}_{M}^{p})^{T}\left(\vec{f}_{M}-\vec{f}^{(0)}_{M}\right),
	\]
	which has been expressed  in terms of 	the Lorentz covariant quantities.
	In fact,  the general source term $\vec{\acute{S}}(\vec{\acute{W}}_{M})$ in
	the  moment system   \eqref{eq:moment1} is also Lorentz invariant.
	The proof is completed.
	\qed\end{proof}

\begin{thm}[Lorentz covariance]
	\label{thm:lorentz}
	The moment system  \eqref{eq:moment1D} with the source term \eqref{eq:colAWfinal}
	 is Lorentz covariant.
\end{thm}

\begin{proof}
	
	From  the 3rd step in Section \ref{subsec:deduction} and Lemma  \ref{lem:lorentz},  one  knows that
	$\hat{\vec{\acute{D}}}_{M}=\vec{\acute{D}}_M(\vec{\acute{D}}_{M}^{u})^{-1}$ can be expressed in terms of the Lorentz covariant quantities, so it is Lorentz invariant.
	Because
	\begin{align*}
	(\vec{\acute{M}}_{M}^{0})' &=\gamma(v)(U^{3}-U^{0}v)\vec{\acute{P}}_{M}^{p}\vec{\acute{A}}_{M}^{3}(\vec{\acute{P}}_{M}^{p})^{T}+\gamma(v)(U^{0}-U^{3}v)\vec{\acute{P}}_{M}^{p}\vec{\acute{A}}_{M}^{0}(\vec{\acute{P}}_{M}^{p})^{T},\\
	 (\vec{\acute{M}}_{M}^{3})'&=\gamma(v)(U^{0}-U^{3}v)\vec{\acute{P}}_{M}^{p}\vec{\acute{A}}_{M}^{3}(\vec{\acute{P}}_{M}^{p})^{T}+\gamma(v)(U^{3}-U^{0}v)\vec{\acute{P}}_{M}^{p}\vec{\acute{A}}_{M}^{0}(\vec{\acute{P}}_{M}^{p})^{T},
	\end{align*}
	and
	\begin{align*}
	\left(\frac{\partial }{\partial t}\right)'=\gamma(v)\left(\frac{\partial }{\partial t}+v\frac{\partial }{\partial x}\right),\quad
	\left(\frac{\partial }{\partial x}\right)'=\gamma(v)\left(\frac{\partial }{\partial x}+v\frac{\partial }{\partial t}\right),
	\end{align*}
	%
	%
	%
	one has
	\begin{align*}
	\left(\vec{\acute{D}}_{M}^{u}\frac{\partial \vec{\acute{W}}_{M}}{\partial t}\right)'=&{\rm diag}\left\{1,\big(U^{0}n_{3}^{3}\big)',1,\cdots,1\right\}\gamma(v)\left(\frac{\partial (n,(U^{3})',\theta,\Pi,f_{0,0}^{(1)},f_{0,0}^{(2)},f_{3,0}^{(0)},\cdots,f_{0,0}^{(M)})^{T}}{\partial t}\right.\\
&	+\left.v\frac{\partial (n,(U^{3})',\theta,\Pi,f_{0,0}^{(1)},f_{0,0}^{(2)},f_{3,0}^{(0)},\cdots,f_{0,0}^{(M)})^{T}}{\partial x}\right)\\
	=&{\rm diag}\left\{1,\big((U^{0}\big)^{-1})',1,\cdots,1\right\}\gamma(v)\left(\frac{\partial (n,(U^{3})',\theta,\Pi,f_{0,0}^{(1)},f_{0,0}^{(2)},f_{3,0}^{(0)},\cdots,f_{0,0}^{(M)})^{T}}{\partial t}\right.\\
&	+\left.v\frac{\partial (n,(U^{3})',\theta,\Pi,f_{0,0}^{(1)},f_{0,0}^{(2)},f_{3,0}^{(0)},\cdots,f_{0,0}^{(M)})^{T}}{\partial x}\right)\\
	=&\vec{\acute{D}}_{M}^{u}\gamma(v)\left(\frac{\partial \vec{\acute{W}}_{M}}{\partial t}+v\frac{\partial \vec{\acute{W}}_{M}}{\partial x}\right),
	\end{align*}
	where  the last equal sign is derived by  following the proof of Lemma \ref{lem:lorentz}.
	Similarly, one has
	\begin{equation*}
	\left(\vec{\acute{D}}_{M}^{u}\frac{\partial \vec{\acute{W}}_{M}}{\partial x}\right)'=\vec{\acute{D}}_{M}^{u}\gamma(v)\left(\frac{\partial \vec{\acute{W}}_{M}}{\partial x}+v\frac{\partial \vec{\acute{W}}_{M}}{\partial t}\right).
	\end{equation*}
	Thus one yields
	\begin{align*}
	&\left(\vec{\acute{B}}_{M}^{0}\frac{\partial \vec{\acute{W}}_{M}}{\partial t} +  \vec{\acute{B}}_{M}^{3}\frac{\partial \vec{\acute{W}}_{M}}{\partial x}\right)'\\
	=&(\vec{\acute{M}}_{M}^{0})'\left(\vec{\acute{D}}_{M}\frac{\partial \vec{\acute{W}}_{M}}{\partial t}\right)'+(\vec{\acute{M}}_{M}^{3})'\left(\vec{\acute{D}}_{M}\frac{\partial \vec{\acute{W}}_{M}}{\partial x}\right)'\\
	=&\left((U^{3})'\vec{\acute{P}}_{M}^{p}\vec{\acute{A}}_{M}^{3}(\vec{\acute{P}}_{M}^{p})^{T}+ (U^{0})'\acute{P}_{M}^{p}\vec{\acute{A}}_{M}^{0}(\acute{P}_{M}^{p})^{T}\right)\vec{\acute{D}}_{M}\left(\gamma(v)\left(\frac{\partial \vec{\acute{W}}_{M}}{\partial t}+v\frac{\partial \vec{\acute{W}}_{M}}{\partial x}\right)\right)\\
	&+\left((U^{0})'\vec{\acute{P}}_{M}^{p}\vec{\acute{A}}_{M}^{3}(\vec{\acute{P}}_{M}^{p})^{T}+ (U^{3})'\vec{\acute{P}}_{M}^{p}\vec{\acute{A}}_{M}^{0}(\vec{\acute{P}}_{M}^{p})^{T}\right)\vec{\acute{D}}_{M}\left(\gamma(v)\left(\frac{\partial \vec{\acute{W}}_{M}}{\partial x}+v\frac{\partial \vec{\acute{W}}_{M}}{\partial t}\right)\right)\\
	=&\left(U^{3}v\vec{\acute{P}}_{M}^{p}\vec{\acute{A}}_{M}^{3}(\vec{\acute{P}}_{M}^{p})^{T}+U^{0}\vec{\acute{P}}_{M}^{p}\vec{\acute{A}}_{M}^{0}(\vec{\acute{P}}_{M}^{p})^{T}\right)
	\vec{\acute{D}}_{M}\frac{\partial \vec{\acute{W}}_{M}}{\partial t}\\
	&+\left(U^{0}\vec{\acute{P}}_{M}^{p}\vec{\acute{A}}_{M}^{3}(\vec{\acute{P}}_{M}^{p})^{T}+U^{3}\vec{\acute{P}}_{M}^{p}\vec{\acute{A}}_{M}^{0}(\vec{\acute{P}}_{M}^{p})^{T}\right)
	\vec{\acute{D}}_{M}\frac{\partial \vec{\acute{W}}_{M}}{\partial x}\\
	=&\vec{\acute{B}}_{M}^{0}\frac{\partial \vec{\acute{W}}_{M}}{\partial t} +  \vec{\acute{B}}_{M}^{3}\frac{\partial \vec{\acute{W}}_{M}}{\partial x}.
	\end{align*}
	Combining it with Lemma \ref{lem:lorentz}  completes the proof.
	\qed\end{proof}

\section{Conclusions}
\label{sec:conclud}
The paper derived the  arbitrary order globally hyperbolic moment system of
the three-dimensional special relativistic Boltzmann equation for the first time
 and studied the hyperbolicity and linear stability of the moment system,
  and Lorentz covariance of the quasi-1D moment system.
The technique was the model reduction  by the operator projection \cite{MR:2014} and \cite{Kuang:2017}.
 The key contributions were the real spherical harmonics
 and  families of the Grad type orthogonal polynomials
were used to establish the bases of the weighted polynomial spaces and
 the careful study on
 their recurrence relations and derivatives as well as the zeros of the Grad type orthogonal polynomials were given.
%
%
It is interesting to  develop  robust, high order accurate numerical schemes for the moment system
and   find other basis for the derivation of  moment system with some good property, e.g. non-negativity.

\section*{Acknowledgements}
This work was partially supported by
{ the Special Project on High-performance Computing under the National Key R\&D Program (No. 2016YFB0200603),
Science Challenge Project (No. JCKY2016212A502), and}
the National Natural Science Foundation
of China (Nos.  91330205, 91630310, \& 11421101).

\end{document}